\documentclass{amsart}
\usepackage{dsfont}
\usepackage[utf8]{inputenc}
\usepackage[T1]{fontenc}
\usepackage{ tipa }
\usepackage{helvet}
\usepackage{color}
\usepackage{mathrsfs}  
\usepackage{stmaryrd}  
\usepackage{amsmath,amsxtra,amsthm,amssymb,xr,enumerate,fullpage}
\usepackage[all]{xy}
\usepackage[bbgreekl]{mathbbol}
\usepackage{bbm}
\DeclareMathSymbol{\invques}{\mathord}{operators}{`>}
\DeclareUnicodeCharacter{00BF}{\tmquestiondown}
\DeclareRobustCommand{\tmquestiondown}{%
  \ifmmode\invques\else\textquestiondown\fi
}
\usepackage{verbatim}

\newtheorem{theorem}{Theorem}[section]
\newtheorem{lemma}[theorem]{Lemma}
\newtheorem{conj}[theorem]{Conjecture}
\newtheorem{proposition}[theorem]{Proposition}
\newtheorem{corollary}[theorem]{Corollary}
\newtheorem{defn}[theorem]{Definition}

\newtheorem{remark}[theorem]{Remark}

\newcommand{\pr}{\mathrm{pr}}

\newcommand{\Gal}{\operatorname{Gal}}
\newcommand{\Fil}{\operatorname{Fil}}

\newcommand{\NN}{\mathbb{N}}
\newcommand{\QQ}{\mathbb{Q}}
\newcommand{\Qp}{\mathbb{Q}_p}
\newcommand{\Zp}{\mathbb{Z}_p}
\newcommand{\ZZ}{\mathbb{Z}}
\renewcommand{\AA}{\mathbb{A}}

\newcommand{\FFF}{\mathcal{F}}

\newcommand{\g}{\mathbf{g}}

\newcommand{\ord}{\mathrm{ord}}

\newcommand{\PP}{\mathfrak{P}}
\newcommand{\fp}{\mathfrak{p}}
\newcommand{\fq}{\mathfrak{q}}

\newcommand{\vp}{\varphi}
\newcommand{\cL}{\mathcal{L}}
\newcommand{\cH}{\mathcal{H}}
\newcommand{\cO}{\mathcal{O}}
\newcommand{\Iw}{\mathrm{Iw}}
\newcommand{\HIw}{H^1_{\mathrm{Iw}}}

\newcommand{\GL}{\mathrm{GL}}

\newcommand{\AQp}{\AA_{\Qp}^+}
\newcommand{\col}{\mathrm{Col}}
\newcommand{\image}{\mathrm{im}}
\newcommand{\cyc}{\textup{cyc}}
\newcommand{\ucol}{\underline{\col}}

\newcommand{\ff}{\mathfrak{f}}
\newcommand{\fL}{\mathfrak{L}}
\newcommand{\fm}{\mathfrak{M}}

\newcommand{\Hom}{\mathrm{Hom}}

\newcommand{\Char}{\mathrm{char}}
\newcommand{\Ind}{\mathrm{Ind}}
\newcommand{\ac}{\textup{ac}}
\newcommand{\LL}{\Lambda}
\newcommand{\TT}{\mathbb{T}}

\newcommand{\RR}{\mathcal{R}}

\newcommand{\fn}{\mathfrak{n}}

\newcommand{\lra}{\longrightarrow}
\newcommand{\ra}{\rightarrow}
\newcommand{\res}{\textup{res}}
\newcommand{\TSym}{\textup{TSym}}
\newcommand{\bc}{\bar{c}}

\newcommand{\BLtwo}{$\mathbf{(L2)}$}

\newcommand{\Bg}{\mathbf{g}}

\newcommand{\BF}{\textup{BF}}

\newcommand{\cF}{\mathcal{F}}
\newcommand{\Dcris}{\mathbb{D}_{\rm cris}}
\newcommand{\Mlog}{M_{\log}}
\newcommand{\fM}{\m}
\newcommand{\adj}{\mathrm{adj}}
\newcommand{\Tw}{\mathrm{Tw}}
\newcommand{\wBF}{\widehat{\BF}}

\newcommand{\mom}{\mathrm{mom}}
\newcommand{\id}{\mathrm{id}}
\newcommand{\an}{\textup{an}}
\newcommand{\han}{H_\textup{an}}
\newcommand{\LLac}{\Lambda_{\cO_L}(\Gamma^\ac)}
\newcommand{\calHac}{\mathcal{H}_L(\Gamma^\ac)}
\newcommand{\calHcyc}{\mathcal{H}(\Gamma^\cyc)}
\newcommand{\calHacn}{\mathcal{H}_n(\Gamma^\ac)}
\newcommand{\Qpinf}{\QQ_{p,\infty}}

\newcommand{\p}{\mathfrak{p}}
\newcommand{\q}{\mathfrak{q}}
\newcommand{\vt}{\vartheta}
\newcommand{\gm}{{\g_\m}}
\newcommand{\m}{\mathfrak{m}}
\newcommand{\n}{\mathfrak{n}}

\usepackage{hyperref}

\begin{document}

\title[Non-ordinary Iwasawa Theory over Imaginary Quadratic Fields]{Iwasawa theory of elliptic modular forms over imaginary quadratic fields at non-ordinary primes}
\begin{abstract}
This article is a continuation of our previous work~\cite{buyukbodukleianticycloord} on the Iwasawa theory of an elliptic modular form over an imaginary quadratic field $K$, where the modular form in question was assumed to be ordinary at a fixed odd prime $p$. We formulate integral Iwasawa main conjectures at non-ordinary primes $p$ for suitable twists of the base change of a newform $f$ to an imaginary quadratic field $K$ where $p$ splits, over the cyclotomic $\Zp$-extension, the anticyclotomic $\ZZ_p$-extensions (in both the \emph{definite} and the \emph{indefinite} cases) as well as the  $\ZZ_p^2$-extension of $K$.  In order to do so, we define Kobayashi-Sprung-style signed Coleman maps, which we use to introduce doubly-signed Selmer groups. In the same spirit, we construct signed (integral) Beilinson-Flach elements (out of the collection of unbounded Beilinson-Flach elements of Loeffler-Zerbes), which we use to define doubly-signed $p$-adic $L$-functions. The main conjecture then relates these two sets of objects. Furthermore, we show that the integral Beilinson-Flach elements form a locally restricted Euler system, which in turn allows us to deduce (under certain technical assumptions) one inclusion in each one of the four main conjectures we formulate here (which may be turned into equalities in favourable circumstances). 
\end{abstract}

\author{K\^az\i m B\"uy\"ukboduk}
\address{K\^az\i m B\"uy\"ukboduk\newline UCD School of Mathematics and Statistics\\ University College Dublin\\ Ireland}
\email{kazim.buyukboduk@ucd.ie}

\author{Antonio Lei}
\address{Antonio Lei\newline
D\'epartement de Math\'ematiques et de Statistique\\
Universit\'e Laval, Pavillion Alexandre-Vachon\\
1045 Avenue de la M\'edecine\\
Qu\'ebec, QC\\
Canada G1V 0A6}
\email{antonio.lei@mat.ulaval.ca}

\thanks{The first author is partially supported by the EU Grant CriticalGZ. The second author is supported by the NSERC Discovery Grants Program 05710.}

\maketitle

\section{Introduction}
Fix forever a prime $p\geq 5$ and an imaginary quadratic field $K$ where $(p)=\fp\fp^c$ splits. The superscript $c$ will always stand for the action of a fixed complex conjugation. We fix a modulus $\mathfrak{f}$ coprime to $p$ with the property that the ray class number of $K$ modulo $\mathfrak{f}$ is not divisible by $p$. We also fix once and for all an embedding $\iota_p: \overline{\QQ}\hookrightarrow \mathbb{C}_p$ and suppose that the prime $\fp$ of $K$ lands inside the maximal ideal of $\cO_{\mathbb{C}_p}$.  Fix also  a ray class character $\chi$ modulo $\ff p^\infty$ with {$\chi(\p)\neq \chi(\p^c)$. Ray class characters satisfying the latter property shall be called \emph{$p$-distinguished}. }

Let $K_\infty$ denote the  $\ZZ_p^2$-extension of $K$ with $\Gamma:=\Gal(K_\infty/K)\cong \ZZ_p^2$. We let $K_\cyc/K$ and $K^\ac/K$ denote the cyclotomic and the anticyclotomic $\Zp$-extensions of $K$  contained in $K_\infty$ respectively. We write $\Gamma^\cyc:=\Gal(K_\cyc/K)$ and $\Gamma^{\ac}:=\Gal(K^\ac/K)$.  
For a finite flat extension $\cO$ of $\ZZ_p$ and for $?=\emptyset,\ac, \cyc$, we define the Iwasawa algebra  
$\LL_{\cO}(\Gamma^?):=\cO[[\Gamma^?]]$ with coefficients in $\cO$. If $F$ is the field of fractions of $\cO$, we write $\LL_F(\Gamma^?)$ for $\LL_{\cO}(\Gamma^?)\otimes F$.  Let $\iota:\Gamma^? \ra \Gamma^?$ denote the involution given by $\gamma\mapsto \gamma^{-1}$. We write $\LL_{\cO_L}(\Gamma)^\iota$ for the free $\LL_{\cO_L}(\Gamma)$-module of rank one  equipped with the $G_K$-action given via 
$$G_K\twoheadrightarrow \Gamma\stackrel{\iota}{\lra}\Gamma\hookrightarrow \LL_{\cO_L}(\Gamma)^\times\,.$$
Let $f \in S_k(\Gamma_0(N_f),\varepsilon_f)$ be a normalized cuspidal eigen-newform of level $N_f$, even weight $k\ge2$ and nebentypus $\varepsilon_f$. We assume that $p\nmid N_f$. In \cite{buyukbodukleianticycloord}, we have studied the Iwasawa theory of the base change $f_{/K}$ of $f$ to $K$ over $K^\ac$ using the Beilinson-Flach elements as constructed in \cite{KLZ1,KLZ2}, under the assumption that $a_p(f)$ is a $p$-adic unit (the ordinary case). Our goal in this article is to eliminate this  condition. That is, we  assume  that  $a_p(f)$ is \textit{not} a $p$-adic unit. Kobayashi has introduced the signed Selmer groups over the $\Zp$-cyclotomic extension of $\QQ$ when $k=2$ and $a_p(f)=0$. This has subsequently been generalized to the $a_p(f)\ne0$ case by Sprung \cite{sprungfactorizationinitial} and to the higher weight case in \cite{lei11compositio,LLZ0,LLZ0.5}. For the cyclotomic $\Zp$-extension of $\QQ$, the Beilinson-Kato Euler system  constructed in \cite{kato04} is utilized to show that the signed Selmer groups are cotorsion over the cyclotomic Iwasawa algebra. In the anticyclotomic setting, Kobayashi's signed Selmer groups can be readily defined under the assumption that $a_p(f)=0$ and $k=2$. The structure of these groups have been studied extensively in
\cite{darmoniovita} (the definite case) and \cite{castellawan1607,longovigni} (the indefinite case) with the aid of the Heegner-point Euler system (used in place of the Beilinson-Kato Euler system). 
In this article, we treat the cyclotomic, anticyclotomic (both definite and indefinite cases) and two-variable main conjectures simultaneously, for higher weight $k\ge2$ and general $a_p(f)$: One of our main results is one divisibility of the Iwasawa main conjecture in all these situations (c.f., Theorems~\ref{thm:leopoldtconjecturesforTfchiintro1}, \ref{thm:leopoldtconjecturesforTfchiintro2} and \ref{thm:mainconjindefiinteintro} below). The natural replacement for the special elements employed in the previous settings  mentioned above (i.e., Beilinson-Kato classes and Heegner points)  are the Beilinson-Flach elements  studied in \cite{LZ1}. 

Let $\alpha$ and $\beta$ be the two roots of the Hecke polynomial $X^2-a_p(f)X+\varepsilon_f(p)p^{k-1}$. We assume that $\alpha\neq \beta$ and let $f^\alpha$ and $f^\beta$ denote the two $p$-stabilizations of $f$.
We also assume throughout that $p>k$. We fix a finite extension $L$ of $\Qp$ that contains the values of our fixed ray class character $\chi$, the Hecke field of $f$ as well as $\alpha$ and $\beta$. {Let $L_0$ be the subextension of $L/\Qp$ generated by  the values of $\chi$ and the Hecke field of $f$.} Let $\varpi$ be a fixed uniformizer of $\cO_L$. We write $W_f$ for Deligne's $2$-dimensional  $L$-representation, whose Hodge-Tate weights are $0$ and $1-k$ (with the normalization that the Hodge-Tate weight of the cyclotomic character is $+1$). We fix a Galois-stable $\cO_L$-lattice $R_f$ inside $W_f$. Let $W_f^*$ and $R_f^*$ denote  the dual representations $\Hom(W_f,L)$ and $\Hom(R_f,\cO_L)$, respectively. We shall regard the ray class character $\chi$ also as a character of $G_K$ via the Artin map of class field theory, which we normalize by sending uniformizers to geometric Frobenius elements. For technical reasons, we will be working with the $p$-distinguished twist $R_{f,\chi^{-1}}^*:=R_f^*\otimes \chi^{-1}$ of $R_f^*$ rather than  $R_f^*$ itself. We note that $R_f^*$ is a representation of $G_{\QQ}$, whereas $R_{f,\chi^{-1}}^*$ is only a representation of $G_K$.
Let $T_{f,\chi}:=R_{f,\chi^{-1}}^*(1-k/2)=R_{\overline{f}}(k/2)\otimes\chi^{-1}$ (where $\overline{f}=f\otimes\varepsilon_f^{-1}$ is the dual form) denote the central critical twist  of Deligne's representation and let $\TT_{f,\chi}$ denote $T_{f,\chi}\otimes\LL_{\cO_L}(\Gamma)^\iota$. Here,  $\TT_{f,\chi}$  is equipped with the diagonal action.
We similarly define $\TT_{f,\chi}^\cyc$ and $\TT_{f,\chi}^\ac$ on replacing $\Gamma$ by $\Gamma^\cyc$ and $\Gamma^\ac$ respectively. 
Let ${\rho}_f: G_\QQ\ra \GL_2(\cO_L)$ denote the representation afforded by the $\cO_L[[G_\QQ]]$-module $R_{f}^*$. Since we assumed $p>k$ and $f$ is non-ordinary at $p$, it follows that the residual representation $\overline{\rho}_f$ (as well as its twists by characters) is absolutely irreducible even when restricted to $G_{\QQ_p}$ according to a result due to Fontaine and Edixhoven.

For $\lambda\in\{\alpha,\beta\}$ and $\q\in\{\p,\p^c\}$, we may project the two-variable Perrin-Riou big logarithm map attached to $W_f^*$ constructed by Loeffler-Zerbes \cite{LZ0,LZ1} to some $\lambda^{-1}$-eigenspace of the Dieudonn\'e module of $W_f^*$. After twisting, this results in a $\Lambda_{\cO_L}(\Gamma)$-morphism
\[
\cL_{\lambda,\fq}:H^1(K_{\fq},\TT_{f,\chi})\rightarrow \cH_L(\Gamma),
\]
where  $\cH_L(\Gamma)$ is the distribution algebra on $\Gamma$ with coefficients in $L$.  Our first result is the following factorization theorem of these maps (to an explicit linear combination of two signed\footnote{This generalizes the work of Kobayashi in the context of cyclotomic Iwasawa theory of elliptic curves. Although we adopt the notation of Sprung to denote these objects by $\#$ and $\flat$ (and not by $\pm$), we still refer to them as \emph{signed} objects, in order to remind the reader that these constructions are in fact in the spirit of Kobayashi's work.} Coleman maps with bounded image), in the spirit of \cite{kobayashi03, LLZ0, LLZ0.5, lei11compositio, sprungfactorizationinitial, sprung16}. Along the way, we provide a precise description of the  the images of these maps, which might be of independent interest.
\begin{theorem}\label{prop:Hidadecompintro}
For each $\mathfrak{q}=\fp,\fp^c$, there exists a pair of  $\Lambda_{\cO_L}(\Gamma)$-morphisms
$$\col_{\#,\q}\,,\,\col_{\flat,\q}\,: H^1\left(K_{\fq},\TT_{f,\chi}\right)\lra \Lambda_{\cO_L}(\Gamma)$$
 such that
\[
\begin{pmatrix}
\cL_{\alpha,\fq}\\ \cL_{\beta,\fq}
\end{pmatrix}=
M\cdot 
\begin{pmatrix}
\col_{\#,\fq}\\
\col_{\flat,\fq}
\end{pmatrix},
\]
where $M$ is some logarithmic matrix that we shall introduce in \S\ref{S:wach}.
\end{theorem}

Let $\mathcal{N}$ denote the collection of square-free integral ideals of $\cO_K$ which are prime to $\ff p$. Given a prime $\mathfrak{q}$ of $K$, we let (as in \cite[\S II.1]{rubin00}) $K(\mathfrak{q})$ denote the maximal $p$-extension of $K$ contained in the ray class field modulo $\mathfrak{q}$. For  $\mathfrak{c}=\mathfrak{q}_1\cdots\mathfrak{q}_s\in\mathcal{N}$, we set $K(\mathfrak{c}):=K(\mathfrak{q}_1)\cdots K(\mathfrak{q}_s)$. In particular, our assumption on  the ray class number modulo $\mathfrak{f}$ implies that $K(1)=K$. Consequently, the extensions $K(\mathfrak{q}_i)$  are linearly disjoint over $K$. For $\mathfrak{c}$ as above, set $\Delta(\mathfrak{c})=\Gal(K(\mathfrak{c})/K)=\prod_{i=1}^s\Gal(K(\mathfrak{q}_i)/K)$.

In \cite{kobayashi03}, Kobayashi applied his  signed  Coleman maps to the Beilinson-Kato classes along the cyclotomic tower to recover Pollack's signed $p$-adic $L$-functions defined in \cite{pollack03}. We shall apply our Coleman maps from Theorem~\ref{prop:Hidadecompintro} to a collection of Beilinson-Flach elements 
$$
\BF^\alpha_{\fn},\BF^\beta_{\fn}\in H^1\left(K(\mathfrak{n}),T_{f,\chi}\otimes\cH_{L}(\Gamma)^\iota\right),\ \fn\in\mathcal{N},
$$
 whose construction is derived from \cite{KLZ1,KLZ2,LZ1}.  The classes constructed in \textit{op. cit.} are a priori associated to the Rankin-Selberg convolution of two modular forms over cyclotomic extensions of $\QQ$. We shall explain in \S\ref{S:ESimag} how we may construct the desired cohomology classes over a sufficiently large collection of ray class extensions of $K$. The fact that these cohomology classes satisfy the Euler system distribution relation will follow from a slight extension of \cite[Proposition 3.8]{buyukbodukleianticycloord} (which is a generalization of \cite[Theorem 3.5.1]{LLZ2} to higher weight modular forms via \cite{KLZ2}). 
 
 The Beilinson-Flach elements thus constructed will not be integral under our assumption that $f$ is non-ordinary at $p$. It turns out that we may factorize these elements in the same way we factorize the Perrin-Riou maps \emph{uniformly} in all tame levels $\mathfrak{n}$ (which is also crucial for our purposes). This in turn equips us with an Euler system with which we may run the Euler system machine and at which we may evaluate our signed Coleman maps. 

\begin{theorem}\label{thm:decompintro}
 For every $\mathfrak{n}\in \mathcal{N}$, there exists a pair of signed Beilinson-Flach elements
$$\BF^{\#}_{\mathfrak{n}}\,,\, \BF^{\flat}_{\mathfrak{n}}\in H^1\left(K(\mathfrak{n}),\TT_{f,\chi}\otimes L\right)$$ 
 verifying the following Euler system distribution relation: For every $\frak{n}\in \mathcal{N}$, prime $\frak{l}$ of $\cO_K$ coprime to $\frak{nf}p$ and $\bullet\in \{\#,\flat\}$ we have
 \begin{equation}
 \label{EQN_ESrelation}
 {\rm cor}^{K(\mathfrak{nl})}_{K(\mathfrak{n})}\left(\BF^{\bullet}_{\mathfrak{nl}}\right)=P_{\frak{l}}(\textup{Fr}_\frak{l})\,\BF^{\bullet}_{\mathfrak{n}}
 \end{equation}
where $P_{\frak{l}}(X):=\det(1-{\rm Fr}_{\frak{l}}X\mid R_f(k/2)\otimes\chi)$ and $\textup{Fr}_\frak{l}$ is the geometric Frobenius at $\frak{l}$. Moreover, we have 
\[
\begin{pmatrix}
\BF^{\alpha}_{\mathfrak{n}}\\ \BF^{\beta}_{\mathfrak{n}}
\end{pmatrix}= M\cdot 
\begin{pmatrix}
\BF^{\#}_{\mathfrak{n}}\\ \BF^{\flat}_{\mathfrak{n}}
\end{pmatrix}.
\]
and there exists a constant $C\in L$, that is independent of $\mathfrak{n}$, such that
\[
C\cdot\BF_{\mathfrak{n}}^{\bullet}\in H^1\left(K(\mathfrak{n}),\TT_{f,\chi}\right)
\]
for both $\bullet =\#$ and $\flat$.
\end{theorem}
This statement is proved in Section~\ref{sec:BFelementsandfactorization} below. We note that special cases of this construction when $\fn=1$ were given by Wan~\cite{wansupersingularellipticcurves} and Sprung \cite{sprung16} in the case of elliptic curves, 
though our proof is independent of theirs. We remark also that the slight difference between the Euler system distribution relation in Theorem~\ref{thm:decompintro} and the one given in \cite[Theorem 3.5.1]{LLZ2} comes from the fact that we are working with a different twist of the motive associated to $f_{/K}\otimes\chi$ than the one considered in op.cit.

Fix $C$ as in Theorem~\ref{thm:decompintro} and with minimal $p$-adic valuation. We shall drop $\mathfrak{n}$  from the notation whenever it equals $1$. We set 
$$c^\bullet:=C\cdot\BF^{\bullet}\in H^1\left(K,\TT_{f,\chi}\right)$$ 
for $\bullet \in\{\#,\flat\}$ and likewise define the projections of these elements $c^\bullet_?\in H^1\left(K,\TT_{f,\chi}^?\right)$ (for $?=\ac,\cyc$) to the cyclotomic and anticyclotomic deformations.
For each of the four choices of the pair $\bullet,\star \in \{\#,\flat\}$, we define the two-variable \emph{doubly-signed Beilinson-Flach $p$-adic $L$-function} by setting
$$\mathfrak{L}_{\star,\bullet}:=\col_{\star,\fp}(c^\bullet) \in \LL_{\cO_L}(\Gamma).$$
We likewise define the cyclotomic (respectively, anticyclotomic) doubly-signed $p$-adic $L$-function $\mathfrak{L}_{\star,\bullet}^\cyc \in \LL_{\cO_L}(\Gamma^\cyc)$ (respectively, $\mathfrak{L}_{\star,\bullet}^\ac  \in \LL_{\cO_L}(\Gamma^\ac)$). Using the signed Coleman maps, we may also define the doubly-signed (integral) Selmer groups $\mathfrak{X}_{\star,\bullet}^?$ (see Definition~\ref{def:doublysignedselmergroups} below) for $?=\ac,\cyc,\emptyset$.
The following is the first step towards the proof of a doubly-signed Iwasawa main conjecture, both over $K_\cyc/K$ and over $K_\infty/K$. 
\begin{theorem}
\label{thm:leopoldtconjecturesforTfchiintro1}
Suppose that $\rho_f(G_K)$ contains a conjugate of $\textup{SL}_2(\ZZ_p)$. Suppose for $?=\cyc$ or $\emptyset$ that we have $\mathfrak{L}_{\star,\bullet}^?\ne 0$.
Then, the $\LL_{\cO_L}(\Gamma^?)$-module $\mathfrak{X}^?_{\star,\bullet}$ is torsion and
$$\mathfrak{L}_{\star,\bullet}^?\subset \xi_{\star}^?\cdot\Char\left(\mathfrak{X}_{\star,\bullet}^{?}\right).$$
\end{theorem}
Here, the fudge factor $ \xi_{\star}^?$ is a product of certain linear polynomials and its detailed description is given at the end of \S\ref{S:semilocal}  below. The presence of this fudge factor is due to the failure of the surjectivity of signed Coleman maps. Theorem~\ref{thm:leopoldtconjecturesforTfchiintro1} corresponds to the assertions proved in Theorem~\ref{thm:leopoldtconjecturesforTfchi}(iii) and Theorem ~\ref{thm:halfofMainConjtwovarCycloDefinite}(ii) in the main body of our article. 
\begin{remark} 
\label{rem:thmleopoldtconjecturesforTfchiintro1}It follows from Corollary~\ref{cor:lonecorrect} below that $\mathfrak{L}_{\star,\bullet}^\cyc\neq 0$ for at least one choice of $\star,\bullet \in\{\#,\flat\}$. Moreover, the containment in Theorem~\ref{thm:leopoldtconjecturesforTfchiintro1} may be upgraded to an equality (up to a power of $p$) using \cite[Theorems 1.7 and 3.8]{wan16} under the assumptions of op.cit.  A particular case when $k=2$ is the subject of~\cite{castellawan1607}.
\end{remark}
{We now fix our sights to the anticyclotomic setting.} We have the following result towards the doubly-signed anticyclotomic main conjecture when the sign in the functional equation of the Hecke $L$-function $L(f/K,\chi,s)$ is $+1$.  The assertions in this theorem follow on combining Theorems~\ref{thm:leopoldtconjecturesforTfchi}(iv) and~\ref{thm:halfofMainConjtwovarCycloDefinite}(ii) below. 

\begin{theorem}
\label{thm:leopoldtconjecturesforTfchiintro2}
Assume that $\rho_f(G_K)$ contains a conjugate of $\textup{SL}_2(\ZZ_p)$. Suppose also that $\mathfrak{L}_{\star,\bullet}^\ac\ne 0$. 
Then the $\LL_{\cO_L}(\Gamma^\ac)$-module $\mathfrak{X}_{\star,\bullet}^{\ac}$ is torsion and 
$$\mathfrak{L}_{\star,\bullet}^\ac\subset \xi_{\star}^\ac\cdot\Char\left(\mathfrak{X}_{\star,\bullet}^{\ac}\right)\,.$$
\end{theorem}
\begin{remark}
\label{rem:thmleopoldtconjecturesforTfchiintro2}
The containment in Theorem~\ref{thm:leopoldtconjecturesforTfchiintro2} may also be upgraded to an equality (up to a power of $p$) using \cite[Theorems~1.7 and 3.8]{wan16}  under the assumptions of \textit{op.cit}.
\end{remark}
\begin{remark}
\label{rem:anticyclodefinitedetails} Under mild hypotheses (which imply that the sign of the functional equation for $L(f/K,\chi,s)$ is $``+"$), we prove in Corollary~\ref{cor:lthreeandhalfcorrect} below that $\mathfrak{L}_{\star,\bullet}^\ac\neq 0$ for at least one choice of $\star,\bullet \in\{\#,\flat\}$. 
\end{remark}
We next turn our attention to the remaining case, that is, the indefinite anticyclotomic Iwasawa theory of the form $f$. We have the following result in this set-up, which corresponds to Theorem~\ref{thm:leopoldtconjecturesforTfchi}(v) combined with Proposition~\ref{prop:Lalpha} below.
\begin{theorem}
\label{thm:leopoldtconjecturesforTfchiintro3}
Assume that $\rho_f(G_K)$ contains a conjugate of $\textup{SL}_2(\ZZ_p)$.  
Suppose that the sign of the functional equation for $L(f/K,\chi,s)$ is $``-"$, as well as that $\varepsilon_f=\mathbb{1}$, $a_p(f)=0$ and that the ramification index of {$L_0/\QQ_p$} is odd. Assume also that $c^\bullet_\ac\neq 0$. Then both modules $\mathfrak{X}_{\bullet,\bullet}^{\ac}$ and $H^1_{\FFF_{\bullet,\bullet}}(K,\TT_{f,\chi}^{\ac})$ have rank one over $\LL_{\cO_L}(\Gamma^\ac)$.
\end{theorem}
\begin{remark}
\label{rem:anticycloBFnonvanishingindefinite}
Under mild hypotheses $($that are somewhat stronger than our assumption that the functional equation for $L(f/K,\chi,s)$ is $``-"$), we prove in Proposition~\ref{prop:lfourcorrect} that $\res_{\fp^c}\left(c^\bullet_\ac\right)\neq 0$ for at least one choice of $\bullet \in\{\#,\flat\}$.
\end{remark}
\begin{remark}
\label{rem:anticycloBFvanishingindefinite}
Under the hypothesis of Theorem~\ref{thm:leopoldtconjecturesforTfchiintro3}, we prove in Proposition~\ref{prop:Lalpha} that $\mathfrak{L}_{\bullet,\bullet}^\ac=0$. Our proof rests on a functional equation for the $p$-adic $L$-function $\frak{L}_{\bullet,\bullet}$, which is the main result of  \cite{BF_Super_Addendum}.
\end{remark}

{\begin{remark} Suppose in this paragraph that $\varepsilon_f=\mathds{1}$. Then the requirement that the global root number for the Rankin-Selberg product $f_{/K}\otimes \chi$ be $\pm 1$ is equivalent to condition that $\chi\cdot \chi\circ c =\mathds{1}$. Note that a character $\chi$ with this property is $p$- distinguished if and only if $\chi(\p) \neq  \pm 1$. In other words, the only Rankin-Selberg convolution $f_{/K}\otimes \chi$ $($with $\varepsilon_f=\mathds{1}$$)$ whose anticyclotomic Iwasawa theory cannot be studied with the methods of the article is limited to the case when $\chi(\p)=\pm 1$. We shall treat these remaining cases (as well as some cases when the prime $p$ is inert) in a forthcoming article. \end{remark} }

Theorem~\ref{thm:leopoldtconjecturesforTfchiintro3} calls for a refinement along the lines of~\cite[Theorem 1.1(ii)]{buyukbodukleianticycloord}, where in an analogous set up (but assuming $f$ is $p$-ordinary) we proved a $\LL$-adic Gross-Zagier formula expressing the anticyclotomic restriction of the cyclotomic derivative of a two-variable Hida-$p$-adic $L$-function in terms of a $\LL_{\cO_L}(\Gamma^\ac)$-adic regulator and the torsion-submodule of an Iwasawa module. As we found no obvious way to define $\LL_{\cO_L}(\Gamma^\ac)$-adic height pairings on doubly-signed Selmer groups (contrary to the $p$-ordinary situation, where one deals with the Greenberg Selmer groups), we were led to consider this problem in the context of trianguline Selmer groups (as introduced by Pottharst \cite{jaycyclotmotives,jayanalyticfamilies}), which come equipped with a natural $p$-adic height pairing.
Let $\lambda,\mu\in\{\alpha,\beta\}$. We set $c^\mu:=C\cdot \BF^{\mu}_{1} \in H^1(K,\TT_{f,\chi}\widehat{\otimes} \mathcal{H}_L(\Gamma)) $  and define the \emph{analytic Beilinson-Flach $p$-adic $L$-function} 
$$\mathfrak{L}_{\lambda,\mu}:=\cL_{\lambda,\fp}(c^\mu) \in \mathcal{H}_L(\Gamma)$$
for $\lambda,\mu\in \{\alpha,\beta\}$. We likewise define the restrictions $c^\mu_\ac \in H^1(K,\TT_{f,\chi}^\ac \widehat{\otimes} \mathcal{H}_L(\Gamma^\ac))=H^1(K,\TT_{f,\chi}^\ac) \widehat{\otimes} \mathcal{H}_L(\Gamma^\ac)$ (the equality here is a consequence of \cite[Prop. 2.4.5]{LZ1}) and $\mathfrak{L}_{\lambda,\mu}^\ac \in \mathcal{H}_L(\Gamma^\ac)$ of $c^\mu$ and $\mathfrak{L}_{\lambda,\mu}$ to the anticyclotomic tower.

Let $\fm_\ac$ denote the maximal ideal of $\LL_{\cO_L}(\Gamma^\ac)$.  For each positive integer $n$, we let $\mathcal{H}_n(\Gamma^\ac)$ denote the $p$-adic completion of $\Lambda_{\cO_L}[\fm^n_\ac/p]$, the $\Lambda_{\cO_L}(\Gamma^\ac)$-subalgebra of $\Lambda_{\cO_L}(\Gamma^\ac)[1/p]$ generated by all $r/p$ with $r\in \fm^n_\ac$. Each $\calHacn[1/p]$ is a strict affinoid algebra (and in fact, also a Euclidean domain). They give an admissible affinoid covering of $\textup{Sp}\,\LL_\infty$ and $\calHac=\varprojlim \calHacn[1/p]$. See \cite[\S1.7]{dejong1995PIHES} for details. 
We fix a positive integer $n$ and $\lambda\in\{\alpha,\beta\}$. Let $A:=\calHacn[1/p]$. Following \cite{jaycyclotmotives, jayanalyticfamilies}, there exists a trianguline Selmer group $\mathfrak{X}_{\lambda,\lambda}^A$ over the affinoid algebra $A$. It arises as the degree-$2$ cohomology of a Selmer complex that is denoted by $\widetilde{\textbf{R}\Gamma}(G_{K,S},\Delta_{\lambda,\lambda},V_A)$ in the main text, where $V_A:=T_{f,\chi}\otimes_{\cO_L} A^\iota$.  We review this construction in Section~\ref{subsec:analyticSelmergroups}. Furthermore, as we explain in Theorem~\ref{thm:p-adicheight} below, there is an $A$-adic height pairing (interpolating $p$-adic height pairings along a $p$-adic family parametrized by $A$)
$$\langle\,,\,\rangle_{\lambda,\lambda}\,:\,H^1_{\lambda,\lambda}(K,V_A)\otimes H^1_{\lambda,\lambda}(K,\mathscr{D}(V_A)) \lra A$$
where $H^1_{\lambda,\lambda}(K,V_A)$ (respectively, $H^1_{\lambda,\lambda}(K,\mathscr{D}(V_A))$) denotes the cohomology of $\widetilde{\textbf{R}\Gamma}(G_{K,S},\Delta_{\lambda,\lambda},V_A)$ (respectively, that attached to the Kummer dual $\mathscr{D}(V_A)$ of $V_A$) in degree one. 
The following theorem is a partial result towards an $A$-adic Gross-Zagier formula. It corresponds to a combination of the results we prove as part of Corollary~\ref{cor:comparisonofBKselmerwithrelaxedBKintheindefinitecase} and Theorem~\ref{thm:padicGZ}.
\begin{theorem}
\label{thm:mainconjindefiinteintro}
Suppose that $\rho_f(G_K)$ contains a conjugate of $\textup{SL}_2(\ZZ_p)$. Suppose also that the number of primes dividing $N_f$ that does not split in $K/\QQ$ is even and the squares of these primes do not divide $N_f$. If we further assume that $\varepsilon_f=\mathbb{1}$, then $\mathfrak{X}_{\lambda,\lambda}$ has rank one over $A$ and 
$$\mathfrak{L}^\prime_{\lambda,\lambda,\ac}\subset\textup{Fitt}_{A}\left((\mathfrak{X}_{\lambda,\lambda}^A)_{\textup{tor}}\right) \mathscr{R}_p.$$ 
Here, $ \mathscr{R}_p$ is the $p$-adic regulator given in Definition~\ref{def:regulator} and $\mathfrak{L}^\prime_{\lambda,\lambda,\ac}$ is the restriction of the partial derivative of $\frak{L}_{\lambda,\lambda}$ with respect to the cyclotomic variable to the anticyclotomic line (as given in Definition~\ref{def:algebraicderivedpadicLfunction}).
\end{theorem}

\begin{remark}
\label{rem:equality}
As was the case for Theorems~\ref{thm:leopoldtconjecturesforTfchiintro1}, \ref{thm:leopoldtconjecturesforTfchiintro2} and \ref{thm:leopoldtconjecturesforTfchiintro3}, it should be possible to upgrade the containment in Theorem~\ref{thm:mainconjindefiinteintro} to an equality (up to a power $p$) using Wan's work on the non-ordinary Iwasawa theory of Rankin-Selberg products in \cite{wan16}. When $k=2$ and $a_p(f)=0$, a $\LL[1/p]$-adic variant of Theorem~\ref{thm:mainconjindefiinteintro} with $\chi=\mathds{1}$ $($which is enhanced via Wan's work to an equality$)$ has been announced in \cite{castellawan1607}. 
\end{remark}
{\begin{remark}
\label{rem:CCSS}
In a recent preprint \cite{CCSS}, Castella-\c{C}iperiani-Skinner-Sprung announced a proof that the inclusion in Theorem~\ref{thm:leopoldtconjecturesforTfchiintro1} is in fact an equality in the case $k=2$ and $\chi=\mathbb{1}$. 

\end{remark}}

\begin{remark}
{When dealing with Rankin-Selberg convolutions $f_{/K}\otimes\chi$ $($rather than the base change $f_{/K}$ itself$)$ with $p$-distinguished $\chi$, note that Theorem~\ref{thm:decompintro} above furnishes us with a full-fledged Euler system in this set up $($which does not only extend in the cyclotomic or the anticyclotomic tower, but in fact over the $\ZZ_p^2$-extension of $K)$. This is one of the main advantages that the use of Beilinson-Flach Euler systems present.}

{In particular, non-ordinary versions of the main results of Castella-Hsieh in \cite{castellahsieh} towards Bloch-Kato conjectures for Rankin-Selberg motives $f_{/K}\otimes\psi$ (where $\psi$ is a locally algebraic anticyclotomic Hecke character of $K$) can be deduced from the results  of the current article towards the anticyclotomic main conjectures, so long as the mod $p$ reduction $\overline{\psi}$ of $\psi$ verifies $\overline\psi(\p)\neq \pm 1$. $($We shall address the remaining few cases for the choice of $\overline{\psi}$ and some instances when $p$ is inert in a forthcoming article.$)$}
\end{remark}
\subsection*{Acknowledgments}We thank Laurent Berger, Francesc Castella, Henri Darmon, Ming-Lun Hsieh, David Loeffler and Tadashi Ochiai for interesting discussions during the preparation of this article. We are also grateful to the anonymous referee for valuable suggestions on an older version of this article, which helped us greatly to improve our paper. 
\subsection*{Notation and hypotheses}
Let $\Sigma$ denote the set of all places of $K$ that divide $pN_f\mathfrak{f}\infty$. For each prime $\nu$ of $K$, we fix a decomposition group at $\nu$ which we identify with $G_{\nu}:=\mathrm{Gal} (\overline K_\nu/K_\nu),$ and denote by $\mathscr{I}_{\nu}$ the inertia subgroup of 
$G_{\nu}$ and by $\textup{Fr}_\nu \in G_\nu/\mathscr{I}_\nu$ the geometric Frobenius. 

Given a complete local regular ring $R$ and a finitely generated torsion $R$-module $M$, we write $\Char_R(M)=\prod_{\mathfrak{P}}\mathfrak{P}^{\textup{length}_{\mathfrak{P}}(M_\mathfrak{P})}$ (where the product runs through height-one primes of $R$) to denote the characteristic ideal of $M$. We will also work with its analytic version, which we define and study in Section~\ref{subsec:refinedindefinitemainconj}. 
When we do not specify what $?=\{\ac,\cyc,\emptyset\}$ is and when there is no danger of confusion, we shall always write $\Char(X)$ in place of writing $\Char_{\LL_{\cO_L}(\Gamma^?)}(X)$ for a ${\LL_{\cO_L}(\Gamma^?)}$-module $X$.
If $F$ is a $p$-adic field and $F'$ is a $p$-adic Lie extension of $F$, then we write $H^i_{\Iw}(F',\bullet)$ for the projective limit $\varprojlim H^i(F'',\bullet)$, where $F''$ runs through all the finite subextension of $F'/F$ and the connecting maps are the corestriction maps. 

Let $\epsilon(f/K,\chi)$ be the global root number in the functional equation of the Rankin-Selberg $L$-function $L(f/K,\chi,s)$. We shall consider the following two conditions. 
\begin{equation}\tag{Sign $-$}
\varepsilon_f\chi\vert_{\mathbb{A}_{\QQ}^\times}=\mathbbm{1} \,\,\hbox{ and }\,\, \epsilon(f/K,\chi)=-1\,.
\end{equation}
\begin{equation}\tag{Sign $+$}
\varepsilon_f\chi\vert_{\mathbb{A}_{\QQ}^\times}=\mathbbm{1} \,\,\hbox{ and }\,\,\epsilon(f/K,s)=+1\,.
\end{equation}
 Throughout this article, we shall assume that the following four conditions hold.
\\\\
\textup{\textbf{(H.Reg.)}} $\alpha\neq \beta$.
\\\\
\textup{\textbf{(H.SS.)}} The order of the ray class group of $K$ modulo $\mathfrak{f}$ is prime to $p$.
\\\\
\textup{\textbf{(H.Dist.)}} {$\chi(\p) \neq\chi(\p^c)$}.
\\\\
The following hypothesis will also be in force from Section~\ref{sec:boundsonselmergroups} onwards:
\\\\
\textup{\textbf{(H.Im.)}} $\rho_f(G_K)$ contains a conjugate of $\textup{SL}_2(\ZZ_p)$.
\begin{remark}
\label{rem:hnaimpliesfreeness}
As a consequence of \textup{\textbf{(H.Im.)}}, it follows that the $\LL_{\cO_L}(\Gamma^?)$-module $H^1(K,\TT_{f,\chi}^?)$ is torsion-free. Also, the assumption that $p>k$  ensures that $H^2(K_{\mathfrak{q}},\TT_{f,\chi}^?)=0$ $($thanks to a theorem of Fontaine and Edixhoven~\cite{edixhoven92}, where they verified that the $p$-local residual representation in our set up is absolutely irreducible$)$ and the $\LL_{\cO_L}(\Gamma^?)$-module $H^1(K_{\mathfrak{q}},\TT_{f,\chi}^?)$ is free for $\mathfrak{q}=\fp,\fp^c$ $($c.f. \cite[Remark 2.8]{kbbESrankr}, which in turn relies on   \cite[Propositions 4.2.9 and 5.2.4]{nekovar06}$)$. 
\end{remark}
\begin{remark}
Let $\mathscr{H}_n$ denote the ring class field of $K$ corresponding to the order $\ZZ+p^n\mathfrak{o}_K$ (in particular, $\mathscr{H}_0=H_K$ is the Hilbert class field of $K$) and set $\mathscr{H}_\infty=\cup_n \mathscr{H}_n$. The anticyclotomic $\ZZ_p$-extension $K^\ac/K$ is the unique $\ZZ_p$-extension contained in $\mathscr{H}_\infty$. Since $p$ is coprime to the class number of $K$ by our assumptions, both $\p$ and $\p^c$ are totally ramified in $K^\ac/K$. 
For $\mathfrak{q}=\fp, \fp^c$, we denote the unique prime of $K^\ac$ over $\mathfrak{q}$ also by $\mathfrak{q}$. The extension $K^\ac_\mathfrak{q}/K_\mathfrak{q}$ $($being a ramified abelian extension of $\QQ_p$$)$ is a $\ZZ_p$-extension obtained from a Lubin-Tate formal group over $\ZZ_p$ of height one. We record these well known facts here since we will rely on them in a crucial way in Section~\ref{subsec:refinedindefinitemainconj} while studying the indefinite anticyclotomic Iwasawa theory over $K$.
\end{remark}

  \tableofcontents
\section{Logarithm matrix and modular forms}\label{S:coleman}
 
\subsection{Properties of $p$-adic power series} 
Recall from the introduction that $L/\Qp$ is a finite extension. Given any power series $F\in L[[X]]$ and $0<\rho<1$, we define the sup-norm $||F||_\rho=\sup_{|z|_p\le \rho}|f(z)|_p$. For any real number $r\ge 0$, we write  $\cH_{L,r}$ for the set of power series $F$ satisfying $\sup_{t} p^{-tr}||F||_{\rho_t}<\infty$,
where $\rho_t=p^{-1/p^{t-1}(p-1)}$. It is common to write $F=O(\log_p^r)$ when $F$ satisfies this condition. Similarly, we write  $F=o(\log_p^r)$ if ${\displaystyle \lim_{t\ra \infty} p^{-tr}||F||_{\rho_t}=0}$. We write $\cH_L$ for the union $\cup_{r\ge 0}\cH_{L,r}$. We say that a sequence of elements in $\cH_L$ converges when it does so under  the topology of pointwise convergence (the weak topology). Note that $\cH_{L,0}=\cO_L[[X]]\otimes_{\cO_L}L$.

Recall the $p$-adic logarithm 
$$\log_p(1+X)=X\prod_{n\ge 1}\frac{\Phi_n(1+X)}{p} \in \mathcal{H}_{\Qp,1},$$
where $\Phi_n$ denotes the $p^n$-th cyclotomic polynomial. If $n$ is a non-negative integer, we write $\omega_n(1+X)=(1+X)^{p^n}-1$.  Fix a topological generator $u\in 1+p\Zp$. For an integer $m\ge1$, we define
\begin{align*}
\Phi_{n,m}(1+X)&=\prod_{j=0}^{m-1}\Phi_{n}(u^{-j}(1+X)-1);\\
\omega_{n,m}(1+X)&=\prod_{j=0}^{m-1}\omega_{n}(u^{-j}(1+X)-1);\\
\delta_m(1+X)&=\prod_{j=0}^{m-1}(u^{-j}(1+X)-1);\\
\log_{p,m}(1+X)&=\prod_{j=0}^{m-1}\log_{p}(u^{-j}(1+X)-1).
\end{align*}
Note that $\Phi_{n,m}, \omega_{n,m}$ and $\delta_m$ are all polynomials over $\ZZ$, whereas $\log_{p,m}\in \cH_{\Qp,m}$.

When $P$ is a polynomial over $L$, we write $||P||=\sup_{|z|_p\le 1}|P(z)|_p$. Recall the following result from \cite[Lemme~1.2.1]{perrinriou94}.
\begin{lemma}\label{lem:PRseries}
Let $P_n$ be a sequence of polynomials in $L[X]$ and $h\ge0$ an integer such that
\begin{itemize}
\item[(i)] ${\displaystyle \lim_{n\rightarrow \infty}||p^{nh}P_n||=0}$;
\item[(ii)] $P_{n+1}\equiv P_n\mod \omega_{n,h}(1+X)L[X]$. 
\end{itemize}
Then, the sequence $P_n$ converges to an element  $P_\infty \in L[[X]]$ that is $o(\log_p^h)$.
\end{lemma}
We prove a slightly more general version of this result.
\begin{lemma}\label{lem:generalPR}
Let $P_n$ be a sequence of polynomials in $L[X]$ and $h\ge0$ an integer satisfying the same conditions in Lemma~\ref{lem:PRseries}. If there exists a real number $0\le r<h$ such that $\sup||p^{rn}P_n||<\infty$, then the limit $P_\infty$ of the sequence $P_n$ is $O(\log_p^r)$.
\end{lemma}
\begin{proof}Our proof is entirely based on the calculations carried out in Perrin-Riou's proof of Lemma~\ref{lem:PRseries}. Let $R_n$ be the polynomial given by 
\[
P_{n+1}-P_n=\omega_{n,h}R_n
\]
and let $C$ be  a fixed constant such that
\begin{equation}\label{eq:PR1}
||P_{n}||\le Cp^{rn}
\end{equation}
for all $n$. Since $\omega_{n,h}$ is monic, we have
\[
||R_n||\le \sup(||P_n||,||P_{n+1}||)\le Cp^{rn}.
\]
Consequently,
\begin{equation}\label{eq:PR2}
||P_{m+1}-P_m||_{\rho_n}\le C p^{rm}||\omega_{m,h}||_{\rho_n}\le Cp^{rm-(m-n)h}||\omega_{n,h}||_{\rho_n}<Cp^{rn}
\end{equation}
for all $m\ge n$, where the last inequality follows from our assumption that $r<h$. Therefore, if $P_\infty$ is the limit of the sequence $P_n$, we deduce from \eqref{eq:PR1} and \eqref{eq:PR2} that  $||P_\infty||_{\rho_n}\le Cp^{rn}$, as required.
\end{proof}
We shall also need the following generalization of \cite[Lemme 1.2.2]{perrinriou94}.
\begin{lemma}\label{lem:PRtwists}
Let $h\ge 0$ be an integer and $0\le r <h$. For $0\le j\le h-1$, let $Q_{n,j}$ be a sequence of polynomials in $L[X]$ satisfying
\begin{itemize}
\item[(a)] $\sup ||p^{rn}Q_{n,j}||<\infty$;
\item[(b)] $Q_{n+1,j}\equiv Q_{n,j}\mod \omega_n(1+X)L[X]$
\end{itemize}
for all $n$. Furthermore, suppose that
\[
\sup\left|\left| p^{(r-j)n} \sum_{i=0}^j(-1)^i\binom{j}{i}Q_{n,i}(u^{-i}(1+X)-1)\right|\right|<\infty
\]
for all $0\le j\le h-1$. Let $P_n\in L[X]$  the unique polynomial of degree $<hp^n$ such that
\[
P_n\equiv Q_{n,j}(u^{-j}(1+X)-1)\mod \omega_n(u^{-j}(1+X)-1)L[X]
\]
for all $0\le j\le h-1$ (such $P_n$ exists thanks to Chinese remainder theorem). Then, the sequence $P_n$ satisfies the hypotheses of Lemma~\ref{lem:generalPR}.
\end{lemma}
\begin{proof}
For $0\le j\le h-1$, let
\[
\delta_j=\sum_{i=0}^j(-1)^i\binom{j}{i}Q_{n,i}(u^{-i}(1+X)-1).
\]
Then, as in the proof of  \cite[Lemme 1.2.2]{perrinriou94}, there exists a constant $C$  such that
\[
||p^{rn}P_n||\le C\sup_{0\le j\le h-1}||p^{(r-j)n}\delta_j||,
\]
which is bounded above independently of $n$ by assumption. Hence,  condition (i) in Lemma~\ref{lem:generalPR} holds. Condition (ii) follows from condition (b), hence the result.
\end{proof}
\begin{remark}\label{rk:integral}
We can in fact say a little bit more about the norm of $P_n$ in the following specific case of interest. Suppose that $||p^{rn}Q_{n,j}||\le1$ and that
\[
\left|\left| p^{(r-j)n} \sum_{i=0}^j(-1)^i\binom{j}{i}Q_{n,i}(u^{-i}(1+X)-1)\right|\right|\le 1\]
for every $n$. The proof of Lemma~\ref{lem:PRtwists} in fact tells us that $||p^{rn}P_n||\le C$, where $C$ is a constant independent of the collection of polynomials $\{Q_{n,j}:0\le j \le h-1\}$ and the integer $n$.
\end{remark}

It is in fact possible to prove the same results if we replace the coefficient field $L$ by a $p$-adic Banach space. The same proofs would go through since we did not use any properties of $L$ being a field.
Let $\Gamma_0^\cyc=\Gal(\Qp(\mu_{p^\infty})/\Qp)\cong \Gal(\QQ(\mu_{p^\infty})/\QQ)$ and let $\Delta$ be the torsion subgroup of $\Gamma_0^\cyc$. We write $\chi_0$ for the cyclotomic character on $\Gamma_0^\cyc$ and  fix a topological generator $\gamma_0$ of $\Gamma_0^\cyc/\Delta$ such that $\chi_0(\gamma_0)=u$. For the rest of the article, we shall identify $\Gamma^{\cyc}$ with $\langle \gamma_0\rangle$. We define $\cH_{L,r}(\Gamma_0^\cyc)$ to be the set of power series
\[
\sum_{n\ge0,\sigma\in\Delta} c_{n,\sigma}\cdot\sigma\cdot (\gamma_0-1)^n\in L[\Delta][[\gamma_0-1]]
\]
such that $\sum_{n\ge0}c_{n,\sigma}X^n\in \cH_{L,r}$ for all $\sigma\in \Delta$. When $r=0$, we shall simply write $\Lambda_L(\Gamma_0^\cyc)$ instead of $\cH_{L,0}(\Gamma_0^\cyc)$. The set of integral power series
\[
\sum_{n\ge0,\sigma\in\Delta} c_{n,\sigma}\cdot\sigma\cdot (\gamma_0-1)^n\in \cO[\Delta][[\gamma_0-1]]
\]
will be denoted by $\Lambda_{\cO_L}(\Gamma_0^\cyc)$. As before, we write $\cH_L(\Gamma_0^\cyc)=\cup_{r\ge0}\cH_{L,r}(\Gamma_0^\cyc)$.

For $\q\in\{\p,\p^c\}$, we define $\Gamma_\q$ to be the Galois group of the  $\Zp$-extension $K(\q^\infty):=\cup\, K(\q^n)$ of $K$, where $K(\q^n)$  denotes the maximal $p$-extension of $K$ contained in the ray class field modulo $\mathfrak{q}^n$ as in the introduction. We may define $\cH_{L,r}(\Gamma_\q)$, $\cH_L(\Gamma_\q)$, $\Lambda_L(\Gamma_\q)$ and $\Lambda_{\cO_L}(\Gamma_\q)$ similarly. We similarly define $\Lambda_{\cO_L}(\Gamma^\cyc)$, $\Lambda_{\cO_L}(\Gamma^\ac)$, $\cH_{L,r}(\Gamma^\cyc)$, $\cH_{L,r}(\Gamma^\ac)$, etc.
\subsection{Wach module and logarithm matrix}\label{S:wach}
Let $f$ be a modular form as given in the introduction. We recall that  we assume that $p>k$ and that the Hecke eigenvalue $a_p(f)$ satisfies $\ord_p(a_p(f))>0$ throughout. We shall also fix an $L$-valued unramified character $\vt$ of $G_{\Qp}$ of finite order and write $T=R_f^*\otimes \vt$.
Let $\NN(T)$ and $\Dcris(T)$ be the Wach module and the Dieudonn\'e module of $T$ respectively (c.f., \cite{berger04}).
 We write $n_1,n_2$ and $v_1$, $v_2$ for the  $\cO_L\otimes\AQp$- and $\cO_L$-bases of $\NN(T)$ and $\Dcris(T)$, which are chosen as the twists of the bases of $\NN(R_f^*)$ and $\Dcris(R_f^*)$ given in \cite[\S3.1]{LLZ3} by the canonical bases of $\NN(\vt)$ and $\Dcris(\vt)$ given in \cite[Appendice A]{berger04}. Here, $\AQp=\Zp[[\pi]]$ and $\pi$ is equipped with actions by $\vp:\pi\mapsto(1+\pi)^p-1$  and $\sigma:\pi\mapsto(1+\pi)^{\chi_0(\sigma)}-1$ for $\sigma\in\Gamma_0^\cyc$. In the rest of this section $q$ denotes $\vp(\pi)/\pi$.
 
Let $c=\vt({\rm Fr}_p^{-1})\in \cO_L^\times$ and let $A_\vp$ and $P$  be the matrices of $\vp$ with respect to these two bases. 
  Then  
 \begin{equation}  A_\vp=c\cdot\begin{pmatrix}
0&\frac{-1}{\epsilon_f(p)p^{k-1}}\\
1&\frac{a_p(f)}{\epsilon_f(p)p^{k-1}}
\end{pmatrix}\quad\text{and}\quad\label{eq:formulaP}
P=c\cdot\begin{pmatrix}
0&\frac{-1}{\epsilon_f(p)q^{k-1}}\\
\delta^{k-1}&\frac{a_p(f)}{\epsilon_f(p)q^{k-1}}
\end{pmatrix}
\end{equation}
for some unit $\delta\in(\AQp)^\times$ by \cite[Proposition~V.2.3]{berger04}. 

We recall from \cite[Definition~3.2]{LLZ3} that there exists a logarithm matrix $\Mlog \in \textup{Mat}_{2\times2}\left(\cH_{L,k-1}(\Gamma^\cyc)\right)$ which is given by
\[
\Mlog:=\fM^{-1}((1+\pi)A_\vp M),
\]
where $M$ is the change of basis matrix
\[
\begin{pmatrix}
n_1& n_2
\end{pmatrix}
=\begin{pmatrix}
v_1& v_2
\end{pmatrix}M
\]
and $\fM$ is the Mellin transform. 
{\begin{remark}\label{rk:one}Note that  $A_\vp$ is defined over $L_0$ and $P$ is defined over $\cO_{L_0}\otimes \AQp[\frac{1}{q}]$. Thus, the matrix $\Mlog$ is in fact defined over $\cH_{L_0,k-1}(\Gamma^\cyc)$. \end{remark}}
\begin{lemma}\label{lem:Cn}
For $n\ge1$, the matrix $\Mlog$ is congruent to
$A_\vp^{n+1}\cdot C_n\mod\omega_{n,k-1}(\gamma_0)\cH_{L,k-1}(\Gamma^\cyc)$
for some $C_n\in \textup{Mat}_{2\times2}\left(\cO_L[\gamma_0-1]\right)$, whose entries are all polynomials of degree $<(k-1)p^n$. Furthermore, both entries on the second row of $C_n$ are divisible by $\Phi_{n,k-1}(\gamma_0)$.
\end{lemma}
\begin{proof}
As explained in \cite[Lemma~3.7]{LLZ3}, we have 
\[
\Mlog\equiv A_\vp^{n+1}\cdot \fM^{-1}\left((1+\pi)\vp^{n}(P^{-1})\cdots\vp(P^{-1})\right)\mod \omega_{n,k-1}(\gamma_0)\cH_{L,k-1}(\Gamma^\cyc).
\]
Therefore, we may take  $C_n$ to be the image of $\fM^{-1}\left((1+\pi)\vp^{n}(P^{-1})\cdots\vp(P^{-1})\right)$ modulo $ \omega_{n,k-1}(\gamma_0)$.
The entries of this matrix are in $\cO_L[\gamma_0-1]$ since 
\[
P^{-1}=c^{-1}\cdot\begin{pmatrix}
\frac{a_p(f)}{\delta^{k-1}}&\frac{1}{\delta^{k-1}}\\
-\epsilon_f(p)q^{k-1}&0
\end{pmatrix}
\]
has coefficients in $\cO_L\otimes\AQp$.
Note that the second row of $\vp^n(P^{-1})$ is divisible by $\vp^n(q^{k-1})$. Therefore, the second row of $C_n$ is divisible by $\Phi_{n,k-1}(\gamma_0)$ thanks to \cite[(2.1)]{LLZ3}.
\end{proof}

\begin{lemma}\label{lem:det}
The determinant of $C_n$ is, up to multiplication by a unit in $\Lambda_{\cO_L}(\Gamma^\cyc)^\times$, equal to $\frac{\omega_{n,k-1}(\gamma_0)}{\delta_{k-1}(\gamma_0)}$.
\end{lemma}
\begin{proof}
Recall from \cite[Corollary~3.2]{LLZ0.5} that up to a unit in $\Lambda_{L}(\Gamma^\cyc)^\times$, we have
\[
\det(\Mlog)\sim\frac{\log_{p,k-1}(\gamma_0)}{\delta_{k-1}(\gamma_0)}=\prod_{m\ge1}\frac{\Phi_{m,k-1}(\gamma_0)}{p^{k-1}}.
\]
For all $m>n$, \cite[Corollary~5.5]{lei11compositio} tells us that there exists some $u_{m,k}\in\Zp^\times$ such that $$\Phi_{m,k-1}(\gamma_0)\equiv p^{k-1}u_{m,k}\mod \omega_{n,k-1}(\gamma_0). $$
Therefore, Lemma~\ref{lem:Cn} implies that 
\[
\det(C_n)\equiv \epsilon\times \frac{\omega_{n,k-1}(\gamma_0)}{\delta_{k-1}(\gamma_0)}\mod \omega_{n,k-1}(\gamma_0)\Lambda_{L}(\Gamma^\cyc) 
\]
for some $\epsilon\in \Lambda_{L}(\Gamma^\cyc)^\times $.
In other words,  there exists $F\in\Lambda_{L}(\Gamma^\cyc)$ such that
\[
\det(C_n)=\epsilon\times \frac{\omega_{n,k-1}(\gamma_0)}{\delta_{k-1}(\gamma_0)}+ \omega_{n,k-1}(\gamma_0) F=\frac{\omega_{n,k-1}(\gamma_0)}{\delta_{k-1}(\gamma_0)}(\epsilon+\delta_{k-1}(\gamma_0)F).
\]
It remains to show that $\epsilon+\delta_{k-1}(\gamma_0)F\in \Lambda_{\cO_L}(\Gamma^\cyc)^\times$. Note that we already know that $\epsilon+\delta_{k-1}(\gamma_0)F\in \Lambda_{\cO_L}(\Gamma^\cyc)$ since $C_n$ is defined over $ \Lambda_{\cO_L}(\Gamma^\cyc)$. As $\epsilon\in \Lambda_L(\Gamma^\cyc)^\times$, it is contained inside $\varpi^r\Lambda_{\cO_L}(\Gamma^\cyc)^\times$ for some integer $r$. Without loss of generality, we may assume that $\epsilon=\varpi^r$. If we consider $\varpi^r+\delta_{k-1}(\gamma_0)F$ as a power series in $\gamma_0-1$, its constant term is $\varpi^r$  as $\gamma_0-1$ divides $\delta_{k-1}(\gamma_0)$. Thus, $r\ge0$. Consequently, $F$ is also integral. Furthermore, $\varpi^r+\delta_{k-1}(\gamma_0)F$ is a unit if and only if $r=0$.

Recall from the proof of Lemma~\ref{lem:Cn}  that
\[
C_n\equiv \fM^{-1}\left((1+\pi)\vp^n(P^{-1})\cdots \vp(P^{-1})\right)\mod \omega_{n,k-1}(\gamma_0).
\]
Modulo $p$, the matrix $\vp^i(P^{-1})$ is congruent, up to units, $\begin{pmatrix}
0&1\\\vp^i(q^{k-1})&0
\end{pmatrix}$. Therefore, modulo $(p,\omega_{n,k-1}(\gamma_0))$, $C_n$ is congruent to a matrix of the form 
\[
\fM^{-1}\left((1+\pi)\begin{pmatrix}
0&*\\ *&0
\end{pmatrix}\right)\quad\text{or}\quad \fM^{-1}\left((1+\pi)\begin{pmatrix}
*&0\\0&*
\end{pmatrix}\right),
\]
where $*$ represents, up to units,  a product of $\vp^i(q^{k-1})$ where $i\in[1,n ]$ and each $\vp^i(q^{k-1})$ appears exactly once in the two entries. This implies that
\[
\det(C_n)\equiv \epsilon'\frac{\omega_{n,k-1}(\gamma_0)}{\delta_{k-1}(\gamma_0)} \mod p
\]
for some $\epsilon'\in \Lambda_{\cO_L}(\Gamma^\cyc)^\times$ by \cite[Theorem~5.4]{LLZ0}. Hence,
\[
\epsilon'\equiv \omega^r+\delta_{k-1}(\gamma_0)F\mod p.
\]
This forces that $r=0$, as otherwise, we would have $\epsilon'\in (\varpi,\gamma_0-1)$, the maximal ideal of $\Lambda_{\cO_L}(\Gamma^\cyc)$, which is impossible.
\end{proof}

Observe that we have the diagonalization
\begin{equation}\label{eq:diag}
Q^{-1}\cdot A_\vp\cdot
Q=\begin{pmatrix}
\frac{c}{\alpha}&0\\
0&\frac{c}{\beta}
\end{pmatrix},
\end{equation}
where $Q=\begin{pmatrix}
\alpha&-\beta\\
-\epsilon_f(p)p^{k-1}&\epsilon_f(p)p^{k-1}
\end{pmatrix}$.

\begin{lemma}\label{lem:growth}
The entries of $Q^{-1}\Mlog$ on the first row are elements of $\cH_{L,\ord_p(\alpha)}(\Gamma^\cyc)$, whereas those on the second row are inside $\cH_{L,\ord_p(\beta)}(\Gamma^\cyc)$.
\end{lemma}
\begin{proof}
By Lemma~\ref{lem:Cn} and \eqref{eq:diag},
\[
Q^{-1}\cdot \Mlog\equiv\begin{pmatrix}
\left(\frac{c}{\alpha}\right)^{n+1}&0\\
0&\left(\frac{c}{\beta}\right)^{n+1}
\end{pmatrix}\cdot Q^{-1}\cdot C_n\mod \omega_{n,k-1}(\gamma_0).
\]
Furthermore, the entries of the matrix $C_n$ are elements of $\cO_L[[\gamma_0-1]]$. Let $C_n=\begin{pmatrix}
Q_{1,n}&Q_{2,n}\\
Q_{3,n}&Q_{4,n}
\end{pmatrix} \in \textup{Mat}_{2\times2}\left(\cO_L[\gamma_0-1]\right)$. We define the polynomials $P_{i,n}$ ($i=1,\cdots,4$) by setting $$\begin{pmatrix}
P_{1,n}&P_{2,n}\\
P_{3,n}&P_{4,n}
\end{pmatrix}=\begin{pmatrix}
\left(\frac{c}{\alpha}\right)^{n+1}&0\\
0&\left(\frac{c}{\beta}\right)^{n+1}
\end{pmatrix}\cdot Q^{-1}\cdot\begin{pmatrix}
Q_{1,n}&Q_{2,n}\\
Q_{3,n}&Q_{4,n}
\end{pmatrix} \in L[\gamma_0-1]\,.$$
These polynomials satisfy:
\begin{itemize}
\item ${\displaystyle \lim_{n\rightarrow \infty}||p^{n(k-1)}P_{i,n}||=0}$\, (for $i=1,\cdots,4$),
\item  $P_{i,n+1}\equiv P_{i,n}\mod \omega_{n,k-1}(\gamma_0)L[\gamma_0-1]$ \, (for $i=1,\cdots,4$),
\item  there exists an absolute constant $C$ such that 
$$p^{n\times\textup{ord}_p(\alpha)+C}P_{i,n} \in \cO_{L}[\gamma_0-1] \hbox{ (for $i=1,2$)}$$ 
$$p^{n\times\textup{ord}_p(\beta)+C}P_{i,n} \in \cO_{L}[\gamma_0-1] \hbox{ (for $i=3,4$)}.$$
\end{itemize}
If we let $P_{i,\infty}$ denote the limit of $P_{i,n}$ as $n$ tends to infinity, we have $Q^{-1}\cdot \Mlog=\begin{pmatrix}
P_{1,\infty}&P_{2,\infty}\\
P_{3,\infty}&P_{4,\infty}
\end{pmatrix}$. Furthermore, it follows from Lemma~\ref{lem:generalPR} that 
$$P_{1,\infty}, P_{2,\infty}\in \cH_{L,\ord_p(\alpha)}(\Gamma^\cyc) \hbox{ and } P_{3,\infty}, P_{4,\infty}\in \cH_{L,\ord_p(\beta)}(\Gamma^\cyc)\,.$$
Our assertion follows.
\end{proof}

Let $h_n$ denote the morphism
\begin{align*}
h_n:\Lambda_{\cO_L}(\Gamma^\cyc)^{\oplus2}&\rightarrow \Lambda_{\cO_L}(\Gamma^\cyc)^{\oplus2}/\omega_{n,k-1}(\gamma_0)\\
\begin{pmatrix}
F\\ G
\end{pmatrix}&\mapsto C_n\begin{pmatrix}
F\\ G
\end{pmatrix}\mod \omega_{n,k-1}(\gamma_0),
\end{align*}
where $C_n$ is as defined in Lemma~\ref{lem:Cn}. Following \cite{sprungfactorizationinitial}, we  consider the inverse system $\Lambda_{\cO_L}(\Gamma^\cyc)^{\oplus 2}/\ker h_n$, where the connecting maps are the natural projections and obtain the following:
\begin{lemma}\label{lem:inverselimit}
The inverse limit $\varprojlim\Lambda_{\cO_L}(\Gamma^\cyc)^{\oplus 2}/\ker h_n$ is isomorphic to $\Lambda_{\cO_L}(\Gamma^\cyc)^{\oplus 2}$.
\end{lemma}
\begin{proof}
It is enough to show that $\varprojlim \ker h_n=0$. That is, if $\underline{F}\in \Lambda_{\cO_L}(\Gamma^\cyc)^{\oplus2}$ such that $C_n\underline{F}\equiv 0\mod \omega_{n,k-1}(\gamma_0)$ for all $n\ge1$, then $\underline{F}=0$. Indeed, such an element would satisfy
\[
\Mlog \underline{F}= \log_{p,k-1}\underline{G}
\]
for some $\underline{G}\in\cH_{L,k-1}(\Gamma^\cyc)^{\oplus 2}$.
But the right-hand side is at least $O(\log_p^{k-1})$ if non-zero, whereas the left-hand side is $o(\log_p^{k-1})$ thanks to Lemma~\ref{lem:growth}. Hence the result follows.
\end{proof}
We now explain how to construct a well-defined element in the quotient $\Lambda_{\cO_L}(\Gamma^\cyc)^{\oplus 2}/\ker h_n$ when we are given a pair of polynomials satisfying certain compatible conditions when evaluated at a collection of characters.
\begin{proposition}\label{prop:prefactor}
Let $n\ge1$ be an integer. For each $\lambda\in\{\alpha,\beta\}$ and $0\le j\le k-2$, let  $P_{\lambda,n,j}\in \cO_L[\gamma_0-1]$ be a polynomial such that 
\[
\left|\left| \sum_{i=0}^j(-1)^i\binom{j}{i}P_{\lambda,n,i}(u^{-i}\gamma_0-1)\right|\right|\le p^{-jn}
\]
for all $0\le j\le k-2$. Furthermore, suppose that for all $j\in\{0,\ldots,k-2\}$ and all Dirichlet characters $\theta$ of conductor $p^m$ where $1\le m\le n+1$, we have
\[
\alpha^{m-n-1}P_{\alpha,n,j}(\theta)=\beta^{m-n-1}P_{\beta,n,j}(\theta).
\]
 
 Let $P_{\lambda,n}$ be the unique polynomial of degree $<(k-1)p^n$ such that
 \[
 P_{\lambda,n}\equiv P_{\lambda,n,j}(u^{-j}\gamma_0-1)\mod \omega_n(u^{-j}\gamma_0-1).
 \]
 Then, there exists a unique element $(P_{\#,n},P_{\flat,n})\in \varpi^{-s}\Lambda_{\cO_L}(\Gamma^\cyc)/\ker h_n$, where $s$ is a fixed integer independent of the collection of polynomials $\{P_{\lambda,n,j}:0\le j\le k-2\}$ and the integer $n$, such that
 \begin{equation}\label{eq:factorization}
  \begin{pmatrix}
P_{\alpha,n} \\ P_{\beta,n}
\end{pmatrix}\equiv Q^{-1} C_n\cdot
\begin{pmatrix}
P_{\#,n}\\P_{\flat,n}
\end{pmatrix}\mod \omega_{n,k-1}(\gamma_0).
 \end{equation}
\end{proposition}
\begin{proof} Let $\adj$ be the adjugate matrix of $Q^{-1}C_n$. Recall from Lemma~\ref{lem:Cn} that the second row of $C_n$ is divisible by $\Phi_{n,k-1}(\gamma_0)$. Therefore,
\[
\adj\equiv \begin{pmatrix}
D_n&-D_n\\
D_n'&-D_n'
\end{pmatrix}\mod \Phi_{n,k-1}(\gamma_0)
\]
for some polynomials $D_n$ and $D_n'$, whose denominators are bounded independent of $n$.

Let $0\le m\le n$. As in the proof of Lemma~\ref{lem:Cn}, we have
\[
A_\vp^{n-m} C_{n}\equiv C_m\mod \omega_{m,k-1}(\gamma_0).
\]
Consequently,
\[
\begin{pmatrix}
\left(\frac{c}{\alpha}\right)^{n-m}&0\\
0&\left(\frac{c}{\beta}\right)^{n-m}
\end{pmatrix}Q^{-1}C_n\equiv Q ^{-1}C_m\mod \omega_{m,k-1}(\gamma_0).
\]
Since the second row of $C_m$ is divisible by $\Phi_{m,k-1}(\gamma_0)$, we deduce that
\[
\adj\equiv \begin{pmatrix}
\left(\frac{c}{\alpha}\right)^{n-m}D_m&-\left(\frac{c}{\beta}\right)^{n-m}D_m\\
\left(\frac{c}{\alpha}\right)^{n-m}D_m'&-\left(\frac{c}{\beta}\right)^{n-m}D_m'
\end{pmatrix}\mod \Phi_{m,k-1}(\gamma_0).
\]
Therefore, if $\theta$ is a Dirichlet character of conductor $p^m$, where $1\le m\le n+1$ (which sends $\gamma_0$ to a primitive $p^{m-1}$-root of unity), then
\[
\adj(\theta)= \begin{pmatrix}
\left(\frac{c}{\alpha}\right)^{n+1-m}*&-\left(\frac{c}{\beta}\right)^{n+1-m}*\\
\left(\frac{c}{\alpha}\right)^{n+1-m}*'&-\left(\frac{c}{\beta}\right)^{n+1-m}*'
\end{pmatrix}
\]
for some  $*$ and $*'$ in $\overline{\Qp}$.
Our assumptions on $P_{\alpha,n,j}$ and $P_{\beta,n,j}$ now imply that
\[
\adj\begin{pmatrix}
P_{\alpha,n}\\ P_{\beta,n}
\end{pmatrix}\equiv 0\mod \Phi_{m,k-1}(\gamma_0)
\]
for all $1\le m\le n$. In particular, there exist  polynomials $P_{\#,n}$ and $P_{\flat,n}$ such that
\begin{equation}\label{eq:factoradj}
\adj\begin{pmatrix}
P_{\alpha,n}\\ P_{\beta,n}
\end{pmatrix}= \prod_{m=1}^n\Phi_{m,k-1}(\gamma_0)\begin{pmatrix}
P_{\#,n}\\ P_{\flat,n}
\end{pmatrix}.
\end{equation}
Remark~\ref{rk:integral} tells us that the polynomials $P_{\lambda,n}\in p^{-s_0}\cO[\gamma_0-1]$ for some $s$ independent of $P_{\lambda,n,j}$. In particular, the denominators of $P_{\#,n}$ and $P_{\flat,n}$ are also  bounded by $\varpi^{s}$.
Recall from Lemma~\ref{lem:det} that $\det(C_n)$ is up to a unit in $\Lambda_{\cO_L}(\Gamma^\cyc)^\times$ equal to $\prod_{m=1}^n\Phi_{m,k-1}(\gamma_0)$.
Therefore, we deduce the factorization \eqref{eq:factorization} on inverting $\adj$ in \eqref{eq:factoradj}.
\end{proof}

\begin{proposition}\label{prop:decomp}
For each $\lambda\in\{\alpha,\beta\}$, suppose that we are given $F_\lambda\in \cH_{L,\ord_p(\lambda)}(\Gamma_0^\cyc)$ such that for all $j\in\{0,\ldots,k-2\}$ and all Dirichlet characters $\theta$ of conductor $p^n>1$,
\[
F_\lambda(\chi_0^j\theta)=\lambda^{-n}\times c_{j,\theta}
\]
for some constant $c_{j,\theta}\in {\bar{\QQ}_p}$ that is independent of the choice of $\lambda$. Then there exist  $F_\#,F_\flat\in\Lambda_{L}(\Gamma_0^\cyc)$ such that
\[
\begin{pmatrix}
F_\alpha \\F_\beta
\end{pmatrix}=Q^{-1}\Mlog\cdot
\begin{pmatrix}
F_\#\\ F_\flat
\end{pmatrix}.
\]
Furthermore, if there exists a sequence  $P_{\lambda,n}\in\Lambda_{\cO_L}(\Gamma_0^\cyc)$ such that $F_\lambda\equiv \lambda^{-n-1}P_{\lambda,n}\mod\omega_{n,k-1}(\gamma_0)$ for both $\lambda\in\{\alpha,\beta\}$, then there exists an integer $s$ that depends only on the pair $\{\alpha,\beta\}$ (but not on $F_\alpha$ and $F_\beta$) such that $F_\#,F_\flat\in\varpi^{-s}\Lambda_{\cO_L}(\Gamma_0^\cyc)$.
\end{proposition}
\begin{proof}
On considering each isotypic component separately, we may assume that $F_\lambda\in\cH_{L,\ord_p(\lambda)}(\Gamma^\cyc)$.
For $\lambda\in\{\alpha,\beta\}$ and $0\le j\le k-2$, take $P_{\lambda,n,j}$ to be the unique polynomial of degree $<p^n$ such that
$$\lambda^nF_\lambda\equiv P_{\lambda,n,j}(u^{-j}\gamma_0-1)\mod \omega_n(u^{-j}\gamma_0-1).$$
Since $F_\lambda\in \cH_{L,\ord_p(\lambda)}(\Gamma_0^\cyc)$, we have $P_{\lambda,n,j}\in\varpi^{-s_0}\cO_L[\gamma_0-1]$ for some $s_0$ that is independent of $\lambda$, $n$ and $j$.
By Proposition~\ref{prop:prefactor}, our hypothesis implies that there exist two sequences of polynomials $P_{\#,n}$ and $P_{\flat,n}$ in $ \omega^{-s-s_0}\cO_L[\gamma_0-1]$ satisfying
\[
\begin{pmatrix}
\alpha^{n+1}F_\alpha\\ \beta^{n+1} F_\beta
\end{pmatrix}\equiv Q^{-1}C_n\begin{pmatrix}
P_{\#,n}\\ P_{\flat,n}
\end{pmatrix}\mod\omega_{n,k-1}(\gamma_0).
\]
Recall from the proof of Lemma~\ref{lem:growth} that we have the congruence
\[
Q^{-1}\cdot \Mlog\equiv\begin{pmatrix}
\left(\frac{c}{\alpha}\right)^{n+1}&0\\
0&\left(\frac{c}{\beta}\right)^{n+1}
\end{pmatrix}\cdot Q^{-1}\cdot C_n\mod \omega_{n,k-1}(\gamma_0).
\]
Consequently,
\[
\begin{pmatrix}
F_\alpha\\  F_\beta
\end{pmatrix}\equiv Q^{-1}\Mlog\begin{pmatrix}
c^{n+1}P_{\#,n}\\c^{n+1} P_{\flat,n}
\end{pmatrix}\mod\omega_{n,k-1}(\gamma_0).
\]
In particular, $\begin{pmatrix}
c^{n+1}P_{\#,n}\\c^{n+1} P_{\flat,n}
\end{pmatrix}$ defines a sequence of compatible elements in $\varpi^{-s-s_0}\Lambda_{\cO_L}(\Gamma^\cyc)^{\oplus2}/\ker h_n$. Lemma~\ref{lem:inverselimit} tells us that there exist $F_\#,F_\flat\in \varpi^{-s-s_0}$ such that
\[
\begin{pmatrix}
F_\alpha\\  F_\beta
\end{pmatrix}= Q^{-1}\Mlog\begin{pmatrix}
F_{\#}\\ F_{\flat}
\end{pmatrix}\
\]
on letting $n\rightarrow \infty$. The last part of the proposition is the special case where $s_0=0$.
\end{proof}
{
\begin{remark}\label{rk:two}
Recall from Remark~\ref{rk:one} that the entries of $\Mlog$ in fact have coefficients in $L_0$. Thus, we can see from the proofs of Lemma~\ref{lem:growth} and Proposition~\ref{prop:decomp} that if $P_{\lambda,n}\in \Lambda_{\cO_{L_0}}(\Gamma_0^\cyc)$, then $F_\#,F_\flat\in \Lambda_{\cO_{L_0}}(\Gamma_0^\cyc)$.
\end{remark}
}

\subsection{Two-variable Coleman maps}\label{S:2varCol}
In this section, we construct two-variable Coleman maps for the representation $T=R_f^*\otimes \vt$ (where we recall that $\vt$ is an unramified character of $G_{\QQ_p}$ of finite order).
First of all, let us quickly review the construction of one-variable Coleman maps. It was shown in \cite[Theorem~3.5]{LLZ0} and \cite[Theorems~2.5 and 3.3]{LLZ3} that $\{(1+\pi)\vp(n_1), (1+\pi)\vp(n_2)\}$ is a $\Lambda_{\cO_E}(\Gamma_0^\cyc)$-basis of $(\vp^*\NN(T))^{\psi=0}$ under the hypothesis $k>2$. When $k=2$, it is possible to obtain the same result on replacing the basis $n_1,n_2$ by another basis that is congruent to the original one mod $\pi$, as explained in  \cite[Lemma~3.9]{LLZ0}. The new basis would still be a lifting of $v_1,v_2$, but $\Mlog$ would differ slightly since $P$ would be replaced by a new matrix. However, this does not affect our calculations carried out in \S\ref{S:wach} as the new  $P$ would be congruent to the original one mod $\pi$. Without loss of generality, we shall  assume that our basis $n_1,n_2$ does satisfy \cite[Theorem~3.5]{LLZ0}.

Under this assumption, there exist two Coleman maps  $\col_\#$ and $\col_\flat$, both of which are $\Lambda_{\cO_L}(\Gamma_0^\cyc)$-morphisms from $\HIw(\Qp(\mu_{p^\infty}),T)$ into $\Lambda_{\cO_L}(\Gamma_0^\cyc)$ such that
\[
(1-\vp)(z)=\col_\#(z)\cdot (1+\pi)\vp(n_1)+\col_\flat(z)\cdot (1+\pi)\vp(n_2)
\]
for all $z\in\HIw(\Qp(\mu_{p^\infty}),T)$. Here, we identify  $\NN(R_f^*)^{\psi=1}$ with $\HIw(\Qp(\mu_{p^\infty}),T)$ via the Fontaine-Herr complex (c.f. \cite[\S A]{berger04} and \cite[Theorem~2.6]{LLZ3}).
If we write 
\[
\cL_T:\HIw(\Qp(\mu_{p^\infty}),T)\rightarrow \cH_{L,k-1}(\Gamma_0^\cyc)\otimes\Dcris(T)
\]
for the  Perrin-Riou big logarithm map, then
\begin{equation}\label{eq:defnColeman}
\cL_T=\begin{pmatrix}
v_1&v_2
\end{pmatrix}\Mlog\begin{pmatrix}
\col_\#\\\col_\flat
\end{pmatrix}.
\end{equation}
Our choice of eigenvectors $v_\alpha$, $v_\beta$ from \eqref{eq:diag} allows us to rewrite this as
\[
\cL_T=\begin{pmatrix}
v_\alpha&v_\beta
\end{pmatrix}Q^{-1}\Mlog\begin{pmatrix}
\col_\#\\\col_\flat
\end{pmatrix}.
\]
If we write $\cL_{T,\alpha}$ and $\cL_{T,\beta}$ for the coordinates of $\cL_T$ with respect to this eigenbasis, we have
\begin{equation}\label{eq:notanotherdecomp}
\begin{pmatrix}
\cL_{T,\alpha}\\ \cL_{T,\beta}
\end{pmatrix}=Q^{-1}\Mlog\begin{pmatrix}
\col_\#\\\col_\flat
\end{pmatrix}.
\end{equation}

Let $F$ be a finite unramified extension of $\Qp$. We shall write $\NN_F(T)$ and $\Dcris(F,T)$ for the Wach module and the Dieudonn\'e module of $T$ over $F$. In particular, \cite[Lemma~ 2.1]{LZ0} tells us that
\[
\NN_{F}(T)=\cO_{F}\otimes_{\Zp}\NN(T),\quad \Dcris(F,T)=\cO_{F}\otimes_{\Zp}\Dcris(T).
\]
Therefore, it is immediate that $n_1,n_2$ and $v_1,v_2$ extend to bases of $\NN_{F}(T)$ and $\Dcris(F,T)$ respectively. Furthermore, the basis $(1+\pi)\vp(n_1), (1+\pi)\vp(n_2)$ of $(\vp^*\NN(T))^{\psi=0}$ extends to an $\cO_F\otimes\Lambda_{\cO_L}(\Gamma_0^\cyc)$-basis of $(\vp^*\NN_F(T))^{\psi=0}$.  This allows us to define the Coleman maps 
\[
\col_{\#,F},\col_{\flat,F}:\HIw(F(\mu_{p^\infty}),R_f^*)\rightarrow \cO_F\otimes_{\Zp}\Lambda_{\cO_L}(\Gamma_0^\cyc)
\]
in the same manner as before. If we write 
\[
\cL_{T,F}:\HIw(F(\mu_{p^\infty}),T)\rightarrow \cH_{L,k-1}(\Gamma_0^\cyc)\otimes\Dcris(F,T)
\]
for the Perrin-Riou big logarithm map, we then have a similar decomposition as \eqref{eq:defnColeman}, namely
\[
\cL_{T,F}=\begin{pmatrix}
v_1&v_2
\end{pmatrix}\Mlog\begin{pmatrix}
\col_{\#,F}\\ \col_{\flat,F}
\end{pmatrix}.
\]

Let $F_\infty$ be the unramified $\Zp$-extension of $F$,  $U=\Gal(F_\infty/F)$, $G=\Gal(F_\infty(\mu_{\mu_{p^\infty}})/F)\cong U\times \Gamma_0^\cyc$ and let $\widehat F_\infty$ denote the completion of $F_\infty$.  We set $S_{F_\infty/F}\subset \Lambda_{\cO_{\widehat F_\infty}}(U)$  to denote the Yager module which is defined in \cite[\S3.2]{LZ0}. We recall that it is a free $\Lambda_{\cO_F}(U)$-module of rank 1 (see the proof of Proposition 3.12 in \textit{op. cit.}). In particular, we may fix a basis $\{\Omega_F\}$ of the Yager module $S_{F_{\infty}/F}$. There is an isomorphism
\[
\HIw(F_\infty(\mu_{\mu_{p^\infty}}),T)\cong \NN_{F_\infty}(T)^{\psi=1},
\]
where $\NN_{F_\infty}(T)=\NN_F(T)\widehat\otimes_{\cO_F} S_{F_\infty/F}$. 
The two-variable big logarithm map of Loeffler-Zerbes, which we denote by $\cL_{T,F_\infty}$, is defined as the compositum of the arrows
\begin{align*}
\HIw(F_\infty(\mu_{p^\infty}),T)&\stackrel{\cong}{\longrightarrow}\NN_{F_\infty}(T)^{\psi=1}\\
&\stackrel{1-\vp}{\longrightarrow}(\vp^*\NN_F(T))^{\psi=0}{\widehat\otimes}_{\cO_F}S_{F_\infty/F}\\
&\stackrel{\subset}{\longrightarrow}(\cO_F\otimes \cH_{L,k-1}(\Gamma_0^\cyc)\otimes_{\cO_F}\Dcris(F,T)){\widehat\otimes}_{\cO_F}S_{F_\infty/F}\\
&\stackrel{\cong }{\longrightarrow}\Omega_F\cdot(\cH_{L,k-1}(\Gamma_0^\cyc)\widehat\otimes \Lambda_{\cO_F}(U))\otimes_{\cO_F}\Dcris(F,T).
\end{align*}

The basis $(1+\pi)\vp(n_1),(1+\pi)\vp(n_2)$ of $(\vp^*\NN_F(T))^{\psi=0}$ allows us to define the Coleman maps 
\begin{equation}\label{eq:decompLTF}
\col_{\#,F_\infty},\col_{\flat,F_\infty}:\HIw(F_\infty(\mu_{p^\infty}),T)\rightarrow \cO_{\widehat F_{\infty}}\otimes\Lambda_{\cO_L}(G)
\end{equation}
via
\begin{align*}
\HIw(F_\infty(\mu_{p^\infty}),T)&\stackrel{\cong}{\longrightarrow}\NN_{F_\infty}(T)^{\psi=1}\\
&\stackrel{1-\vp}{\longrightarrow}(\vp^*\NN_F(T))^{\psi=0}{\widehat\otimes}_{\cO_F}S_{F_\infty/F}\\
&\stackrel{\cong}{\longrightarrow}\left(\cO_F\otimes\Lambda_{\cO_L}(\Gamma_0^\cyc)^{\oplus2}\right){\widehat\otimes}_{\cO_F}\Omega_F\cdot \Lambda_{\cO_F}(U)\\
&\stackrel{= }{\longrightarrow}\Omega_F\cdot \cO_{ F}\otimes\Lambda_{\cO_L}(G)^{\oplus2}.
\end{align*}
In particular, this yields the decomposition 
\[
\cL_{T,F_\infty}=\begin{pmatrix}
v_1&v_2
\end{pmatrix}\Mlog\begin{pmatrix}
\col_{\#,F_\infty}\\ \col_{\flat,F_\infty}
\end{pmatrix}.
\]
When no confusion could arise, we shall omit $\Omega_F$ from the notation and identify $\Lambda_{\cO_F}(U)$ with $\Omega_F\Lambda_{\cO_F}(U)$. This allows us to consider $\col_{\bullet, F_\infty}$ as maps landing in $\cO_F\otimes\Lambda_{\cO_L}(G)$.

\subsection{Images of Coleman maps}
We now study the images of the Coleman maps defined above in the case where $F=\Qp$.
\begin{lemma}
Let $j\in\{0,\ldots, k-2\}$, $z\in\HIw(\Qp(\mu_{p^\infty}),T)$ and $\theta$ be a Dirichlet character of conductor $p$. Then,
\begin{align*}
(p^{k-j-2}-a_p(f)c+\epsilon_f(p)^{-1}p^jc^2)\col_\flat(z)(\chi_0^j)&=c(p-1)p^{k-2}\col_\#(z)(\chi_0^j)\,,\\
\col_\flat(z)(\chi_0^j\theta)&=0
\end{align*}
where $c$ is as at the start of \S\ref{S:wach}.
\end{lemma}
\begin{proof}
If we write
\[
\begin{pmatrix}
\cL_1\\
\cL_2
\end{pmatrix}
=\Mlog\begin{pmatrix}
\col_\#\\
\col_\flat
\end{pmatrix},
\]
then following the calculations of \cite[\S5.1]{LLZ0.5}, we have
\begin{align*}
(\epsilon_f(p)p^{k-j-1}-a_p(f)c+p^{j+1}c^2)\cL_1(z)(\chi_0^j)&=c(1-p)\cL_2(z)(\chi_0^j);\\
\cL_1(z)(\chi_0^j)&=0.
\end{align*}
Lemma~\ref{lem:Cn} implies that 
\[
 \Mlog(\chi_0^j\theta)=\Mlog(\chi_0^j)=A_\vp=c\cdot\begin{pmatrix}
 0&-\frac{1}{\epsilon_f(p)p^{k-1}}\\
 1&\frac{a_p(f)}{\epsilon_f(p)p^{k-1}}
 \end{pmatrix}.
\]
Hence the result.
\end{proof}

As in \cite[Remark~5.8]{LLZ0.5}, we have $p^{k-j-1}-a_p(f)c+\epsilon_f(p)^{-1}p^{j+1}c^2\ne 0$ thanks to the Ramanujan-Petersson bound. Given a character $\eta\in\widehat\Delta$, we write $e_{\eta}$ for the idempotent attached to $\eta$ and 
$$C_{j,\eta}=\begin{cases}
\frac{c(1-p)}{p^{k-j-1}-a_p(f)c+\epsilon_f(p)^{-1}p^{j+1}c^2}&\text{if $\theta$ is trivial,}\\
0&\text{otherwise.}
\end{cases}
$$
\begin{theorem}\label{thm:1varimage}
If $\ucol:=(\col_\#,\col_\flat)$ and $\eta\in\widehat\Delta$, then 
\[
e_\eta\cdot L\otimes_{\cO_L} \image(\ucol)=\left\{(F_1,F_2)\in (L\otimes_{\cO_L}\cO_L[[X]])^{\oplus 2}:F_2(u^j-1)=C_{j,\eta\omega^{-j}}\cdot F_1(u^j-1), 0\le j\le k-2\right\},
\]
where we identify $\gamma_0-1$ with $X$ and $e_\eta\cdot \Lambda_{\cO_L}(\Gamma_0^\cyc)$ with $\cO_L[[X]]$.
For $\bullet\in\{\#,\flat\}$ and $\eta\in\widehat\Delta$, there exists a factor $\xi_{\eta,\bullet}$ of $\delta_{k-1}(\gamma_0)$ such that 
\[
e_\eta\cdot\image(\col_{\bullet}) \subset \xi_{\eta,\bullet}\cO_L[[X]].
\]
Furthermore, the containment is of finite index.
\end{theorem}
\begin{proof}
This follows from \cite[Corollary~5.3 and Theorem~5.10]{LLZ0.5}.
\end{proof}
\begin{remark}\label{rk:image}
We note that $\xi_{\eta,\#}=1$ for all $\eta$, whereas $\xi_{\eta,\flat}$ contains a non-trivial factor whenever $C_{j,\eta}=0$ for some $j$. The latter occurs for all $\eta$ if $k>2$. In the case $k=2$, $\xi_{\eta,\flat}$ is trivial for $\theta=1$, whereas $\xi_{\eta,\flat}=\gamma_0-1$ if $\eta\ne1$.
\end{remark}

\begin{lemma}\label{lem:2varimage}
For $\bullet\in\{\#,\flat\}$ and $\eta\in\widehat\Delta$, let $\xi_{\eta,\bullet}$ be as given in Theorem~\ref{thm:1varimage}, then
\[
e_\eta\cdot\image(\col_{\bullet, \Qpinf}) \subset e_\eta\Omega_{\Qp}\xi_{\eta,\bullet}\cdot \Lambda_{\cO_L}(G),
\]
where $G=\Gal(\Qpinf(\mu_{p^\infty})/\Qp)$. Furthermore, the containment is of finite index.
\end{lemma}
\begin{proof}
Given that
\[
\NN_{\Qpinf}(T)=\NN_{\Qp}(T)\widehat\otimes S_{\Qpinf/\Qp},
\]
we have
\begin{equation}\label{eq:Zp2image}
\image(\col_{\bullet, \Qpinf})= \image(\col_{\bullet, \Qp})\widehat\otimes S_{\Qpinf/\Qp}= \image(\col_{\bullet, \Qp})\widehat\otimes \Omega_{\Qp}\Lambda_{\Zp}(U).
\end{equation}
Therefore, our result follows from combining the second half of Theorem~\ref{thm:1varimage} with \eqref{eq:Zp2image}.
\end{proof}

Our description on the images allows us to describe the following \emph{error term}, which we shall need later.
\begin{proposition}\label{prop:errorterms}
Let $\eta\in\widehat\Delta$, then the $e_\eta$-component of the quotient 
\[
\frac{\HIw(\Qp(\mu_{p^\infty}),T)}{\ker\col_{\#}+\ker\col_{\flat}}
\]
is $\cO_L[[X]]$-torsion. Furthermore, its characteristic ideal, up to a power of $\varpi$, is generated by $\delta_{k-1}(\gamma_0)/\xi_{\eta,\flat}$.
\end{proposition}
\begin{proof}
The isomorphism theorem tells us that the stated quotient is isomorphic to
\[
\image(\col_{\#})/\col_\#(\ker\col_\flat).
\]
Recall from Remark~\ref{rk:image} that $\col_\#$ is pseudo-surjective, so we may replace this by
\[
\Lambda_{\cO_L}(\Gamma_0^\cyc)/\col_\#(\ker\col_\flat).
\]
The first half of Theorem~\ref{thm:1varimage} says that $e_\eta \cdot L\otimes \image(\ucol)$ equals
\[
\cF:=\left\{(F_1,F_2)\in (L\otimes\cO_L[[X]])^{\oplus 2}:F_2(u^j-1)=C_{j,\eta\omega^{-j}}\cdot F_1(u^j-1), 0\le j\le k-2\right\}.
\]
Therefore,
\[
e_\eta\cdot L\otimes \col_\#(\ker\col_\flat)=\{F_1\in L\otimes \cO_L[[X]]:(F_1,0)\in \cF\}.
\]
This in turn can be expressed as
\[
\{F_1\in L\otimes \cO_L[[X]]:F_1(u^j-1)=0, \forall j \text{ such that } C_{j,\eta\omega^{-j}}\ne 0\}.
\]
The result follows.
\end{proof}
This in turn yields the two-variable analogue  on tensoring everything by $\Omega_{\Qp}\Lambda_{\Zp}(U)$. More specifically, we may deduce that  the $e_\eta$-component of the quotient 
\[
\frac{\HIw(\Qpinf(\mu_{p^\infty}),T)}{\ker\col_{\#,\Qpinf}+\ker\col_{\flat,\Qpinf}}
\]
is $\cO_L[[X]]\widehat\otimes \Lambda_{\Zp}(U)$-torsion. Furthermore, its characteristic ideal, up to a power of $\varpi$, is generated by $\delta_{k-1}(\gamma_0)/\xi_{\eta,\flat}$ as given in Proposition~\ref{prop:errorterms}.

We shall need a similar result for the maps $\cL_{T,\alpha}$ and $\cL_{T,\beta}$. To prove this, we begin with the following observation on their images. 

\begin{lemma}\label{lem:imageL}
Let $j\in\{0,\ldots, k-2\}$, $\theta$ a Dirichlet character of conductor $p^n>1$ and $z\in\HIw(\Qp(\mu_{p^\infty}),T)$. Then,
\begin{align*}
\left(1-\frac{\beta}{p^jc}\right)\left(1-\frac{p^{j+1}c}{\alpha}\right)\cL_{T,\alpha}(z)(\chi_0^j)&=\left(1-\frac{\alpha}{p^jc}\right)\left(1-\frac{p^{j+1}c}{\beta}\right)\cL_{T,\beta}(z)(\chi_0^j);\\
\alpha^n\cL_{T,\alpha}(z)(\chi_0^j\theta)&=\beta^n\cL_{T,\beta}(z)(\chi_0^j\theta).
\end{align*}
\end{lemma}
\begin{proof}
This follows from the interpolation formulae of the Perrin-Riou logarithm map. See for example \cite[Lemma~3.5]{lei11compositio}.
\end{proof}
For a Dirichlet character $\eta$ of conductor $p^n$ and $0\le j\le k-2$, define
\[
C_{j,\eta}
=\begin{cases}
\frac{(1-\frac{\alpha}{p^jc})(1-\frac{p^{j+1}c}{\beta})}{(1-\frac{\beta}{p^jc})(1-\frac{p^{j+1}c}{\alpha})}&\text{if $\eta$ is trivial,}\\
\frac{\beta^n}{\alpha^n}&\text{otherwise.}
\end{cases}
\]

\begin{corollary}\label{cor:Limagepair}
Let $\underline{\cL}=(\cL_{T,\alpha},\cL_{T,\beta})$. For all $\eta\in\widehat\Delta$, we have
 $$e_\eta\cdot\underline{\cL}\left(\HIw(\Qp(\mu_{p^\infty}),T)\otimes \cH_L(\Gamma_0^\cyc)\right)=\cF_\eta,$$ where
\[
\cF_\eta= \left\{(F_1,F_2)\in \cH_L^{\oplus2}:F_1(\theta(\gamma_0) u^j-1)=C_{j,\eta\omega^{-j}\theta}F_2(\theta(\gamma_0)u^j-1),\forall 0\le j\le k-2,\text{ finite $\theta\in\widehat\Gamma_0^\cyc$ with }\theta|_{\Delta}=1\right\}.
\]
\end{corollary}
\begin{proof}
The inclusion $\subset$ is a reformulation of  Lemma~\ref{lem:imageL}. The equality then follows from comparing determinants of the two sets, which are both $\log_{p,k-1}(1+X)\cH_L$.
\end{proof}
On projecting to the two coordinates of $\underline{\cL}$, we infer the following:
\begin{corollary}\label{cor:imageLlambda}
 For all $\eta\in\widehat\Delta$ and $\lambda\in\{\alpha,\beta\}$, we have
 $$e_\eta\cdot{\cL_{T,\lambda}}\left(\HIw(\Qp(\mu_{p^\infty}),T)\otimes \cH_L(\Gamma_0^\cyc)\right)=\cH_L.$$
\end{corollary}

\begin{proposition}\label{prop:errorL}
Let $\eta\in\widehat\Delta$, then the $e_\eta$-component of the quotient
\[
\frac{\HIw(\Qp(\mu_{p^\infty}),T)\otimes \cH_L(\Gamma_0^\cyc)}{\ker\cL_{T,\alpha}+\ker\cL_{T,\beta}}
\]
is isomorphic to $\cH_L/(\log_{p,k-1}(1+X))$ as $\cH_L$-modules.
\end{proposition}
\begin{proof}
As in the proof of Proposition~\ref{prop:errorterms}, the quotient in the statement of the proposition is isomorphic to
\[
\frac{e_\eta\cdot\cL_{T,\alpha}(\HIw(\Qp(\mu_{p^\infty}),T)\otimes\cH_L(\Gamma_0^\cyc))}{e_\eta\cdot\cL_{T,\alpha}(\ker\cL_{T,\beta})}.
\]
On applying Corollaries~\ref{cor:Limagepair} and~\ref{cor:imageLlambda}, we may rewrite this  as
\[
\frac{\cH_L}{\{F_1\in \cH_L:(F_1,0)\in F_\eta\}}=\frac{\cH_L}{\{F_1\in \cH_L:F_1(\zeta u^j-1)=0\ \forall 0\le j\le 1,\zeta\in\mu_{p^\infty}\}}=\frac{\cH_L}{(\log_{p,k-1}(1+X))},
\]
as required.
\end{proof}
As before, a two-variable analogue can be obtained on tensoring everything by $\Omega_{\Qp}\Lambda_{\Zp}(U)$.

\subsection{Coleman maps for $\TT_{f,\chi}$}\label{S:semilocal}
Let $\q\in\{\p,\p^c\}$. We write $\mathscr{D}_\q$ for the decomposition group at a fixed prime above $\q$ in $\Gamma$ and choose a set of coset representatives $\sigma_1,\ldots,\sigma_{p^t}$ of $\Gamma/\mathscr{D}_\q$. We may identify $\mathscr{D}_\q$ with $U\times \Gamma_0^\cyc/\Delta$. Note that our choice of embedding  $\iota_p$ implies that  $U$ and $\Gamma_0^\cyc/\Delta$ correspond to subgroups of $\Gamma_{\p^c}$ and $\Gamma_\p$ respectively.
This gives rise to the decomposition
\[
H^1(K_\q,(-)\otimes\Lambda_{\cO_L}(\Gamma)^\iota)=\bigoplus_{i=1}^{p^t}H^1(K_\q,(-)\otimes \Lambda_{\cO_L}(\mathscr{D}_\q)^\iota)\cdot \sigma_i\cong \bigoplus_{v|\q}\HIw(K_{\infty,v},(-)),
\]
where the last direct sum runs through all places $v$ of $K_\infty$ lying above $\q$. Recall from the introduction that $R_{f,\chi^{-1}}^*=R_f^*\otimes \chi^{-1}$ and $\TT_{f,\chi}=R_{f,\chi^{-1}}^*(1-k/2)\otimes\LL_{\cO_L}(\Gamma)^\iota$. We may define the $\#/\flat$-Coleman maps for $\TT_{f,\chi}$ at $\q$ via
\begin{align*}
\col_{\bullet,\q}:H^1(K_\q,\TT_{f,\chi})&\rightarrow \Lambda_{\cO_L}(\Gamma)\\
x=\sum x_i\cdot \sigma_i&\mapsto \sum e_{1}\Tw_{k/2-1}\col_{\bullet,\Qpinf}(x_i\cdot e_{k/2-1})\cdot \sigma_i
\end{align*}
for $\bullet\in\{\#,\flat\}$, where the character $\vt$ in \S\S\ref{S:wach}-\ref{S:2varCol} is chosen to be $\chi^{-1}|_{G_{K_\q}}$, the notation $\cdot e_{k/2-1}$ signifies the natural map
\[
\HIw(\Qpinf(\mu_{p^\infty}),R_{f,\chi^{-1}}^*(1-k/2))\rightarrow\HIw(\Qpinf(\mu_{p^\infty}),R_{f,\chi^{-1}}^*)
\]
 and $\Tw_{k/2-1}$ denotes the twisting map $\sigma\mapsto \chi_0(\sigma)^{k/2-1}\sigma$ for $\sigma\in\Gamma_0^\cyc$.

We may similarly define  twisted Perrin-Riou maps
\begin{equation}\label{eq:semilocalPRmap}
\cL_\q: H^1(K_\q,\TT_{f,\chi})\rightarrow \cH_{L}(\Gamma_\q)\widehat\otimes \Lambda_{\cO_L}(\Gamma_{\q^c})\otimes \Dcris(R_f^*).
\end{equation}
The  Coleman maps then decompose it as
\begin{equation}\label{eq:semidecomp}
\cL_\q=\begin{pmatrix}
v_1&v_2
\end{pmatrix}\Tw_{k/2-1}M_{\log,\q}\begin{pmatrix}
\col_{\#,\q}\\ \col_{\flat,\q}
\end{pmatrix},
\end{equation}
where $M_{\log,\q}$ is the logarithmic matrix obtained from $\Mlog$ on replacing $\gamma_0$ by its image in $\Gamma_\q$.
{
\begin{remark}\label{rk:col0}
Recall that the Deligne representation $W_f$ can be realized over $L_0$. In particular, there exists a 2-dimensional $L_0$-representation $W_{f,0}$ such that $W_{f,0}\otimes_{L_0}L$. Let $R_{f,\chi^{-1},0}$ and $T_{f,\chi,0}$ be the corresponding rank-two $\cO_{L_0}$-lattices inside $R_{f,\chi^{-1}}$ and $T_{f,\chi}$ respectively and write $\TT_{f,\chi,0}=T_{f,\chi,0}\otimes\Lambda_{\cO_{L_0}}(\Gamma)^\iota$. Then we may define $$\col_{\bullet,\q,0}:H^1(K_\q,\TT_{f,\chi,0})\rightarrow \Lambda_{\cO_{L_0}}(\Gamma)$$
using the Wach module of $R_{f,\chi^{-1},0}$. The resulting Coleman map $\col_{\bullet,\q,0}$  extends to $\col_{\bullet,\q}$ after tensoring by $\cO_L$.
\end{remark}
}
Recall from \eqref{eq:diag} that we have chosen a basis  $\{v_\alpha,v_\beta\}$ of $\vp$-eigenvectors allowing us to define (for $\lambda\in\{\alpha,\beta\}$) the morphism
\begin{equation}\label{eqn:PRMapcoordinatewise}
\cL_{\lambda,\q}:H^1(K_\q,\TT_{f,\chi})\lra \cH_{L,\ord_p(\lambda)}(\Gamma_\q)\widehat\otimes \Lambda_{\cO_L}(\Gamma_{\q^c})
\end{equation}
by considering the coordinates of the map (\ref{eq:semilocalPRmap}) with respect to the eigenbasis $\{v_\alpha,v_\beta\}$. We may rewrite \eqref{eq:semidecomp} as
\begin{equation}\label{eq:2vardecomp}
\begin{pmatrix}
\cL_{\alpha,\q}\\ \cL_{\beta,\q}
\end{pmatrix}=Q^{-1}\Tw_{k/2-1}M_{\log,\q}\begin{pmatrix}
\col_{\#,\q}\\ \col_{\flat,\q}
\end{pmatrix}.
\end{equation}
By linearity, we may extend the map $\cL_{\lambda,\q}$ to a map
$$\cL_{\lambda,\q}:\,H^1\left(K_\q,\TT_{f,\chi}\right)\,\widehat{\otimes}_{\LL_{\cO_L}(\Gamma_{\fq})}\, \cH_{L,\ord_p(\lambda)}(\Gamma_\q)\lra \cH_{L,\ord_p(\lambda)}(\Gamma_\q)\widehat\otimes_{\cO_L} \Lambda_{\cO_L}(\Gamma_{\q^c})$$
and the following extension of a result due to Colmez enables us to view it as a map
\begin{equation}\label{eqn:PRMapcoordinatewiseextended}
\cL_{\lambda,\q}:H^1\left(K_\q,\TT_{f,\chi}\,\widehat{\otimes}_{\LL_{\cO_L}(\Gamma_{\fq})}\, \cH_{L,\ord_p(\lambda)}(\Gamma_\q)\right)\lra \cH_{L,\ord_p(\lambda)}(\Gamma_\q)\widehat\otimes_{\cO_L} \Lambda_{\cO_L}(\Gamma_{\q^c})\,.
\end{equation}
\begin{proposition}[Colmez]
\label{prop:Colmezsidentificationofanalyticohomology}
The natural map
$$H^1\left(K_\q,\TT_{f,\chi}\right)\,\widehat{\otimes}_{\LL_{\cO_L}(\Gamma_{\fq})}\, \cH_{L,\ord_p(\lambda)}(\Gamma_\q)\lra H^1\left(K_\q,\TT_{f,\chi}\,\widehat{\otimes}_{\LL_{\cO_L}(\Gamma_{\fq})}\, \cH_{L,\ord_p(\lambda)}(\Gamma_\q)\right)$$
is an isomorphism.
\end{proposition}
\begin{proof}
The analogous statement for Galois representations with coefficients in a finite extension of $\QQ_p$ is proved in \cite[Prop. II.3.1]{colmez98}. We explain briefly how the proof in op. cit. extends to cover the case when the Galois representation in question is the $\LL_{\cO_L}(\Gamma_{\fq})$-module $\TT_{f,\chi}$. All references in this proof is to  \cite{colmez98} unless otherwise stated. 

As in op.cit., let us fix a topological generator $\gamma_\frak{q}$ of $\Gamma_{\fq}$. We set $\nu_n:=(\gamma_{\frak{q}}-1)^n$ and $\LL_{\cO_L}(\Gamma_{\fq})_n:=\LL_{\cO_L}(\Gamma_{\fq})/\nu_n$ for every positive integer $n$.  Set $\LL_{L}(\Gamma_{\fq})_n:=\LL_{L}(\Gamma_{\fq})/\nu_n=\LL_{\cO_L}(\Gamma_{\fq})_n\otimes_{\cO_L}L$. We let $||\cdot||_n$ denote the unique norm on $\LL_{L}(\Gamma_{\fq})_n\widehat\otimes_{\cO_L} \Lambda_{\cO_L}(\Gamma_{\q^c})$ such that $\LL_{\cO_L}(\Gamma_{\fq})_n\widehat\otimes_{\cO_L} \Lambda_{\cO_L}(\Gamma_{\q^c})$ is the unit ball. We let $\pi_n$ denote the natural map $\LL_{\cO_L}(\Gamma_{\fq})\ra \LL_{\cO_L}(\Gamma_{\fq})_n$. Colmez' arguments reduce the proof to the verification of Lemma II.3.2 in the situation when Colmez' $M$ is replaced by $\TT_{f,\chi}$, his rings $\LL_K$ and $\LL_n$ are replaced by $\LL_{\cO_L}(\Gamma_{\fq})$ and $\LL_{\cO_L}(\Gamma_{\fq})_n$, respectively. We indicate below how Colmez' arguments apply verbatim to prove each one of the three parts of the version of this lemma in our setting.

The proof of Lemma II.3.2(i) carries over without a change in the current set up, on observing that  we have $H^2\left(K_\q,\TT_{f,\chi}\right)=0$ thanks to our assumption that $p>k$ (as a consequence of a result due to Fontaine and Edixhoven on the residual irreducibility of $\rho_f$). Translation of Lemma II.3.2(ii) to our setting follows as a formal consequence of Lemma II.3.2(i). 

To conclude with the proof, we now verify Lemma II.3.2(iii) in the current situation. We note that 
$$H^1(K_{\frak{q}},\TT_{f,\chi}\otimes_{\cO_L}L)=H^1(K_{\frak{q}},\TT_{f,\chi})\otimes_{\cO_L}L$$ 
is a free $\LL_{L}(\Gamma_\q)\widehat\otimes_{\cO_L} \Lambda_{\cO_L}(\Gamma_{\q^c})$-module of rank $2$ (still as a consequence of our assumption that $p>k$). Let $||\cdot ||_{{\TT_{f,\chi},n}}$ denote the unique norm on $H^1(K_{\frak{q}},\TT_{f,\chi}/\nu_n\TT_{f,\chi}\otimes_{\cO_L}L)$ such that $H^1(K_{\frak{q}},\TT_{f,\chi}/\nu_n\TT_{f,\chi})$ is the unit ball. Let $\{\mu_1,\mu_2\}$ be a basis of $H^1(K_{\frak{q}},\TT_{f,\chi}\otimes_{\cO_L}L)$. According to Lemma II.3.2(ii) (which holds in our setting as we explained above), it follows that $\{\pi_n(\mu_1),\pi_n(\mu_2)\}$ is a basis of the $\LL_{L}(\Gamma_\q)_n\widehat\otimes_{\cO_L} \Lambda_{\cO_L}(\Gamma_{\q^c})$-module $H^1(K_{\frak{q}},\TT_{f,\chi}/\nu_n\TT_{f,\chi}\otimes_{\cO_L}L)$. Let $\mu\in H^1(K_{\frak{q}},\TT_{f,\chi}/\nu_n\TT_{f,\chi}\otimes_{\cO_L}L)$ be any element and define $\lambda_i\in \LL_{L}(\Gamma_\q)_n\widehat\otimes_{\cO_L} \Lambda_{\cO_L}(\Gamma_{\q^c})$ so that 
$$\mu=\lambda_1\pi_n(\mu_1)+\lambda_2\pi_n(\mu_2).$$ 
Choose a natural number $C>0$ so that the $\LL_{\cO_L}(\Gamma_\q)\widehat\otimes_{\cO_L} \Lambda_{\cO_L}(\Gamma_{\q^c})$-submodule of $H^1(K_{\frak{q}},\TT_{f,\chi}\otimes_{\cO_L}L)$ generated by $\mu_1$ and $\mu_2$ contains $C\cdot H^1(K_{\frak{q}},\TT_{f,\chi})$. Such $C$ clearly exists, as we have
$$H^1(K_{\frak{q}},\TT_{f,\chi}\otimes_{\cO_L}L)=H^1(K_{\frak{q}},\TT_{f,\chi})\otimes_{\cO_L}L.$$ Since $H^2(K_{\frak{q}},\TT_{f,\chi})=0$, it follows that the $\LL_{\cO_L}(\Gamma_\q)_n\widehat\otimes_{\cO_L} \Lambda_{\cO_L}(\Gamma_{\q^c})$-submodule of 
\begin{align*}
H^1(K_{\frak{q}},\TT_{f,\chi}/\nu_n\TT_{f,\chi}\otimes_{\cO_L}L)&=H^1(K_{\frak{q}},\TT_{f,\chi}\otimes_{\cO_L}L)/\nu_n H^1(K_{\frak{q}},\TT_{f,\chi}\otimes_{\cO_L}L)\\
&=\left(H^1(K_{\frak{q}},\TT_{f,\chi})/\nu_n H^1(K_{\frak{q}},\TT_{f,\chi})\right)\otimes_{\cO_L}L
\end{align*}
generated by $\pi_n(\mu_1)$ and $\pi_n(\mu_2)$ contains 
$$C\cdot H^1(K_{\frak{q}},\TT_{f,\chi}/\nu_n\TT_{f,\chi})=C\cdot \left(H^1(K_{\frak{q}},\TT_{f,\chi})/\nu_n H^1(K_{\frak{q}},\TT_{f,\chi})\right).$$ This in turn shows that
$$||\lambda_i||_n\leq C||\mu||_{\TT_{f,\chi},n}$$
as required by Lemma II.3.2(iii) and concludes the proof as in \cite{colmez98}.
\end{proof}
Let $K'/K$ be a finite $p$-extension where all primes above $p$ are unramified. A typical choice for $K^\prime$ in this article will be $K(\mathfrak{n})$ for some ideal $\mathfrak{n}$ of $\cO_K$ that is coprime to $p$. If $v$ is a prime of $K'$ lying above $\q$, then $K'_v$ can be considered as a sub-extension of $\Qpinf$ when we identify $K_\q$ with $\Qp$. Therefore, we may define the maps $\cL_v$, $\cL_{\alpha,v}$, $\cL_{\beta,v}$ and $\col_{\bullet, v}$ on $H^1(K'_v,\TT_{f,\chi})$ using $\cL_{T,K_{v,\infty}'}$ and $\col_{\bullet, K_{v,\infty}'}$, where $K_{v,\infty}'$ is the unramified $\Zp$-extension of $K'_v$, as before. When $K'=K(\mathfrak{n})$, we shall write
\[
H^1(K(\mathfrak{n})_\q,\TT_{f,\chi})=\bigoplus_{v|\q}H^1(K'_v,\TT_{f,\chi}),
\]
on which we have the maps $\cL_{\mathfrak{n},\q}=\oplus \cL_v$, $\cL_{\lambda,\mathfrak{n},\q}=\oplus\cL_{\lambda,v}$ and  $\col_{\bullet,\mathfrak{n},\q}=\oplus \col_{\bullet,v}$, whose images lie inside $\cO_L[\Gal(K'/K)]\otimes \Lambda_{\cO_L}(\Gamma)$ and $\cO_L[\Gal(K'/K)]\otimes \cH_L(\Gamma)$ respectively. 
For $?\in\{\cyc,\ac\}$, we shall write $\col_{\bullet,\n,\q}^?:H^1(K(\mathfrak{n})_\q,\TT_{f,\chi}^?)\lra\cO_L[\textup{Gal}(K(\frak{n})/K)]\otimes \Lambda_{\cO_L}(\Gamma^?)$ and $\cL_{\lambda,\n,\q}^?:H^1(K(\mathfrak{n})_\q,\TT_{f,\chi}^?)\lra\cO_L[\textup{Gal}(K(\frak{n})/K)]\otimes \cH_{L}(\Gamma^?)$ for the natural maps induced from $\col_{\bullet,\n,\q}$ and $\cL_{\lambda,\n,\q}$, respectively. As before, when $\n$ is trivial, we suppress $\n$ from the notation and simply write $\col_{\bullet, \q}$, $\col_{\bullet,\q}^?$ and so on.
Let $\omega$ be the Teichmuller character on $\Delta$. Lemma~\ref{lem:2varimage} tells us that 
\[
\image(\col_{\bullet,\q})\subset \Tw_{1-k/2}(\xi_{\omega^{1-k/2},\bullet}^\q)\Lambda(\Gamma),
\]
where $\xi_{\omega^{1-k/2},\bullet}^\q$ is obtained from $\xi_{\omega^{1-k/2},\bullet}$ on replacing $\Gamma_0^\cyc$ by $\gamma_\q$. We recall that this containment is of finite index. From now on, we write 
\[
\xi_\bullet^\q=\Tw_{1-k/2}(\xi_{\omega^{1-k/2},\bullet})\in\Lambda_{\cO_L}(\Gamma_\q)
\]
and  for $?\in\{\cyc,\ac\}$, we define $\xi_\bullet^?$ similarly. For $?=\emptyset$, we shall define $\xi_\bullet^?$ to be $\xi_\bullet^\fp$.
 
\section{Beilinson-Flach elements and factorization}
\label{sec:BFelementsandfactorization}

\subsection{Review on Beilinson-Flach elements}\label{S:defnBF}
 Let $m,N\ge1$ be integers and $a\in \ZZ/m\ZZ$. For an integer\footnote{The integer $c$ here has nothing to do with the other two objects  denoted by $c$ previsouly. Despite this conflict, we decided to keep this notation from \cite{KLZ2,LZ1} because it will ultimately be eliminated from our discussion; c.f. Remark~\ref{rem:auxintegerc} below.} $c>1$ that is coprime to $6mpN$, let
\[
_c\mathcal{RI}_{m,mN,a}^{[j]}\in H^3_{\text{\'et}}\left(Y(m,mN)^2,\Lambda(\mathscr{H}_{\Zp}\langle t\rangle)^{[j,j]}(2-j)\right),
\]
denote the Rankin-Iwasawa class as defined in \cite[Definition~3.2.1]{LZ1} and \cite[Definition~5.1.5]{KLZ2}. 

Let $f$ be our fixed modular form.  Let $\g$ be a  CM Hida family (in the sense of \cite[\S7]{KLZ2}) that has CM over our fixed imaginary quadratic field $K$. We fix $N$ to be an integer that is divisible by $N_f$ and the  tame level of $\g$. Throughout, we assume that $f$ is non-CM. In particular, $f$ is not a twist of any member of $\g$. Let $0\le j\le k-2$. Consider the image of the Rankin-Iwasawa class $_{c}\mathcal{RI}^{[j]}_{m,mpN,1}$ under the compositum of the morphisms
\begin{align}
\label{eq:imageBF}H^3_{\text{\'et}}&\left(Y(m,mpN)^2,\Lambda(\mathscr{H}_{\Zp}\langle t\rangle)^{[j,j]}(2-j)\right)\\
&\longrightarrow H^1\left(\QQ(\mu_m), H^1_{\text{\'et}}\left(\overline{Y_1(pN)},\Lambda(\mathscr{H}_{\Zp}\langle t\rangle)^{[j]}\right)^{\boxtimes 2}(2-j)\right)\notag\\
&\longrightarrow H^1\left(\QQ(\mu_m), H^1_{\text{\'et}}\left(\overline{Y_1(pN)},{\rm TSym}^{k-2}\mathscr{H}_{\Zp}(1)\right)\boxtimes H^1_{\text{\'et}}\left(\overline{Y_1(pN)},\Lambda(\mathscr{H}_{\Zp}\langle t\rangle)^{[k-2]}(1)\right)(-j)\right)\notag\\
&\longrightarrow H^1\left(\QQ(\mu_m), R_{f^\lambda}^*\otimes H^1_{\ord}(Np^\infty)^{[k-2]}(-j)\right),\notag
\end{align}
where   the first arrow is the \'etale regulator map composed with the K\"unneth decomposition, the second arrow is given by the moment map $\mom^{k-2-j}$ (as defined in \cite[\S2.3]{kings}) on the Iwasawa-theoretic sheave on the first component and $\id\otimes \mom^{k-2-j}$ on the second component, the last arrow is given by the obvious projection and $H^1_{\ord}(Np^\infty)^{[k-2]}$ is defined to be
$$ e_{\ord}'H^1_{\text{\'et}}\left(\overline{Y_1(pN)},\Lambda(\mathscr{H}_{\Zp}\langle t_{mpN}\rangle)^{[k-2]}(1)\right)\cong e_{\ord}'\varprojlim_r H^1_{\text{\'et}}\left(\overline{Y_1(p^rN)},\mathrm{TSym}^{k-2}\mathscr{H}_{\Zp}(1)\right)$$ (c.f. \cite[Theorem~4.5.1 and Proposition~7.2.1(a)]{KLZ2}).
We write
\[
_cz_{m,j}^\lambda\in 
H^1\left(\QQ(\mu_m), R_{f^\lambda}^*\otimes H^1_{\ord}(Np^\infty)^{[k-2]}(-j)\right),
\]
for this image.
We remark the moment map
\[
\id\otimes \mom^{k-2-j}:H^1_{\text{\'et}}\left(\overline{Y_1(p^rN)},\Lambda(\mathscr{H}_{\Zp}\langle t\rangle)^{[j]}(1)\right)\rightarrow H^1_{\text{\'et}}\left(\overline{Y_1(p^rN)},\Lambda(\mathscr{H}_{\Zp}\langle t\rangle)^{[k-2]}(1)\right)
\]
 corresponds to the twisting map
\[
\varprojlim H^1_{\text{\'et}}\left(\overline{Y_1(p^rN)},\TSym^j\mathscr{H}_r(1)\right)
\rightarrow\varprojlim H^1_{\text{\'et}}\left(\overline{Y_1(p^rN)},\TSym^{k-2}\mathscr{H}_r(1)\right)
\]
 given by the cup-product in with $(N\cdot t_{Np^{r}})^{\otimes (k-2-j)}$, which is Hecke-equivariant as explained in \cite[Theorem~4.5.1(2), Remark 4.5.3]{KLZ2}. 

 Let $R_\g^*$ be the representation defined by \cite[Definition~7.2.5]{KLZ2}, which is a quotient of $H^1_\ord(Np^\infty)$ and write $R_\g^{*,[k-2]}$ for the corresponding quotient of $H^1_\ord(Np^\infty)^{[k-2]}$. Then, $R_\g^*$ and $R_\g^{*,[k-2]}$ agree up to a twist by the $(k-2)$-nd power of the weight character. On the level of Iwasawa sheaves, they differ by the moment map $\id\otimes \mom^{k-2}$.
 
Since we have assumed that $f$ is non-CM, we have
\begin{equation}\label{eq:trivial}
H^0\left(\QQ(\mu_{mp^\infty}),R_f^*\otimes R_\g^{*,[k-2]}(-j)\right)=0.
\end{equation}
We may further project $_cz_{m,j}^\lambda$ to
\[
_cz_{m,j}^{\lambda,\g}\in 
H^1\left(\QQ(\mu_m), R_{f^\lambda}^*\otimes R_\g^{*,[k-2]}(-j)\right).
\]
Following the proof of Theorem~6.3.4 of \textit{op. cit.}, for $ n\ge1$ and a fixed $m$ that is coprime to $p$, let 
\[
_cx_{n,j}^{\lambda,\g}\in H^1\left(\QQ(\mu_{mp^n}), R_{f^\lambda}^*\otimes R_\g^{*,[k-2]}\otimes L(-j)\right)
\]
denote the the image of
\[
\frac{(-1)^j\left((U_p')^{-r},(U_p')^{-r}\right)}{j!\binom{k-2}{j}^2} {}_{c}\mathcal{RI}^{[j]}_{mp^n,mp^{n+1}N,1}
\]
under the same series of maps in \eqref{eq:imageBF}.

\begin{remark}
\label{rem:auxintegerc}
As explained in \cite[\S5.3]{LLZ2}, it is possible to dispose of the factor $c$ in our construction of the elements $_cz^{\lambda,\g}_{m,j}$ and $_cx^{\lambda,\g}_{n,j}$. We shall do so from now on and simply write $z^{\lambda,\g}_{m,j}$ and $x^{\lambda,\g}_{n,j}$ for these elements.
\end{remark}
\begin{theorem}\label{thm:BFHida}
Let $m$ be a fixed integer coprime to $p$ and $\lambda\in \{\alpha,\beta\}$. 
 There exists an element 
\[\BF_m^{\lambda,\g}\in H^1\left(\QQ(\mu_{m}),D_{\ord_p(\lambda)}(\Gamma_0^\cyc,R_{f^\lambda}^*)\widehat\otimes R_\g^{*,[k-2]}\right)=H^1\left(\QQ(\mu_{mp^\infty}),D_{\ord_p(\lambda)}(\Gamma_0^\cyc,R_{f^\lambda}^*)\widehat\otimes R_\g^{*,[k-2]}\right)^{\Gamma_0^\cyc}
\]
$($where $D_{\ord_p(\lambda)}(\Gamma_0^\cyc,R_{f^\lambda}^*)$ is the $R_{f^\lambda}^*$-valued distributions of order $\lambda$ as in \cite[\S2.2]{LZ1} and the identification of the two cohomology groups comes from the inflation-restriction exact sequence and \eqref{eq:trivial}$)$ such that 
\[
\int_{(\Gamma_0^\cyc)^{p^n}}\chi_0^{j}\BF_m^{\lambda,\g}=x_{n,j}^{\lambda,\g}
\]
for all $n\ge 1$ and $0\le j\le k-2$, where $x_{n,j}^{\lambda,\g}$ is considered as a class in $H^1\left(\QQ(\mu_{mp^\infty}),R_{f^\lambda}^*\otimes R_\g^{*,[k-2]}\otimes L\right)^{(\Gamma_0^\cyc)^{p^n}=\chi_0^j}.$
\end{theorem}
\begin{proof}
As in the proof of \cite[Theorem~3.5.9]{LZ1}, there is a natural norm $||\cdot||$ on $R_{f^\lambda}^*\otimes R_\g^{*,[k-2]}$, which in turn induces a norm on  $H^1\left(\QQ(\mu_{mp^\infty}),R_{f^\lambda}^*\otimes R_\g^{*,[k-2]}\otimes L\right)$. Theorem~3.3.5 of \textit{op. cit.} tells us that
\[
\left|\left|\sum_{j=0}^{k-2}(-1)^j\binom{k-2}{j}\lambda^n\Tw_j\circ \res\left( x_{r,j}^{\lambda,\g}\right)\right|\right|=O(p^{-(k-2)n}),
\]
where $\res$ is the restriction map from $\QQ(\mu_{mp^n})$ to $\QQ(\mu_{mp^\infty})$ and $\Tw_j$ is the twisting map  $R_f^*\otimes R_\g^{*,[k-2]}(-j)\rightarrow R_f^*\otimes R_\g^{*,[k-2]}$.
This allows us to apply \cite[{Proposition~2.3.3}]{LZ1} to the classes $x_{n,j}^{\lambda,\g}$, which are norm compatible  as $n$ varies thanks to \cite[Theorem~5.4.1]{KLZ2}.
\end{proof}

We shall identify $R_{f}^*$ with $R^*_{f^{\lambda}}$ following  \cite[\S3.5]{LZ1}. More specifically, let $\pr_1$ and $\pr_2$ be the two degeneracy maps on the modular curves $Y_1(pN)\rightarrow Y_1(N)$ as defined in \cite[Definition~2.4.1]{KLZ2} and write  ${\Pr}^{\lambda}=\pr_1-\frac{\lambda'}{p^{k-1}}\pr_2$, where $\lambda'$ denotes the other root of the Hecke polynomial of $f$ at $p$.  Realizing $R_{f^{\lambda}}^*$ and $R_{f}^*$ as quotients of the first \'etale cohomology of $\overline{Y_1(pN)}$ and $\overline{Y(N)}$ respectively, $\Pr^{\lambda}_*$ gives rise to an isomorphism between these two Galois representations. In particular, if we identify $_cz_{m,j}^{\lambda,\g}$ as an element of $H^1\left(\QQ(\mu_m),R_f^*\otimes R_\g^{*,[k-2]}(-j)\right)$ under this identification (which we shall do from now on), it is the image of 
\[
_c\mathcal{RI}_{m,mN,1}^{\lambda,[j]}:=(\mathrm{Pr}^\lambda\times\pr_1)_*\left({}_c\mathcal{RI}_{m,mpN,1}^{[j]}\right)\in H^3_{\text{\'et}}\left(Y(m,mN)^2,\Lambda(\mathscr{H}_{\Zp}\langle t_{mN}\rangle)^{[j,j]}(2-j)\right)
\]
under the same series of maps as in \eqref{eq:imageBF}.
\begin{proposition}\label{prop:BFcomparison}Let  $m$ be an integer coprime to $p$ and  $\theta$ a finite-order character of $\Gal(\QQ(\mu_{mp^\infty})/\QQ)$. Let $p^n$ be the $p$-part of the conductor of $\theta$. If $n>0$, then
\[
\sum_{\sigma\in \Gal(\QQ(\mu_{mp^n})/\QQ)}\theta(\sigma)^{-1}\sigma\cdot {}_c\mathcal{RI}_{mp^n,mp^nN,1}^{\lambda,[j]}=\sum_{\sigma\in \Gal(\QQ(\mu_{mp^n})/\QQ)}\theta(\sigma)^{-1}\sigma\cdot {}_c\mathcal{RI}_{mp^n,mp^nN,1}^{[j]}.
\]
\end{proposition}
\begin{proof}
By definition, 
\[
(\mathrm{Pr}^\lambda\times\pr_1)_*=(\pr_1\times\pr_1)_*-\frac{\lambda'}{p^{k-1}}(\pr_2\times \pr_1)_*.
\]
The term $(\pr_1\times\pr_1)_*$ applied to ${}_c\mathcal{RI}^{[j]}_{mp^n,mp^{n+1}N,1}$ gives ${}_c\mathcal{RI}_{mp^n,mp^nN,1}^{[j]}$ thanks to \cite[Theorem~5.3.1]{KLZ2}. Corollary~5.5.2 in \textit{op. cit.} tells us that if we apply $(\pr_2\times\pr_1)_*$  to the same element, we obtain an expression of the form
\[
(\square){}_c\mathcal{RI}_{mp^n,mp^nN,p}^{[j]}\,,
\]
where $(\square)$ is some explicit factor whose exact value we need not know. We identify the Galois group $\Gal(\QQ(\mu_{mp^n})/\QQ)$ with $(\ZZ/mp^n\ZZ)^\times$. By the definition of Rankin-Iwasawa classes we have ${}_c\mathcal{RI}_{mp^n,mp^nN,x}^{[j]}=\mathcal{RI}_{mp^n,mp^nN,y}^{[j]}$ for $x\equiv y\mod mp^n$, and it follows applying \cite[Proposition~5.2.3]{KLZ2} that
\begin{align*}
\sum_{\sigma\in\Gal(\QQ(\mu_{mp^n})/\QQ)}\theta(\sigma)^{-1}\sigma\cdot{}_c\mathcal{RI}_{mp^n,mp^nN,p}^{[j]}&=\sum_{a\in(\ZZ/mp^n\ZZ)^\times}\theta(a)^{-1}{}_c\mathcal{RI}_{mp^n,mp^nN,a^{-1}p}^{[j]}\\
&=\sum_{b\in(\ZZ/mp^{n-1}\ZZ)^\times}{}_c\mathcal{RI}_{mp^n,mp^nN,b^{-1}p}^{[j]}\sum_{\substack{a\in (\ZZ/mp^{n}\ZZ)^\times\\a\equiv b \textup{ mod } mp^{n-1}}}\theta(a)^{-1}=0.
\end{align*}
The proof follows.
\end{proof}
\begin{corollary}\label{cor:samething}
Let  $m$ be an integer coprime to $p$ and  $\theta$ a finite-order character of $\Gal(\QQ(\mu_{mp^\infty})/\QQ)$. Let $p^n$ be the $p$-part of the conductor of $\theta$. If $n>0$, then
\[
\sum_{\sigma\in \Gal(\QQ(\mu_{mp^n})/\QQ)}\theta(\sigma)^{-1}\sigma\cdot z_{n,j}^{\alpha,\g}=\sum_{\sigma\in \Gal(\QQ(\mu_{mp^n})/\QQ)}\theta(\sigma)^{-1}\sigma\cdot z_{n,j}^{\beta,\g}
\]
for all $0\le j\le k-2$.
\end{corollary}

\begin{lemma}\label{lem:resimage}
Let $m$ be an integer coprime to $p$, $n\ge0$ an integer and $z\in H^1(\QQ(\mu_{mp^{n+1}}), R_f^*\otimes R_\g^{*,[k-2]})$. Suppose that 
for all characters $\theta$ of $\Gal(\QQ(\mu_{mp^{n+1}})/\QQ)$ whose $p$-conductor is  $p^{n+1}$, we have
\[
\sum_{\sigma\in \Gal(\QQ(\mu_{mp^{n+1}})/\QQ)}\theta(\sigma)^{-1}\sigma\cdot z=0.
\] Then, $z$ is in the  image of $H^1\left(\QQ(\mu_{mp^{n}}), R_f^*\otimes R_\g^{*,[k-2]}\right)$ in $H^1\left(\QQ(\mu_{mp^{n+1}}), R_f^*\otimes R_\g^{*,[k-2]}\right)$ under restriction.
\end{lemma}
\begin{proof}
Our hypothesis on $z$ tells us that 
\[
z\in H^1\left(\QQ(\mu_{mp^{n+1}}), R_f^*\otimes R_\g^{*,[k-2]}\right)^{\Gal(\QQ(\mu_{mp^{n+1}})/\QQ(\mu_{mp^{n}}))}.
\]
The latter coincides with the image of the restriction of $H^1\left(\QQ(\mu_{mp^{n}}), R_f^*\otimes R_\g^{*,[k-2]}\right)$ via the inflation-restriction sequence thanks to \eqref{eq:trivial}.
\end{proof}
\begin{corollary}\label{cor:divideclasses}
Let $m$ be an integer coprime to $p$, $n\ge1$ an integer and $$z\in H^1\left(\QQ(\mu_m), R_{f}^*\otimes R_\g^{*,[k-2]}\otimes \cO_L[\Gamma_0^\cyc/(\Gamma_0^\cyc)^{p^n}]^\iota\right).$$ Suppose that  for all characters $\theta$ of $\Gal(\QQ(\mu_{mp^{n+1}})/\QQ)$ whose $p$-conductor is  $p^{n+1}$, we have
\[
\sum_{\sigma\in \Gal(\QQ(\mu_{mp^{n+1}})/\QQ)}\theta(\sigma)^{-1}\sigma\cdot z=0.
\]
 Then, $z$ is in the  image of $H^1\left(\QQ(\mu_{m}), R_f^*\otimes R_\g^{*,[k-2]}\otimes \Phi_n(\gamma_0)\cO_L[\Gamma_0^\cyc/(\Gamma_0^\cyc)^{p^{n}}]^\iota\right)$, where $\Phi_n$ denotes the $p^n$-th cyclotomic polynomial.
\end{corollary}
\begin{proof}
Shapiro's lemma gives 
 $$H^1\left(\QQ(\mu_m), R_{f}^*\otimes R_\g^{*,[k-2]}\otimes \cO_L[\Gamma_0^\cyc/(\Gamma_0^\cyc)^{p^n}]^\iota\right)\cong H^1\left(\QQ(\mu_{mp^{n+1}}),R_{f}^*\otimes R_\g^{*,[k-2]}\right).$$
 The restriction map $H^1\left(\QQ(\mu_{mp^{n}}),R_{f}^*\otimes R_\g^{*,[k-2]}\right)\rightarrow H^1\left(\QQ(\mu_{mp^{n+1}}),R_{f}^*\otimes R_\g^{*,[k-2]}\right)$ corresponds to the morphism
\[
H^1\left(\QQ(\mu_{m}), R_f^*\otimes R_\g^{*,[k-2]}\otimes\cO_L[\Gamma_0^\cyc/(\Gamma_0^\cyc)^{p^{n-1}}]^\iota\right)\rightarrow H^1\left(\QQ(\mu_{m}), R_f^*\otimes R_\g^{*,[k-2]}\otimes\cO_L[\Gamma_0^\cyc/(\Gamma_0^\cyc)^{p^{n}}]^\iota\right)
\] 
induced by the multiplication  by $\Phi_n(\gamma_0)$. Hence, our result follows from Lemma~\ref{lem:resimage}.
\end{proof}

\begin{theorem}\label{thm:decompBFhida}
There exist two elements $\BF_m^{\#,\g}$ and $\BF_m^{\flat,\g}$ inside
$$H^1\left(\QQ(\mu_m),R_f^*\otimes R_\g^{*,[k-2]}\widehat\otimes \Lambda_{L}(\Gamma_0^\cyc)^\iota\right)=H^1\left(\QQ(\mu_m),  D_0(\Gamma_0^\cyc,R_f^*)\widehat\otimes R_\g^{*,[k-2]}\right)$$ 
such that
\[
\begin{pmatrix}
\BF_m^{\alpha,\g}\\ \BF_m^{\beta,\g}
\end{pmatrix}
=Q^{-1}\Mlog\cdot
\begin{pmatrix}
\BF_m^{\#,\g}\\ \BF_m^{\flat,\g}
\end{pmatrix},
\]
where we identify all $\BF_m^{?,\g}$  as elements of $H^1\left(\QQ(\mu_m),D_{k-1}(\Gamma_0^\cyc,R_f^*)\widehat\otimes R_\g^{*,[k-2]}\right)$.
Furthermore, there exists an integer $s$ such that 
\[
 \BF_m^{\#,\g},\BF_m^{\flat,\g}\in \varpi^{-s}H^1\left(\QQ(\mu_m),R_f^*\otimes R_\g^{*,[k-2]}\widehat\otimes \Lambda_{\cO_L}(\Gamma_0^\cyc)^\iota\right)
\]
for all $m$ and all $\g$.
\end{theorem}
\begin{proof}
 Let $0\le j\le k-2$, $n\ge1$ be integers and $\lambda\in\{\alpha,\beta\}$. 
 Consider the commutative diagram
 \[
 \xymatrix{
H^1\left(\QQ(\mu_m),D_{\ord_p(\lambda)}(\Gamma_0^\cyc,R_{f}^*)\widehat\otimes R_\g^{*,[k-2]}\right) \ar[d]\ar[r]^{\rm res\ \ } &H^1\left(\QQ(\mu_{mp^\infty}),D_{\ord_p(\lambda)}(\Gamma_0^\cyc,R_{f}^*)\widehat\otimes R_\g^{*,[k-2]}\right)^{\Gamma_0^\cyc} \ar[dd]^{\int_{(\Gamma_0^\cyc)^{p^{n-1}}}\chi_0^j}\\
 H^1\left(\QQ(\mu_m), R_{f}^*\otimes R_\g^{*,[k-2]}(-j)\otimes L[\Gamma_0^\cyc/(\Gamma_0^\cyc)^{p^n}]^\iota\right)\ar[d]& \\
 H^1\left(\QQ(\mu_{mp^n}), R_{f}^*\otimes R_\g^{*,[k-2]}(-j)\otimes L\right)\ar[r]^{\rm \Tw_j\circ res\ \ }&H^1\left(\QQ(\mu_{mp^\infty}),R_f^*\otimes R_\g^{*,[k-2]}\otimes L\right)^{(\Gamma_0^\cyc)^{p^{n-1}}=\chi_0^j}},
 \]
 where the first vertical map on the left is the natural projection given by modulo $(\Gamma_0^\cyc)^{p^{n-1}}=\chi_0^j$ (which is our shorthand for modulo the ideal generated by $\chi_0^j(\gamma_0^{p^{n-1}})-\gamma_0^{p^{n-1}}$), whereas the second vertical map sends the cocycle $\left(\sigma\mapsto \sum_{g\in \Gamma_0^\cyc/(\Gamma_0^\cyc)^{p^{n-1}}} c_g(\sigma)g\right)$ to $\left(\sigma\mapsto c_1(\sigma)\right)$ as given by Shapiro's Lemma. The map $\Tw_j$ is the twisting map $R_f^*\otimes R_\g^{*,[k-2]}(-j)\rightarrow R_f^*\otimes R_\g^{*,[k-2]}$.

 Given that $j!$ and $\binom{k-2}{j}$ are $p$-adic units under our running assumption that $p>k$, we have
\[
\lambda^nx_{n,j}^{\lambda,\g}\in H^1\left(\QQ(\mu_{mp^n}), R_{f^\lambda}^*\otimes R_\g^{*,[k-2]}(-j)\right)\cong H^1\left(\QQ(\mu_m), R_{f^\lambda}^*\otimes R_\g^{*,[k-2]}(-j)\otimes \cO_L[\Gamma_0^\cyc/(\Gamma_0^\cyc)^{p^{n-1}}]^\iota\right).
\]
As in the proof of Theorem~\ref{thm:BFHida}, \cite[Theorem~3.3.5]{LZ1} tells us that, as elements of $H^1\left(\QQ(\mu_{mp^\infty}),R_f^*\otimes R_\g^{*,[k-2]}\otimes L\right)$, we have 
\[
\sum_{i=0}^j(-1)^i\binom{j}{i} \Tw_j\circ \res\left(\lambda^nx_{n,i}^{\lambda,\g}\right)\in p^{jn}H^1\left(\QQ(\mu_{mp^\infty}),  R^*_f\otimes R_\g^{*,[k-2]}\right).
\] 
Given that $\binom{j}{i}$ is a $p$-adic unit, we may recursively lift $\res \left(\lambda^nx_{n,i}^{\lambda,\g}\right)$ to a cocycle 
$$c_{n,i}^{\lambda,\g}\in Z^1\left(G_{\QQ(\mu _{mp^\infty})}, R_f^*\otimes R^{*,[k-2]}_\g\right)\otimes \cO_L[\Gamma_0^\cyc/(\Gamma_0^\cyc)^{p^{n-1}}]$$
as $i$ runs through $0,1,\ldots,k-2$ such that
\begin{equation}\label{eq:boundcocylces}
\left|\left|p^{-jn}\sum_{i=0}^j(-1)^i\binom{j}{i} c_{n,i}^{\lambda,\g}(u^{-i}\gamma_0)\right|\right|\le 1.
\end{equation}

Let $c_n^{\lambda,\g}$ be the unique cocycle 
$$
c_n^{\lambda,\g}\in Z^1\left(G_{\QQ(\mu_{m p^\infty})}, R_f^*\otimes R^{*,[k-2]}_\g\right)\otimes \Lambda_{\cO_L}(\Gamma_0^\cyc)/\omega_{n-1,k-1}(\gamma_0)\otimes L
$$
satisfying
\[
c_n^{\lambda,\g}\equiv c_{n,i}^{\lambda,\g}(u^{-i}\gamma_0)\mod \omega_{n-1}(u^{-i}\gamma_0)
\]
for $i=0,\ldots, k-2$ (it exists, by Chinese remainder theorem, as $G_{\QQ(\mu_{mp^\infty})}$ acts trivially on $\Lambda_{\cO_L}(\Gamma_0^\cyc)$).
On taking $r=0$ in Remark~\ref{rk:integral}, \eqref{eq:boundcocylces} tells us that there exists an integer $s_0$, independent of $n$, $\g$ and $\lambda$, such that
$$
\varpi^{s_0}c_n^{\lambda,\g}\in Z^1\left(G_{\QQ(\mu_{m p^\infty})}, R_f^*\otimes R^{*,[k-2]}_\g\right)\otimes \Lambda_{\cO_L}(\Gamma_0^\cyc)/\omega_{n-1,k-1}(\gamma_0).
$$
 Hence, this gives
$$\varpi^{s_0}\lambda^nx_{n}^{\lambda,\g}\in H^1\left(\QQ(\mu_{mp^\infty}),R_f^*\otimes R_\g^{*,[k-2]}\otimes  \cO_L[\Gamma_0^\cyc/(\Gamma_0^\cyc)^{p^{n-1}}]\right),$$
satisfying
\[
\lambda^nx_{n}^{\lambda,\g}\equiv\Tw_j\circ\res\left( \lambda^nx_{n,i}^{\lambda,\g}\right)\mod \omega_{n-1}(u^{-i}\gamma_0).
\]
Let $\adj$  be the adjugate matrix of $Q^{-1}C_{n-1}$ as in the proof of Proposition~\ref{prop:prefactor}. Then, Corollaries~\ref{cor:samething} and \ref{cor:divideclasses} tell us that
 \[
 \varpi^{s_0}\adj\begin{pmatrix}
 \alpha^nx_{n}^{\alpha,\g}\\
 \beta^nx_{n}^{\beta,\g}
 \end{pmatrix}\in H^1\left(\QQ(\mu_{mp^\infty}),R_f^*\otimes R_\g^{*,[k-2]}\otimes \prod_{i=1}^{n-1}\Phi_{i,k-1}(\gamma_0)  \Lambda_{\cO_L}(\Gamma_0^\cyc)^{\oplus 2}/\omega_{n-1,k-1}(\gamma_0)\right).
 \]
 Consequently, this defines 
 \[
 \begin{pmatrix}
\varpi^{s_0} x_{\#,n}\\ \varpi^{s_0}x_{\flat,n}
 \end{pmatrix}\in  H^1\left(\QQ(\mu_{mp^\infty}),R_f^*\otimes R_\g^{*,[k-2]}\otimes \varpi^{-s} \Lambda_{\cO_L}[[\Gamma_0^\cyc]]^{\oplus2}/\ker h_{n-1}\right)
 \]
 satisfying
 \[
 \begin{pmatrix}
 \alpha^nx_{n}^{\alpha,\g}\\
 \beta^nx_{n}^{\beta,\g}
 \end{pmatrix}=Q^{-1}C_{r-1}
 \begin{pmatrix}
 x_{\#,n}\\x_{\flat,n}
 \end{pmatrix}.
 \]
 On letting $n\rightarrow\infty$, we obtain two classes $x_\#$ and $x_\flat$ inside $H^1\left(\QQ(\mu_{mp^\infty}),R_f^*\otimes R_\g^{*,[k-2]}\otimes \varpi^{-s} \Lambda_{\cO_L}(\Gamma_0^\cyc)^\iota\right)$ satisfying
 \[
 \begin{pmatrix}
 \BF_m^{\alpha,\g}\\
 \BF_m^{\beta,\g}
 \end{pmatrix}=Q^{-1}\Mlog
 \begin{pmatrix}
 x_{\#}\\x_{\flat}
 \end{pmatrix}.
 \]
 These classes are invariant under $\Gamma_0^\cyc$ since $\BF_m^{\lambda,\g}$ are. Hence, we obtain the classes $\BF_m^{\#,\g}$ and $\BF_m^{\flat,\g}$ as required.
\end{proof}

{
\begin{remark}\label{rk:three}
Since the classes $_c\mathcal{RI}_{m,mN,a}^{[j]}$ (before $p$-stabilization) can be realized over $L_0$, Remark~\ref{rk:two} tells us that the classes $\BF_m^{\#,\g},\BF_m^{\flat,\g}$ can also be realized over $L_0$. In other words, they live in $$\varpi_0^{-s}H^1\left(\QQ(\mu_m),R_{f,0}^*\otimes R_\g^{*,[k-2]}\widehat\otimes \Lambda_{\cO_{L_0}}(\Gamma_0^\cyc)^\iota\right),$$ where $\varpi_0$ is a uniformizer of $L_0$ and  $R_{f,0}$ is a free $\cO_{L_0}$-module of rank two inside $R_{f}$.
\end{remark}
}

\begin{remark}
\label{rem:alternativewaystofactor}
We have constructed the signed Beilinson-Flach elements $\BF_m^{\#,\g}$ and $\BF_m^{\flat,\g}$ via those classes that land in the module $\varpi^{-s} \Lambda_{\cO_L}[[\Gamma_0^\cyc]]^{\oplus 2}/\ker h_{r-1}$, which are patched together via Lemma~\ref{lem:inverselimit} as $r$ varies. The advantage of this approach (namely, working at finite layers of the cyclotomic $\ZZ_p$-tower) is that we only needed to lift finitely many cohomology classes to cocycles at each finite level. 
If a generalization of \cite[Proposition~II.3.1]{colmez98} and \cite[Proposition~2.4.5]{LZ1} to the cohomology group $H^1\left(\QQ(\mu_m),D_{\ord_p(\lambda)}(\Gamma_0^\cyc,R_{f}^*)\widehat\otimes R_\g^{*,[k-2]}\right)$  was available, it would allow us to "pull out"  $\cH_{L,\ord_p(\lambda)}(\Gamma_0^\cyc)$ to conclude that
$$H^1\left(\QQ(\mu_m),D_{\ord_p(\lambda)}(\Gamma_0^\cyc,R_{f}^*)\widehat\otimes R_\g^{*,[k-2]}\right)\stackrel{?}{\cong} H^1\left(\QQ(\mu_m),R_{f}^*\widehat\otimes R_\g^{*,[k-2]}\right)\otimes_{\LL_{cO_L}(\Gamma_0^\cyc)}\, \cH_{L,\ord_p(\lambda)}(\Gamma_0^\cyc).$$
It would then be possible to obtain the factorization in Theorem~\ref{thm:decompBFhida} in an alternative manner. Namely, we could carry out the sought after factorization by working with the coordinates with respective a fixed basis of $H^1(\QQ(\mu_m),R_f^*\otimes \Lambda_{L}(\Gamma_0^\cyc)^\iota\hat\otimes R_\g^{*,[k-2]} )$.
\end{remark}
\subsection{CM Hida families}
\label{subsec:CMhidafamilies}
We follow the notation of \cite[\S\S3-5]{LLZ2} unless we caution readers otherwise. All references in this section are to this article. In particular, for a general modulus $\frak{n}$ of $K$, we let $H_{\frak{n}}$ denote the ray class group modulo $\frak{n}$ and let $K(\frak{n})$ denote the maximal $p$-extension contained in the ray class field modulo $\frak{n}$. We set $H_{\frak{n}}^{(p)}:=\Gal(K(\frak{n})/K)$. 
We recall from the introduction that we have fixed an integral ideal $\frak{f}$ of $\cO_K$ that is prime to $p$ with $K(\frak{f})=K$. We have  also fixed a $p$-distinguished ray class character $\chi$  modulo $\frak{f}$ (in the sense that {$\chi(\p)\neq \chi(\p^c)$; note in particular that $\Ind_{K/\QQ} \overline{\chi}$ is absolutely irreducible, where $\overline{\chi}$ is the mod $p$ reduction of $\chi$)}; we assume that the local field $L$ is large enough to realize the character $\chi$. We caution readers that our field $L$ is denoted by $L_{\frak{P}}$ in \textit{op. cit.} and their $\chi$ is different from ours. 
For any ideal $\frak{m}=\n\ff$ divisible by $\frak{f}$, we will regard $\chi$ as a character of $\varprojlim H_{\frak{m}\fp^r}$ via the surjection $\varprojlim H_{\frak{m}\fp^r} \twoheadrightarrow H_\frak{f}$. We let $\eta_\chi$ denote the Dirichlet character of conductor $N_\chi:=\NN\ff$, which is characterized by the requirement that $\eta_\chi(n)=\chi((n))$ for integers prime to $\ff$. Notice that when $\chi$ is a ring class character, $\eta_\chi=\mathbbm{1}$.
   Let $\psi_0$ denote the unique Hecke character of $\infty$-type $(-1,0)$, conductor $\p$ and whose associated $p$-adic Galois character factors through $\Gamma_\p$. (It is unique since the ratio of two such characters would have finite $p$-power order and conductor dividing $\p$. This has to be the trivial character because of our assumption on the class number.) Set $\psi=\chi\psi_0$ and let $\psi_L:G_K\lra L^\times$ denote the Galois character associated to the $p$-adic avatar of $\psi$, which is obtained via the geometrically normalized Artin map. Note that $\psi$ is a Hecke character of of $\infty$-type $(-1,0)$ with conductor $\ff\p$. The theta-series $$\Theta(\psi):=\sum_{(\frak{a},\ff \p)=1}\psi(\frak{a})q^{\NN\frak{a}}\in S_2(\Gamma_1(D_KN_\chi p),\eta_\chi\epsilon_K\omega^{-1})$$
  is the weight two specialization (with trivial wild character) of a CM Hida family with tame level $D_KN_\chi$ and character $\eta_\chi\epsilon_K\omega$. The weight one specialization of this CM Hida family with trivial wild character equals the $p$-ordinary theta-series $\Theta^{\ord}(\chi):=\sum_{(\frak{a},\ff \p)=1}\chi(\frak{a})q^{\NN\frak{a}}\in S_{1}(\Gamma_1(D_KN_\chi p),\eta_\chi\epsilon_K)$ of $\chi$.

 For each ideal $\m$ divisible by $\frak{f}$, we let $\LL_{\frak{m}}^{(p)}=\cO_L[H_{\frak{m}}^{(p)}]$ (this ring is denoted by $\LL_{\frak{m}}^{\frak{P}}$ in op. cit.). If in addition $\frak{m}$ is coprime to $\fp^c$, we let $H^1(\psi,\frak{m})$ be the $G_{\QQ}$-representation that is denoted by $H^1(\psi,\frak{m},\frak{P})$ in \cite[Definition~5.2.2]{LLZ2}. This Galois representation is cut out from the \'etale cohomology of a suitable modular curve and it is free of rank $2$ over $\LL_\frak{m}^{(p)}$ (c.f. Proposition 5.2.3 in \textit{op. cit.} and \S~\ref{S:ESimag}) below. 
\begin{remark}
It is assumed at the start of \cite[\S 5.1]{LLZ2} that $\psi$ is a an algebraic Hecke character whose conductor is prime to $p$. However, as we are assuming $p$ splits in $K$ and that $\chi$ is $p$-distinguished, Remark 5.1.3 of \textit{op. cit.} still applies and tells us that the same results hold true in our setting.
\end{remark}
The following statement is a very slight extension\footnote{There is a typo in the statement of Corollary 5.2.6: In order to apply the previous results in \S5, the set of ideals considered should be the ones coprime to $\fp^c$, not $\fp$.} of Corollary 5.2.6 via Proposition 5.2.5 of \textit{op. cit.}: 
\begin{proposition}
\label{prop:LLZ2Cor526enhanced}
There exists a family of isomorphisms
$$\nu_{\frak{m},r}\,:\,H^1(\psi,\frak{mp}^r)\stackrel{\sim}{\lra} \textup{Ind}_{K(\frak{mp}^r)}^\QQ\,\cO_L(\psi^{-1}_L)$$
of  $\LL_{\frak{mp}^r}^{(p)}[[G_\QQ]]$-modules such that the diagram
$$\xymatrix{
H^1(\psi,\frak{m}^\prime\frak{p}^r)\ar[d]_{\mathcal{N}^{\frak{m}^\prime\fp^r}_{\frak{m}\fp^s}}\ar[r]^(.42){\nu_{\frak{m}^\prime,r}}_(0.42)\sim& \textup{Ind}_{K(\frak{m}^\prime\frak{p}^r)}^\QQ\,\cO_L(\psi^{-1}_L)\ar[d]\\
H^1(\psi,\frak{m}\frak{p}^s)\ar[r]^(.42){\nu_{\frak{m},s}}_(0.42)\sim& \textup{Ind}_{K(\frak{m}\frak{p}^s)}^\QQ\,\cO_L(\psi^{-1}_L)
}$$
commutes as $\frak{m}\mid \frak{m}^\prime$ range over integral ideals of $\cO_K$ that are divisible by $\frak{f}$ and coprime to $p$; and $r\geq s$ over non-negative integers. 
\end{proposition}
The vertical map on the right is the obvious map induced from $\Gal(K(\frak{m}^\prime\fp^r)/K)\ra \Gal(K(\frak{m}\fp^s)/K)$ and $\mathcal{N}^{\frak{m}^\prime\fp^r}_{\frak{m}\fp^s}$ is the norm map which is given (together with its fundamental property) in Proposition 5.2.5 of \textit{op. cit.} 
\begin{corollary}
\label{cor:nun}
If $h\ge0$ is a real number,  $\frak{m}$ is an ideal divisible by $\frak{f}$ and coprime to $p$, there exists a family of isomorphisms
$$\nu_{\frak{m}}:H^1\left(\QQ, R_f^*\otimes H^1(\psi,\frak{mp}^r)\widehat\otimes\cH_{L,h}(\Gamma_{0}^\cyc)^\iota\right)\stackrel{\sim}{\lra}  H^1\left(K(\frak{mp}^r), R_f^*(\psi_L^{-1})\widehat\otimes\cH_{L,h}(\Gamma_{0}^\cyc)^\iota\right),$$
which are compatible as $\frak{m}$ varies (in the sense that they induce a commutative diagram, analogous to that in Proposition~\ref{prop:LLZ2Cor526enhanced}). 
\end{corollary}
\begin{proof}
This is a consequence of Shapiro's Lemma and Proposition~\ref{prop:LLZ2Cor526enhanced}, by passing to limit in $r$.
\end{proof}

\begin{corollary}
\label{cor:topropLLZ2Cor526enhanced}Let $h$ and $\n$ be as in Corollary~\ref{cor:nun}.
There exists a family of morphisms
$$\mu_{\frak{m}}:\varprojlim_r H^1\left(\QQ, R_f^*\otimes H^1(\psi,\frak{mp}^r)\widehat\otimes\cH_{L,h}(\Gamma_{0}^\cyc)^\iota\right){\lra}  H^1\left(K(\frak{m}), \TT_{f,\chi}\widehat\otimes_{\Lambda_{\cO_L}(\Gamma_{\fp^c})}\cH_{L,h}(\Gamma_{0}^\cyc)^\iota\right),$$
which are compatible as $\frak{m}$ varies (in the obvious sense). 
\end{corollary}
\begin{proof}
The maps $\mu_\frak{m}$ are obtained by composing ${\nu}_\m$ from Corollary~\ref{cor:nun} with the compositum of the arrows
\begin{align*}
\varprojlim_r H^1\left(K(\frak{mp}^r), R_f^*(\psi_L^{-1})\widehat\otimes\cH_{L,h}(\Gamma_{0}^\cyc)^\iota\right) &\stackrel{\textup{cor}}{\lra}\varprojlim_r H^1\left(K(\frak{m})K({\fp}^r), R_f^*(\psi_L^{-1})\widehat\otimes\cH_{L,h}(\Gamma_{0}^\cyc)^\iota\right)\\
&\lra H^1\left(K(\frak{m}), R_f^*(\psi_L^{-1})\otimes\,{\LL_{\cO_L}(\Gamma_{\p})}\,\widehat\otimes_{\cO_L}\cH_{L,h}(\Gamma_{0}^\cyc)^\iota\right)
\\&\lra H^1\left(K(\frak{m}), R_f^*(\chi^{-1})\otimes\,{\LL_{\cO_L}(\Gamma_{\p})}\,\widehat\otimes_{\cO_L}\cH_{L,h}(\Gamma_{0}^\cyc)^\iota\right)\\
&\stackrel{\cdot e_{1-k/2}}{\lra}H^1\left(K(\frak{m}), \TT_{f,\chi}\widehat\otimes_{\LL_{\cO_L}(\Gamma_
\cyc)}\, \cH_{L,h}(\Gamma_{0}^\cyc)^\iota\right)
\end{align*}
where $\textup{cor}=\textup{cor}_{K(\frak{m})K({\fp}^r)}^{K(\frak{mp}^r)}$ is the corestriction map, the second arrow is deduced from Shapiro's lemma, the third arrow is induced from the fact that $\chi^{-1}\psi_L$ factors through $\Gamma$. The compositum of these arrows has the desired compatibility as $\frak{m}$ varies since each of these arrows does.
\end{proof}
The family of morphisms $\mu_{\frak{m}}$ may be rewritten as 
\begin{equation}
\label{eqn:compatiblemorphismsmufrakn}
\mu_{\frak{m}}: H^1\left(\QQ, R_f^*\otimes H^1(\psi,\frak{mp}^\infty)\widehat\otimes\cH_{L,h}^\iota\right){\lra}  H^1\left(K(\frak{m}), \TT_{f,\chi}\widehat\otimes\cH_{L,h}(\Gamma_{0}^\cyc)^\iota\right)
\end{equation}
where we have set $H^1(\psi,\frak{mp}^\infty):=\varprojlim_r H^1(\psi,\frak{mp}^r)$.

\subsection{Beilinson-Flach elements over imaginary quadratic fields}\label{S:ESimag}

Let $K$, $\chi$ and $\ff$ be as before. 
Recall from the introduction that $\mathcal{N}$ is the collection of square-free integral ideals of $\cO_K$ which are prime to $\ff p$. For each $\mathfrak{n}\in \mathcal{N}$, we set $\m=\mathfrak{nf}$. Set $M:=D_K\NN\frak{m}$ and let ${\bf{T}}_{Mp}$ denote the Hecke algebra given as at the start of \cite[\S4.1]{LLZ2} and define the maximal ideal $\mathcal{I}_{\frak{mp}}$ of ${\bf{T}}_{Mp}$ as in Definition 5.1.1 of \textit{op.cit.} We remark that in order to determine the map $\phi_{\frak{m}\p}$ that appears in this definition, we use the algebraic Hecke character $\psi$ we have chosen in \S\ref{subsec:CMhidafamilies} above. It follows by Prop. 5.1.2 that $\mathcal{I}_{\frak{mp}}$ is non-Eisenstein, $p$-ordinary and $p$-distinguished. By Theorem~4.3.4 of \textit{op. cit.}, the ideal $\mathcal{I}_{\frak{mp}}$ corresponds uniquely to a non-Eisenstein maximal ideal $\mathcal{I}$ of the universal ordinary Hecke algebra ${\bf{T}}_{Mp^\infty}$ acting on $H^1_{\ord}(Y_1(Mp^\infty))$ (definitions of these objects may be found in Definition 4.3.1 of loc. cit.). The said correspondence is induced from Ohta's control theorem~\cite[Theorem 1.5.7(iii)]{ohta99}, which also attaches to $\mathcal{I}_{\m\p}$ a unique non-Eisenstein, $p$-ordinary and $p$-distinguished 
maximal ideal $\mathcal{I}_{\m\p^r}$ of ${\bf{T}}_{Mp^r}$ for each $r\geq 1$ (which is easily seen to coincide with the kernel of the compositum of the arrows  ${\bf{T}}_{Mp^r}\stackrel{\phi_{\m\p^r}}{\lra}\cO_L[H_{\m\p^r}]\ra\cO_L\ra \cO_\varpi$, and therefore with its original form given in \cite[Definition 5.1.1]{LLZ2}).   The ideal $\mathcal{I}$ determines a CM Hida family that we shall denote by $\g_{\frak{m}}$, whose associated Galois representation $R_{\g_{\frak{m}}}^*$ is $H^1_{\ord}(Y_1(Mp^\infty))_{\mathcal{I}}$. When $\m=\ff$, note that $\g_\ff$ is the CM Hida family of tame level $D_KN_\chi$ and character $\eta_\chi\epsilon_K\omega$ we have discussed in Section~\ref{subsec:CMhidafamilies}.
 
 We now choose $\Bg$ in the construction of the class $\BF_{1}^{\lambda,\Bg}$ in Theorem~\ref{thm:BFHida} as the CM family $\Bg_\m$ given as above. Notice than that once we have fixed the Hecke character $\psi$ as above, there exists a family of morphisms of Galois modules $\pi_{\frak{m}}: R_{\g_{\m}}^*\ra H^1(\psi,\m\p^\infty)$ that are compatible as $\m$ varies. The construction of these morphisms is based on the arguments already present in \cite[\S5.1]{LLZ2} and we provide a brief outline here. The maps $\phi_{\m\p^r}$ induce morphisms $({\bf{T}}_{Mp^r})_{\mathcal{I}_{\m\p^r}}\stackrel{\phi_{\m\p^r}}{\lra}\cO_L[H_{\m\p^r}^{(p)}]$, which are compatible as $r$ varies (thanks to the choice of the $p$-ordinary maximal ideals ${\mathcal{I}_{\m\p^r}}$) and gives rise in the limit to a map
 $$\phi_{\m,\p^\infty}:\,({\bf{T}}_{Mp^\infty})_{\mathcal{I}}\lra \varprojlim_r \cO_L[H_{\m\p^r}^{(p)}]=:\LL_{\m\p^\infty}\,.$$
 On the other hand, the $G_\QQ$-representations $H^1(\psi,\m\p^r)$ are defined by setting 
 $$H^1(\psi,\m\p^r):= H^1(Y_1(Mp^r))_{\mathcal{I}_{\m\p^r}}\,\otimes_{\phi_{\m\p^r}}\cO_L[H_{\m\p^r}^{(p)}]\,.$$
 On passing to limit in $r$, we have 
 $$H^1(\psi,\m\p^\infty)=H^1_\ord(Y_1(M))_\mathcal{I}\otimes_{\phi_{\m\p^\infty}}\LL_{\m\p^\infty}$$ 
 and the the map $\pi_\m$ is given by $r\mapsto r\otimes 1$ under these identifications.
 
 After twisting by the weight character, the  elements $\BF_{1}^{\lambda,\Bg_{\m}}$  give rise to classes
\[
\BF^\lambda_{\m}\in  H^1(K(\m),\TT_{f,\chi}\widehat\otimes\cH_{L,\ord_p(\lambda)}(\Gamma^\cyc)^\iota),
\]
for each $\m$ and $\lambda\in\{\alpha,\beta\}$,  via the morphisms $\pi_{\m}$ and the identification $\mu_{\m}$ given by Corollary~\ref{cor:topropLLZ2Cor526enhanced}. Recall that $\m=\n\ff$. We may define
 \[
\BF^\lambda_{\mathfrak{n}}\in  H^1(K(\mathfrak{n}),\TT_{f,\chi}\widehat\otimes\cH_{L,\ord_p(\lambda)}(\Gamma^\cyc)^\iota)
\]
to be the image of $\BF^\lambda_{\m}$ under the corestriction map.
Furthermore, arguing as in the proof of \cite[Proposition~3.10]{buyukbodukleianticycloord}, we see that these elements satisfy the twisted Euler-system norm relation (\ref{EQN_ESrelation}). 
We may decompose the pair of Beilinson-Flach classes $\{\BF^\alpha_{\mathfrak{n}}, \BF^\beta_{\mathfrak{n}}\}$ using Theorem~\ref{thm:decompBFhida} to obtain bounded classes, namely
\begin{equation}\label{eq:BFdecomp}
\begin{pmatrix}
\BF^{\alpha}_\mathfrak{n}\\ \BF^{\beta}_{\mathfrak{n}}
\end{pmatrix}= Q^{-1}\Tw_{k/2-1}M_{\log}\cdot 
\begin{pmatrix}
\BF^{\#}_{\mathfrak{n}}\\ \BF^{\flat}_{\mathfrak{n}}
\end{pmatrix},
\end{equation}
for some $\BF^{\#}_{\mathfrak{n}},\BF^{\flat}_{\mathfrak{n}}$ inside $H^1(K(\mathfrak{n}),\TT_{f,\chi}\otimes_{\cO_L} L)$. Here, we treat all the  Beilinson-Flach elements as elements of $H^1(K(\mathfrak{n}),\TT_{f,\chi}\widehat\otimes\cH_{L,k-1}(\Gamma^\cyc)^\iota)$.    

For $\bullet \in \{\#,\flat\}$, we write $c_\n^\bullet \in H^1(K(\n),\TT_{f,\chi})$ for the element $\varpi^s\cdot \BF_{\n}^\bullet$, where  $s$ was introduced as part of Theorem~\ref{thm:decompBFhida}. We simply write $c^\bullet$ in place of $c_1^\bullet \in H^1(K,\TT_{f,\chi})$ and denote by $c^\bullet_\cyc$ (respectively, $c^\bullet_\ac$) its image in $H^1(K,\TT_{f,\chi}^\cyc)$ (respectively, in $H^1(K,\TT_{f,\chi}^\ac)$). Similarly, we define $c_\n^\lambda$ to be $\varpi^s\cdot \BF_\n^\lambda$ and we have $c^\lambda=c_1^\lambda$ and $c_\cyc^\lambda$ and $c_\ac^\lambda$ may be defined in the same way.

\subsection{Local images of the Beilinson-Flach elements}
In this section, we study the image of the Beilinson-Flach elements defined in the previous sections under the Coleman maps at $\fp$ and $\fp^c$ defined in \S\ref{S:2varCol}. 

Let $\gm$ be the CM Hida family over $K$ we have fixed in \S\ref{S:ESimag}; recall also its relation with the Galois representation $H^1(\psi,\m\p^\infty)$, which is the main object of interest for us. We recall from \cite[\S7.2]{KLZ2} that there exists a short exact sequence of $G_{\QQ_p}$-representations  
\begin{equation}
\label{eqn:ohtawilespordseq}
0\rightarrow \cF^+H^1(\psi,\m\p^\infty)\rightarrow H^1(\psi,\m\p^\infty)\rightarrow \cF^-H^1(\psi,\m\p^\infty)\rightarrow 0
\end{equation}
such that both $\cF^+H^1(\psi,\m\p^\infty)(-1-\kappa)$ and $\cF^-H^1(\psi,\m\p^\infty)$ are unramified at $p$, where $\kappa$ is the weight character. Recall from Proposition~\ref{prop:LLZ2Cor526enhanced} that $H^1(\psi,\m\p^\infty)\cong \Ind_{K}^\QQ\Lambda_{\gm}$, where $\Lambda_\gm=\varprojlim_r \Ind_{K(\mathfrak{mp}^r)}^K\cO_L(\psi_L^{-1})$. 
Moreover, the sequence (\ref{eqn:ohtawilespordseq}) admits a splitting and  we may identify $ \cF^+ H^1(\psi,\m\p^\infty)|_{G_{\Qp}}$ with $ \Lambda_\gm|_{G_{K_\p}}$ whereas $\cF^-H^1(\psi,\m\p^\infty)|_{G_{\Qp}}$ gets identified with $\Lambda_\gm|_{G_{K_{\p^c}}}$.
We define the restriction map
\[
\res_\q:H^1(K(\m),\TT_{f,\chi})\rightarrow H^1(K(\m)_\q,\TT_{f,\chi}).
\]
for $\q\in\{\p,\p^c\}$.
We have the following result on the  image of $\BF_\gm^{\lambda}$ under $\res_{\p^c}$.
\begin{proposition}\label{prop:sameprojection}
For $\lambda\in\{\alpha,\beta\}$, we have
\[
\cL_{\lambda,\m,\p^c}\circ\res_{\p^c}(\BF_\m^\lambda)=0.
\]
Furthermore,
\[
\cL_{\alpha,\m,\p^c}\circ\res_{\p^c}(\BF_\m^\beta)=-\cL_{\beta,\m,\p^c}\circ\res_{\p^c}(\BF_\m^{\alpha}).
\]
\end{proposition}
\begin{proof}
For notational simplicity, we shall only consider the case when $\m=\ff$ (so that $K(\m)=K$ and $\LL_{\g_\ff}=\cO_L(\psi_L^{-1})\otimes\LL_{\cO_L}(\Gamma_\p)^\iota$) since the general case can be proved similarly. The image of $\BF_{1}^{\lambda,\Bg}\cdot e_{1-k/2}$ under the restriction map composed with the projection  
\[
H^1\left(\QQ,T_{f,\chi}\otimes H^1(\psi,\ff\p^\infty)\otimes \Lambda(\Gamma_0^\cyc)^\iota\right)\lra H^1\left(\Qp,T_{f,\chi}\otimes\cF^-H^1(\psi,\ff\p^\infty)\otimes \Lambda(\Gamma_0^\cyc)^\iota\right)
\]
coincides with $\res_{\p^c}(\BF^\lambda_\ff)$ under our identification of $\cF^-H^1(\psi,\ff\p^\infty)|_{G_{\Qp}}$ with  $\Lambda_{\g_\ff}|_{G_{K_{\p^c}}}$. Our first assertion therefore follows from  \cite[Theorem~7.1.2]{LZ1}. 
 We now prove the second statement. Let $g$ be a fixed classical specialization of $\gm$. It is enough to prove that our result holds for a Zariski dense set of $g$. We shall show that this is indeed the case for all $g$ whose weight is at least $ k$. As before, we may replace $\res_{\p^c}(\BF_\ff^\lambda)$ by $\res_p(\BF_1^{\lambda,g})$.

For $\lambda\in\{\alpha,\beta\}$, we write $\lambda'$ for the unique element of $\{\alpha,\beta\}\setminus\{\lambda\}$. Let $\cL_{\lambda}=\cL_{T,\lambda}\circ\res_{p}(\BF_1^{\lambda',g})$, where $T=R_f^*\otimes \cF^-R_g^*$ and $\cL_{T,\lambda}$ is as defined in \eqref{eq:notanotherdecomp}. Since $\cL_{T,\lambda}$ sends $H^1(\Qp,T\otimes \Lambda(\Gamma_0^\cyc)^\iota)$ to $\cH_{L,\ord_p(\lambda)}(\Gamma_0^\cyc)$ and $\res_{\fp^c}(\BF_1^{\lambda',g})$ can be considered as an element of  $H^1(\Qp,T\otimes \cH_{L,\ord_p(\lambda')}(\Gamma_0^\cyc)^\iota)$ , it follows that $\cL_{\lambda}$ lies inside $\cH_{L,k-1}(\Gamma_0^\cyc)$.

For each $j\in\{0,\ldots,k-2\}$ and a finite-order character  $\theta$ of $\Gamma_0^\cyc$, we have $\cL_\lambda(\chi_0^j\theta)=0$. This is due to the fact that the local image of the Beilinson-Flach elements in $H^1(\Qp(\mu_{p^n}),T(-j))$ are geometric for any $n\ge0$  by \cite[Proposition~3.3.3]{KLZ2}. This means precisely that they lie in the kernel of the dual exponential map. We therefore infer that $\cL_\lambda$ is divisible by $\log_{p,k-1}$. As $\cL_\lambda\in\cH_{L,k-1}(\Gamma_0^\cyc)$, the quotient $\cL_\lambda/\log_{p,k-1}$ is an element of $\Lambda_L(\Gamma_0^\cyc)$. 

By l'H\^opital's rule, $$\left(\cL_\alpha/\log_{p,k-1}\right)(\theta)=\left(\cL_\beta/\log_{p,k-1}\right)(\theta)\Leftrightarrow \cL_\alpha'(\theta)=\cL_\beta'(\theta)$$ if $\theta$ is a character where $\log_{p,k-1}$ vanishes. Therefore,
\begin{equation}\label{eq:compare}
\cL_\alpha=-\cL_\beta\Leftrightarrow \cL_\alpha'(\theta)=\cL_\beta'(\theta),
\end{equation}
for  all finite characters $\theta$ with conductors $p^n>1$.
From the derivative formula in \cite[Theorem~3.1.3]{LVZ}, we deduce that
\[
\left(\frac{\lambda}{c_g}\right)^n\cL_{\lambda}'(\theta)\cdot v_\lambda\equiv\frac{p^n}{\tau(\theta)}\log_{T,n}(e_\theta\cdot\BF_{1}^{f^{\lambda'},g})\mod L(\mu_{p^n})\otimes\Fil^0\Dcris(T),
\]
where $c_g$ is the $U_p$-eigenvalue of $g$, $e_\theta$ represents the idempotent attached to the character $\theta$ and $\tau(\theta)$ denotes its Gauss sum.
Recall that the image of $\BF_{1}^{f^{\lambda'},g}$ inside $ H^1(\QQ(\mu_{p^n}),T\otimes\Qp)$ is given by $(\lambda'c_g)^{-n}z_{n,0}^{\lambda',g}$,  (c.f. \S\ref{S:defnBF}). Hence we may rewrite the congruence above as
\[
\cL_{\lambda}'(\theta)\cdot v_\lambda\equiv\frac{p^n}{\tau(\theta)(\lambda\lambda')^n}\log_{T,n}(e_\theta\cdot z_{n,0}^{\lambda',g})\mod L(\mu_{p^n})\otimes\Fil^0\Dcris(T),
\] 
which is independent of $\lambda$ thanks to Corollary~\ref{cor:samething}. Consequently, \eqref{eq:compare} implies that
\[
\cL_\beta\cdot v_\beta\equiv \cL_\alpha\cdot v_\alpha\mod L(\mu_{p^n})\otimes\Fil^0\Dcris(T).
\]
However, modulo $\cH_L(\Gamma_0^\cyc)\otimes\Fil^0\Dcris(T)$, we have
\begin{align*}
\cL_\beta \cdot v_\beta&=\cL_\beta( -\beta v_1+p^{k-1} v_2)\equiv p^{k-1}\cL_\beta\cdot  v_2;\\
\cL_\alpha \cdot  v_\alpha&=\cL_\alpha (  \alpha v_1-p^{k-1}  v_2)\equiv -p^{k-1}\cL_\alpha \cdot v_2.
\end{align*}
Therefore, the result follows from combining the three congruences above.
\end{proof}

Recall from \S\ref{S:semilocal} that  $M_{\log,\q}$ is the logarithmic matrix obtained from $\Mlog$ on replacing $\gamma_0$ by its image in $\Gamma_\q$ for $\q\in\{\p,\p^c\}$.

\begin{corollary}\label{cor:ColBF}
For $\bullet\in\{\#,\flat\}$, we have
\[
\col_{\bullet, \m,\p^c}\circ\res_{\p^c}(\BF_\m^{\bullet})=0.
\]
Furthermore,
\[
\col_{\#,\m,\p^c}\circ\res_{\p^c}(\BF_\m^{\flat})=-\col_{\flat,\m,\p^c}\circ\res_{\p^c}(\BF_\m^{\#})=\frac{\cL_{\beta,\m,\p^c}\circ\res_{\p^c}(\BF_\m^{\alpha})}{\det(Q^{-1}\Tw_{k/2-1}M_{\log,\p^c})}.
\]
\end{corollary}
\begin{proof}
Once again, we assume that $\m=\ff$ for simplicity. Recall that our choice of $\iota_p$ implies that when we localize at $\p^c$, the Galois group $\Gamma_0^\cyc/\Delta=\overline{\langle\gamma_0\rangle}$ is identified with $\Gamma_{\p^c}$. Therefore, the factorization \eqref{eq:BFdecomp} implies that
\[
\begin{pmatrix}
\res_{\p^c}(\BF^{\alpha}_\mathfrak{n})\\\res_{\p^c}( \BF^{\beta}_{\mathfrak{n}})
\end{pmatrix}= Q^{-1}\Tw_{k/2-1}M_{\log,\p^c}\cdot 
\begin{pmatrix}
\res_{\p^c}(\BF^{\#}_{\mathfrak{n}})\\\res_{\p^c}( \BF^{\flat}_{\mathfrak{n}})
\end{pmatrix}.
\]
Let  $Q^{-1}\Tw_{k/2-1}M_{\log,\p^c}=\begin{pmatrix}
a&b\\c&d
\end{pmatrix}$ and 
\[
\wBF_{\#}=\res_{\p^c}\left(d\BF_{\m}^{\alpha}-b\BF_{\m}^{\beta}\right).
\]

 Proposition~\ref{prop:sameprojection} together with \eqref{eq:2vardecomp} imply that
\[
\begin{pmatrix}
b\cL_{\beta,\p^c}\circ\res_{\p^c}(\BF_\m^{\alpha})\\d\cL_{\beta,\p^c}\circ\res_{\p^c}(\BF_\m^{\alpha})\end{pmatrix}=
\begin{pmatrix}
a&b\\
c&d
\end{pmatrix}
\begin{pmatrix}
\col_{\#,\fp^c}(\wBF_{\#})\\
\col_{\flat,\fp^c}(\wBF_{\#})
\end{pmatrix},
\]
which tells us that
\[
\begin{pmatrix}
\col_{\#,\fp^c}(\wBF_{\#})\\
\col_{\flat,\fp^c}(\wBF_{\#})
\end{pmatrix}=
\begin{pmatrix}
0\\ \cL_{\beta,\p^c}\circ\res_{\p^c}(\BF_\m^{\alpha})
\end{pmatrix}.
\]
But 
 \[
\det(Q^{-1}\Tw_{k/2-1}M_{\log,\p^c})\BF_{\m}^{\#}=\wBF_{\#}
\]
 by definition, hence the result for $\BF_\m^{\#}$ follows. That for $\BF_{\m}^{\flat}$ may be obtained in the same way.
\end{proof}
\begin{remark}
Note that $\cL_{\beta,\m,\p^c}\circ\res_{\p^c}(\BF_\m^{\alpha})$ is the analogue of the $p$-adic $L$-functions  in \cite[Theorem~6.1.3(ii)]{LLZ2} and \cite[Theorem~2.3]{buyukbodukleianticycloord}, both of which are related to the Beilinson-Flach elements via the explicit reciprocity law (see \cite[Theorem~6.4.1]{LLZ2} and \cite[Theorem~3.12]{buyukbodukleianticycloord}). These $p$-adic $L$-functions interpolate the complex $L$-values $L(f,g,j)$, where $g$ is a theta series whose weight is greater than $k$ and $1\le j\le k-1$. 
\end{remark}

We now consider the local images of the Beilinson-Flach elements at $\fp$. 
\begin{defn}
 For $\lambda,\mu\in\{\alpha,\beta\}$, we define via the Perrin-Riou maps given by (\ref{eqn:PRMapcoordinatewiseextended})
\[
\fL_{\lambda,\mu}:=\cL_{\lambda,\p}\circ\res_\p(c^{\mu})\in\cH_{L,\ord_p(\lambda)}(\Gamma_{\fp})\,\widehat{\otimes}\,\cH_{L,\ord_p(\mu)}(\Gamma_{\fp^c}).
\]
We call  these four elements \emph{analytic Beilinson-Flach  $p$-adic $L$-functions}.
\end{defn}
 As discussed in Theorem~7.1.5 of \cite{LZ1} (also in \cite[Theorem~6.5.9]{KLZ1} in the $p$-ordinary case), these elements are expected to interpolate the (twisted) $L$-values of $f$ over $K$, when evaluated at $\chi_0^j$ for $j\in\{-k/2+1,\ldots,k/2-1\}$. In the situation when $\lambda=\mu$ we have the following result due to Loeffler and Loeffler-Zerbes~\cite{loefflerERL,LZ1} (which is adjusted to fit with our framework).
\begin{theorem}[Loeffler, Loeffler-Zerbes]
\label{ref:thmDavidsinterpolationformulae}
Let $\rho$ be a Hecke character of the form $\rho=\rho_0\,|\cdot|^{j}\,\nu$ where $\rho_0$ has $\infty$-type $(-u,0)$ with $0\leq u \leq k-2$ and has $\p$-power conductor; $j \in [1-k/2+u,-1+k/2]$ is an integer and finally, $\nu$ is a Dirichlet character of conductor $p^r$. Then we have,
\begin{align*}\frak{L}_{\lambda,\lambda}(\rho)=\frac{\mathcal{E}(f^\lambda,\rho)}{\mathcal{E}(f^\lambda)\,\mathcal{E}^*(f^\lambda)}&\times\frac{(j+k/2-1)!(j-u+k/2-1)!\,i^{k-u-1}}{\langle f,f,\rangle_{N_f}\pi^{k+2j-u}2^{2k+2j-u-1}}\times L(f/K,\chi\rho,k/2).\end{align*}
Here $\mathcal{E}(f^\lambda,\rho)$, $\mathcal{E}(f^\lambda)$ and $\mathcal{E}^*(f^\lambda)$ are suitable  interpolation factors.
\end{theorem}
\begin{remark}
\label{rem:insufficiencyoftheinterpolationformula}
The formula above  is a direct (but very rough) translation of Proposition 2.12 and Theorem 6.3 of \cite{loefflerERL}. We refer the readers to our companion article \cite{BF_Super_Addendum} $($the proof of {Theorem 3.4}$)$ where the interpolation factors  $\mathcal{E}(f^\lambda,\rho)$, $\mathcal{E}(f^\lambda)$ and $\mathcal{E}^*(f^\lambda)$ are explicitly calculated for a certain class of Hecke characters. One may of course state even a more general interpolation formula for $\frak{L}_{\lambda,\lambda}(\rho)$ using Loeffler's calculations, where $\rho$ is an arbitrary algebraic Hecke character $\rho$ of infinity type $(a,b)$ with $1-k/2\leq a\leq b \leq k/2-1$ and whose associated $p$-adic Galois character factors through $\Gamma$. Since we do not need this here $($nor in our companion article \cite{BF_Super_Addendum}$)$, we shall not be doing that.
We are currently unable to prove that these interpolation formulae uniquely determine $\frak{L}_{\lambda,\lambda}$ when $k>2$. This is the same reason why we were forced to prove a functional equation for the three-variable geometric $p$-adic $L$-function in \cite{BF_Super_Addendum}, rather than proving a functional equation for $\frak{L}_{\lambda,\lambda}$ directly.
\end{remark}
For $\bullet,\star\in\{\#,\flat\}$, we define 
\[
\fL_{\bullet,\star}:=\col_{\bullet,\p}\circ\res_\p(c^{\star})\in\Lambda_{\cO_L}(\Gamma)
\]
and call them  \emph{doubly-signed Beilinson-Flach $p$-adic $L$-functions}. If we combine the decomposition \eqref{eq:2vardecomp} with \eqref{eq:BFdecomp}, we have the factorization
\begin{equation}
\label{eqn:mainfactorization}
\begin{pmatrix}
\fL_{\alpha,\alpha}&\fL_{\beta,\alpha}\\
\fL_{\alpha,\beta}&\fL_{\beta,\beta}
\end{pmatrix}=Q^{-1}\Tw_{k/2-1}M_{\log,\fp^c}
\begin{pmatrix}
\fL_{\#,\#}&\fL_{\flat,\#}\\
\fL_{\#,\flat}&\fL_{\flat,\flat}
\end{pmatrix}\left(Q^{-1}\Tw_{k/2-1}M_{\log,\fp}\right)^T.
\end{equation}
We note that Loeffler's 2-variable $p$-adic $L$-fucntions for weight two modular forms constructed using modular symbols in \cite{loeffler13} satisfy a similar factorization. See  \cite{lei14}.

{
\begin{remark}
Recall from Remark~\ref{rk:col0} that the Coleman map $\col_{\bullet,\p}$ can be realized over $\cO_{L_0}$. Together with Remark~\ref{rk:three}, we see that $\fL_{\bullet,\star}$ lies inside $\Lambda_{\cO_{L_0}}(\Gamma)$ for $\bullet,\star\in\{\#,\flat\}$.
\end{remark}
}

\section{Locally restricted Euler systems and bounds on Selmer groups}
\label{sec:boundsonselmergroups}
Armed with the Euler systems we have constructed in \S\ref{sec:BFelementsandfactorization}, we may now prove bounds on the Selmer groups relevant to our study here. Until the end of this article, the hypotheses \textup{\textbf{(H.Im.)}}, \textup{\textbf{(H.Reg.)}} and \textup{\textbf{(H.SS.)}} are in effect. Throughout, we will also adopt the convention that  whenever we restrict our attention to the anticyclotomic line, the character $\chi$ will be assumed to be a ring class character. In particular,  we implicitly assume that $\chi$ is a ring class character whenever our discussion involves $(\textup{Sign} \pm)$. 
\subsection{Selmer structures}
\label{subsec:selmerstructures}
For any modulus $\n \in \mathcal{N}$ of $K$ as above and $?=\ac,\cyc,\emptyset$, we set $\LL_{\n}^{?}:=\cO_L[\textup{Gal}(K(\frak{n})/K)]\otimes \LL_{\cO_L}(\Gamma^?)$. For each $\bullet\in\{\#,\flat\}$ and $\mathfrak{q}\in\{\fp,\fp^c\}$, we recall from the end of \S\ref{S:semilocal} that we have the semi-local Coleman maps $\col_{\bullet,\n,\q}^?: H^1(K(\n)_\mathfrak{q},\TT_{f,\chi}^{?})\ra\LL_{\n}^{?}$.
We set
$$H^1_\bullet(K(\mathfrak{n})_\mathfrak{q},\TT_{f,\chi}^{?}):=\ker\left(\col_{\bullet,\n,\mathfrak{q}}^?: H^1(K(\n)_\mathfrak{q},\TT_{f,\chi}^{?})\ra\LL_{\n}^{?}\right)$$
and consider the following Selmer structures on $\TT_{f,\chi}^{?}$ for $?\in\{\ac,\cyc,\emptyset\}$:
\begin{itemize}
\item The canonical Selmer structure $\FFF_{\textup{can}}$ by requiring no local conditions at any place belonging to $\Sigma$.
\item The Greenberg Selmer structure $\FFF_{\emptyset,0}$ by replacing the local conditions determined by $\FFF_{\textup{can}}$ at $\fp^c$ by the $0$-subspace.
\item Partially-signed Selmer structure $\FFF_{\emptyset,\bullet}$ by replacing the local conditions determined by $\FFF_{\textup{can}}$ at $\mathfrak{p}^c$ with 
$$H^1_{\FFF_{\emptyset,\bullet}}(K(\fn)_{\mathfrak{p}^c},\TT_{f,\chi}^{?}):=H^1_\bullet(K(\fn)_{\mathfrak{p}^c},\TT_{f,\chi}^{?})\,.$$  
\item Doubly-signed Selmer structure $\FFF_{\bullet,\star}$ (where $\star \in \{\#,\flat\}$) by replacing the local conditions determined by $\FFF_{\emptyset,\bullet}$ at $\mathfrak{p}$ with 
$$H^1_{\FFF_{\star,\bullet}}(K(\fn)_\mathfrak{p},\TT_{f,\chi}^{?}):=H^1_\star(K(\fn)_{\mathfrak{p}},\TT_{f,\chi}^{?})\,.$$
\end{itemize}
 \begin{remark}
 \label{rem:signedstructurespropagatetoeachother}
We have chosen to denote the Selmer structures $\FFF$ (for $\FFF=\FFF_{\textup{can}}, \FFF_{\emptyset,\bullet}, \FFF_{\bullet,\star}$ or $\FFF_{\emptyset,0}$ ) on the three different Galois modules $\TT_{f,\chi}$, $\TT_{f,\chi}^{\ac}$, and $\TT_{f,\chi}^{\cyc}$ with the same symbols. This is intentional as it is easily checked that each $\FFF$ on $\TT_{f,\chi}$ propagates to the Selmer structure $\FFF$ on $\TT_{f,\chi}^{?}$ (for $?=\ac,\cyc$).
 \end{remark}
 
As usual, we also define the dual Selmer structures $\FFF^*$ on $\TT_{f,\chi}^{?,*}:=\textup{Hom}(\TT_{f,\chi},\bbmu_{p^\infty})$ by setting $H^1_{\FFF^*}(K_\lambda,\TT_{f,\chi}^{?,*}):=H^1_\FFF(K_\lambda,\TT_{f,\chi}^{?})^\perp$ for each $\lambda\in \Sigma$, the orthogonal complement with respect to the local Tate duality at $\lambda$. Each of these Selmer structures will allow us to define a Selmer group 
$$H^1_{\FFF}(K,\TT_{f,\chi}^{?}):=\ker\left(H^1(K_\Sigma/K,\TT_{f,\chi}^{?})\lra \bigoplus_{\lambda \in \Sigma}\frac{H^1(K_\lambda,\TT_{f,\chi}^{?})}{H^1_\FFF(K_\lambda,\TT_{f,\chi}^{?})}\right)$$
as well as their counterparts associated to the dual Selmer structure $\FFF^*$ on $\TT_{f,\chi}^{?,*}$.
\begin{defn}
\label{def:doublysignedselmergroups}
For $\bullet,\star\in \{\#,\flat\}$, and $?\in\{\ac,\cyc,\emptyset\}$, we set 
$$\mathfrak{X}_{\star,\bullet}^{?}:=H^1_{\FFF_{\star,\bullet}^*}(K,\TT_{f,\chi}^{?,*})^\vee\,,$$
$$\mathfrak{X}_{\bullet}^?=H^1_{\FFF_{\emptyset,\bullet}^*}(K,\TT_{f,\chi}^{?,*})^\vee\,,$$
and finally, also set $\mathfrak{X}_{\emptyset,0}^?=H^1_{\FFF_{\emptyset,0}^*}(K,\TT_{f,\chi}^{?,*})^\vee$.
\end{defn}
\subsection{Triviality and non-triviality of Beilinson-Flach elements}
\label{subsec:ERL}
We shall consider the following six properties concerning the local positions of the Beilinson-Flach elements at the primes above $p$. Even though we expect their validity at all times, we are able to verify them only partially (due to a variety of technical reasons, which we shall explain below). This is still sufficient for our goals towards Iwasawa main conjectures.
Suppose  $\star,\bullet\in\{\#,\flat\}$ and we write $\fL_{\star,\bullet}^\ac \in \LL_{\cO_L}(\Gamma^\ac)$ (respectively, $\fL_{\star,\bullet}^\cyc \in \LL_{\cO_L}(\Gamma^\cyc)$) for the restriction of the doubly-signed Beilinson-Flach $p$-adic $L$-function to the anticyclotomic (respectively, to the cyclotomic) characters. Fix $\bullet\in\{\#,\flat\}$ and $\lambda\in\{\alpha,\beta\}$.
\\\\
$\mathbf{(L0)}$ $\res_\fp\left(c_\cyc^\bullet\right) \neq 0$.
\\\\
$\mathbf{(L1)}$ For $\star,\bullet \in \{\#,\flat\}$, $\fL_{\star,\bullet}^\cyc \neq 0$.
\\\\
$\mathbf{(L2)}$ $\res_{\fp^c}(c^\bullet)\neq 0$\,.
\\\\
$\mathbf{(L3)}$   $(\textup{Sign} +)$ holds and $\res_{\fp}\left(c_\ac^\bullet\right) \neq 0$\,.
\\\\
$\mathbf{(L4)}$  $(\textup{Sign} -)$ holds and $\res_{\fp^c}(c_\ac^\bullet)\neq 0$\,.
\\\\
$\mathbf{(L5)}$ $(\textup{Sign} -)$ holds and $\fL_{\bullet,\bullet}^\ac=0$\,.
\\\\$\mathbf{(L.\lambda)}$ $(\textup{Sign} -)$ holds and $\mathfrak{L}_{\lambda,\lambda}^\ac=0$\,.
\\
\begin{proposition}
\label{prop:lzerocorrect}
There exists $\bullet\in \{\#,\flat\}$ such that the property $\mathbf{(L0)}$ holds true.
\end{proposition}
\begin{proof}
We may use~\cite[Corollary 5.3]{vanorderIII} to choose a Hecke character $\psi$ (whose $p$-adic avatar is cyclotomic) with infinity type $(0,0)$ and the property that $L(f,\chi\psi,k/2)\neq 0$. For $\lambda\in \{\alpha,\beta\}$ and $\bullet \in \{\#,\flat\}$, we let $c^\lambda_\psi, c^\bullet_\psi \in H^1_{\FFF_{\textup{can}}}(K,V_{f,\chi}\otimes\psi^{-1})$ {(where $V_{f,\chi}=T_{f,\chi}\otimes_{\ZZ_p}\QQ_p$)} denote the respective images of the classes $c^\lambda$ and $c^\bullet$. The explicit reciprocity law for Beilinson-Flach elements implies that $c^\lambda_\psi\neq 0$. Our assertion then follows from \eqref{eq:BFdecomp}. 
\end{proof}
\begin{remark}
One may obtain a more straightforward proof of Proposition~\ref{prop:lzerocorrect} when $k>2$ (without relying on Van Order's result), since the Rankin-Selberg $L$-series $L(f,\chi,s)$ admits a critical point at which its defining Euler product converges absolutely. 
\end{remark}
\begin{remark}
Assuming in addition that there is no prime $\mathfrak{q}\mid \mathfrak{f}$ such that ${\mathfrak{q}^c}\mid\mathfrak{f}$, one may rely on a result of Rohrlich to prove that $\res_\fp\left(c^\bullet\right) \neq 0$. The existence of a Hecke character $\psi$ as in the proof of Proposition~\ref{prop:lzerocorrect}  that guarantees the desired non-vanishing statement for the relevant Rankin-Selberg $L$-function follows from \cite[Theorem 2]{rohrlich88Annalen}. Note that Rohrlich's result does not apply with cyclotomic characters of $K$, as these violate Rohrlich's hypothesis on ramification.
\end{remark}
\begin{corollary}
\label{cor:lonecorrect}
There exists a choice of a pair $\bullet,\star\in \{\#,\flat\}$ such that $\mathbf{(L1)}$ holds true. More precisely, for $\bullet\in \{\#,\flat\}$ verifying the conclusion of Proposition~\ref{prop:lzerocorrect}, there exists $\star\in \{\#,\flat\}$ such that $\col_{\star,\fp}^\cyc(\res_{\fp}(c_\cyc^\bullet))\neq 0$.
\end{corollary}
\begin{proof}
Let $F/\Qp$ be a finite unramified extension. The kernel of  $1-\vp$  on $\NN_{F}(T)^{\psi=1}$ is contained in $T^{G_{F(\mu_{p^\infty})}}$ {(see for example \cite[proof of Proposition~4.11]{LZ0}). It follows from \cite[Lemma~4.4]{lei11compositio} that $T^{G_{F(\mu_{p^\infty})}}=0$ under our running assumption that $p>k$. In particular, $\cL_{T,F}$ is injective and the decomposition \eqref{eq:decompLTF} implies that} the map
$$(\col_{\#,\fp}^\cyc,\col_{\flat,\fp}^\cyc): H^1(K_\fp,\TT_{f,\chi}^\cyc)\lra \LL_{\cO_L}(\Gamma^\cyc)^{\oplus2}$$
is injective. The corollary now follows from Proposition~\ref{prop:lzerocorrect}.
\end{proof}
\begin{proposition}
\label{prop:ltwocorrect}
There exists $\bullet\in \{\#,\flat\}$ such that the property $\mathbf{(L2)}$ holds true.
\end{proposition}
\begin{proof}
Our proof is based on an argument due to Kings-Loeffler-Zerbes. 
We let $\BF^\alpha_* \in H^1(K_\fp,\TT_{f,\chi})$ denote the class obtained by interchanging the two factors of modular curves that appear in the definition of Beilinson-Flach classes (that amounts to interchanging the order of the eigenform $f$ and the CM Hida family $\g$). Notice that we have disposed of the the auxiliary choice of $c$ (as well as all the fudge factors involving it) as before and we are still working with the central critical twist. It follows from \cite[Proposition~5.2.3]{KLZ2}  that it suffices to prove the required non-vanishing for $\BF^\alpha_*$. 
Then the explicit reciprocity law for Beilinson-Flach classes reads (with the notation of \cite{KLZ2,LZ1})
$$\left\langle\log_{\fp^c}\left(\res_{\fp^c}\left(\phi\circ\BF^{\alpha}_{*,-j} \right)\right),\eta_{g_\phi}\otimes\omega_{f^\alpha}\right\rangle=\triangle(\phi,j)\cdot L_p(g_\phi,f,j+k/2)$$
where, 
\begin{itemize}
\item $j$ is any integer; 
\item $g_\phi$ is any specialization of the CM Hida family $\g$ that arises as the $p$-depleted theta series of a Hecke character $\phi$ of infinity type $(-w,0)$ with $w\geq 0$ (so that $g_\phi$ is an eigenform of weight $w+1$), whose $p$-adic avatar is denoted by $\phi_p$; 
\item $\phi\circ\BF^{\alpha}_{*,j} \in H^1(K,R_f^*(1-k/2-j)\otimes\phi_p^{-1})$ is the image of $\BF^{\alpha}_{*}$ under the obvious specialization map;
\item $\triangle(\phi,j)$ is a non-zero factor whose exact value we do not need to know; 
\item $L_p(\g,f,\bf{j})$ is Hida's $p$-adic $L$-function that interpolates the special values $L(g_\phi,f,1+r)$ for integers $r \in [k-1,w-1]$.
\end{itemize}
 In particular, we infer that
 $$\left\langle\log_{\fp^c}\left(\res_{\fp^c}\left(\phi\circ\BF^{\alpha}_{*,-j} \right)\right),\eta_{g_\phi}\otimes\omega_{f^\alpha}\right\rangle\neq0$$
 if we choose $w> k+3$ and $ w/2+1\leq j< w-k/2$, since the Euler product for $L(f/K,\phi,s)$ is absolutely convergent (and therefore non-vanishing) at $s=j+k/2$. The desired non-vanishing statement now follows from \eqref{eq:BFdecomp}.  
\end{proof}
\begin{proposition}
\label{prop:lthreecorrect}
Suppose that $p\nmid N_f\rm{disc}(K/\QQ)$, $N^-$ is a square-free product of an odd number of primes  (where we factor $N_f=N^+N^-$ so that $N^+$ (resp., $N^-$) is only divisible by primes that are split (resp., inert or ramified) in $K$) and if $V\vert N^-$ is ramified in $K/\QQ$, the condition \textup{(ST)} in \cite{pinchihung} holds with $\frak{n}_b=N^-$. Then, there exists $\bullet\in \{\#,\flat\}$ such that $\mathbf{(L3)}$ holds.
\end{proposition}
\begin{proof}First of all, our hypotheses ensure the validity of $(\textup{Sign}+)$. We use~\cite[Theorem 6.12]{pinchihung} to choose a Hecke character $\psi$ (whose $p$-adic avatar is anticyclotomic) with infinity type $(0,0)$ and the property that $L(f,\chi\psi,k/2)\neq 0$; the rest of the argument in the proof of Proposition~\ref{prop:lzerocorrect} goes through verbatim.
\end{proof}
\begin{corollary}
\label{cor:lthreeandhalfcorrect}
Assume the validity of the hypotheses of Proposition~\ref{prop:lthreecorrect} so that there exists $\bullet\in  \{\#,\flat\}$ that verifies $\mathbf{(L3)}$. Then there exists $\star\in  \{\#,\flat\}$ such that $\col_{\star,\fp}^\ac(\res_{\fp}(c_\ac^\bullet))\neq 0$.
\end{corollary}
\begin{proof}The proof of Corollary~\ref{cor:lonecorrect} applies verbatim to deduce our claim using Proposition~\ref{prop:lthreecorrect}.
\end{proof}
\begin{proposition}
\label{prop:lfourcorrect}
Suppose that $p\nmid N_f\rm{disc}(K/\QQ)$ and $N^-$ is a square-free product of an even number of primes. Then exists $\bullet\in \{\#,\flat\}$ such that the property $\mathbf{(L4)}$ holds true. 
\end{proposition}
\begin{proof}Our hypotheses  ensure the validity of $(\textup{Sign}-)$. Fix a Hecke character $\psi$ of infinity type $(-k/2-m,k/2+m)$ with $m\geq 0$ whose $p$-adic avatar is anticyclotomic with trivial conductor. It follows from \cite[Theorem C]{Hseihanticyclomuandnonvanish} that we have $L(f,\chi\psi\eta,k/2)\neq 0$ for all but finitely many ring class characters $\eta$ with $p$-power conductor. Arguing as in the proof of Proposition~\ref{prop:ltwocorrect} (with $j=k/2+m$ and the Hecke character $\phi:=\chi\psi\eta|\cdot|^{-k/2-m}$), the result follows from the explicit reciprocity laws and \eqref{eq:BFdecomp}. 
\end{proof}
\begin{remark}
The reason why we are able to verify the properties $\mathbf{(L0)}$ -- $\mathbf{(L4)}$ for only some choice of $\bullet,\star \in \{\#,\flat\}$ is that  we currently do not have access to a sufficiently explicit interpolation formula for $\frak{L}_{\bullet,\star}$ for general $a_p(f)$. When $a_p(f)=0$, we expect that the properties $\mathbf{(L0)}$ -- $\mathbf{(L4)}$ could be verified for all possible choices  of $\bullet,\star$. In order to keep the length of the current article within reasonable limits, we postpone this discussion (which is secondary for our purposes here) to a future article. 
\end{remark}

\begin{proposition}
\label{prop:Lalpha}
If $\varepsilon_f=\mathbb{1}$, then $\mathbf{(L.\lambda)}$ holds true. Moreover, if $a_p(f)=0$ and the ramification index of {$L_0/\QQ_p$} is odd, then $\mathbf{(L5)}$ holds true. 
\end{proposition}
\begin{proof}
The first assertion is \cite[{Corollary~3.8}]{BF_Super_Addendum}, whereas the second assertion is {Theorem~4.1} in \textit{op. cit}. 
\end{proof}
\begin{proposition}
\label{prop:positionofBFelements}
$\,$
\begin{enumerate}[(i)]
\item  For every $\n\in \mathcal{N}$ and $\bullet\in \{\#,\flat\}$, we have $\res_{\fp^c}(c_\n^\bullet) \in H^1_\bullet(K(\n)_\mathfrak{p^c},\TT_{f,\chi})$.
\item Assume \BLtwo\ holds for some $\bullet\in \{\#,\flat\}$. Then, $\LL_{\cO_L}(\Gamma)\cdot c^\bullet \cap H^1_{\FFF_{\emptyset,0}}(K,\TT_{f,\chi})=0$.
\item Assume $\mathbf{(L4)}$ holds for some $\bullet\in \{\#,\flat\}$. Then,  $\LL_{\cO_L}(\Gamma^\ac)\cdot c_\ac^\bullet \cap H^1_{\FFF_{\emptyset,0}}(K,\TT_{f,\chi}^\ac)=0$.
\end{enumerate}
\end{proposition}
\begin{proof}
(i) follows from  Corollary~\ref{cor:ColBF}. Notice that the quotient
$$\frac{ H^1_{\FFF_{\emptyset,\bullet}}(K,\TT_{f,\chi})}{ H^1_{\FFF_{\emptyset,0}}(K,\TT_{f,\chi})}\hookrightarrow H^1_\bullet(K_{\fp^c},\TT_{f,\chi})\subset  H^1(K_{\fp^c},\TT_{f,\chi})$$
 is torsion-free and (ii) is now a consequence of (i) and \BLtwo. (iii) similarly follows from (i) and $\mathbf{(L4)}$.
\end{proof}

\subsection{Signed main conjectures}

\begin{theorem}
\label{thm:leopoldtconjecturesforTfchi}
For $\bullet,\star\in\{\#,\flat\}$ and $?\in \{\ac,\cyc,\emptyset\}$ we have the following.
\begin{itemize}
\item[(i)] If $c^\bullet_?\neq0$, then the $\LL_{\cO_L}(\Gamma^?)$-module $H^1_{\FFF_{\emptyset,\bullet}}(K,\TT_{f,\chi}^{?})$ is of rank one and  $\mathfrak{X}_{\bullet}^?$ is torsion. In particular, this conclusion is valid for some choice of $\bullet \in \{\#,\flat\}$ if
\begin{itemize}
\item[a)] $?=\cyc,\emptyset$,
\item[b)]  $?=\ac$ and the hypotheses of Proposition~\ref{prop:lthreecorrect} are verified (so that $(\textup{Sign}+)$ is in force),
\item[c)]  $?=\ac$ and the hypotheses of Proposition~\ref{prop:lfourcorrect} are verified (so that $(\textup{Sign}-)$ is in force).
\end{itemize}
\item[(ii)]   $H^1_{\FFF_{\emptyset,0}}(K,\TT_{f,\chi}^{?})=0$ and  $\mathfrak{X}_{\emptyset,0}^?$ is torsion over $\Lambda_{\cO_L}(\Gamma^?)$ in the following situations.
\begin{itemize}
\item[a)] $?=\emptyset$.
\item[b)] $?=\ac$ and the hypotheses of Proposition~\ref{prop:lfourcorrect} are verified.
\end{itemize}
\item [(iii)] There exists a choice $\bullet,\star \in \{\#,\flat\}$ such that the $\LL_{\cO_L}(\Gamma)$-module $\mathfrak{X}_{\star,\bullet}$ and the $\LL_{\cO_L}(\Gamma^\cyc)$-module $\mathfrak{X}_{\star,\bullet}^{\cyc}$ are both torsion. Furthermore, $H^1_{\FFF_{\star,\bullet}}(K,\TT_{f,\chi})=H^1_{\FFF_{\star,\bullet}}(K,\TT_{f,\chi}^{\cyc})=0$
for the same choice of $\bullet$ and $\star$.
\item[(iv)] Under the hypotheses of Proposition~\ref{prop:lthreecorrect},  there exists a choice $\bullet,\star \in \{\#,\flat\}$ such that $\LL_{\cO_L}(\Gamma^\ac)$-module $\mathfrak{X}_{\bullet,\star}^{\ac}$ is torsion and $H^1_{\FFF_{\bullet,\star}}(K,\TT_{f,\chi}^{\ac})=0$ .
\item[(v)] Suppose $a_p(f)=0$ as well as that the hypotheses of Proposition~\ref{prop:lfourcorrect} and Proposition~\ref{prop:Lalpha} hold true. Then for the choice of $\bullet \in \{\#,\flat\}$ we make to validate \textup{(i).c} above, both modules $\mathfrak{X}_{\bullet,\bullet}^{\ac}$ and $H^1_{\FFF_{\bullet,\bullet}}(K,\TT_{f,\chi}^{\ac})$ have rank $1$ over $\LL_{\cO_L}(\Gamma^\ac)$.
\end{itemize}
\end{theorem}
\begin{proof}
(i) The assertion that $\mathfrak{X}_{\bullet}^?$ is torsion (under the assumption that $c^\bullet_?\neq0$) is a consequence of the locally restricted Euler system machinery developed in \cite[Appendix A]{buyukbodukleiPLMS}. We explain here how to apply the results therein. For each $\n \in \mathcal{N}$  (we remark that these moduli are denoted by $\eta$ in \textit{loc. cit.}) we recall that $\LL_{\n}=\cO_L[\textup{Gal}(K(\frak{n})/K)]\otimes \LL_{\cO_L}(\Gamma)$ and let
$$\Psi^{(\n)}:=(\col_{\bullet,\n,\fp},\col_{\bullet,\n,\fp^c}):H^1(K(\n)_p,\TT_{f,\chi})=H^1(K(\n)_\fp,\TT_{f,\chi})\oplus H^1(K(\n)_{\fp^c},\TT_{f,\chi})\lra \LL_{\n}^{\oplus 2}.$$ 
We note that the $L$-restricted Selmer structure $\FFF_L$ considered in Definition A.10 of \textit{op. cit.} corresponds to our $\FFF_{\emptyset,\bullet}$ (with $L=L(1)$ and $L(\n)=(\Psi^{(\n)})^{-1}(\LL_{\n}^{\fp})$, where $\LL_{\n}^{\fp}\subset \LL_{\n}^{\oplus 2}$ is the summand corresponding to $\col_{\bullet,\n,\fp}$). Proposition~\ref{prop:positionofBFelements}(i) shows that the collection $\{c_{\n}^\bullet\}$ forms an $L$-restricted Euler system and the proof of Theorem A.11 in \textit{op. cit.} yields an $L$-restricted Kolyvagin system for the Selmer structure $\FFF_L=\FFF_{\emptyset,\bullet}$. The equation (1.5) in the proof of Theorem A.14 in \textit{op. cit.} together with  the non-triviality of $c^\bullet_?$ complete the proof that $\mathfrak{X}_{\bullet}^?$ is torsion. The second assertion of (i) follows from Poitou-Tate global duality, a) from Proposition~\ref{prop:lzerocorrect}, b) from Proposition~\ref{prop:lthreecorrect} and c) from Proposition~\ref{prop:lfourcorrect}.
(ii) Since the module $H^1_{\FFF_{\emptyset,0}}(K,\TT_{f,\chi}^{?})$ is torsion-free, the first assertion follows from (i) (parts a and c) and Proposition~\ref{prop:positionofBFelements}(ii) and (iii). The second assertion that $\mathfrak{X}_{\emptyset,0}^?$ is torsion over $\Lambda_{\cO_L}(\Gamma^?)$ follows from the first via global duality.
(iii) In order to prove that $H^1_{\FFF_{\star,\bullet}}(K,\TT_{f,\chi})=H^1_{\FFF_{\star,\bullet}}(K,\TT_{f,\chi}^{\cyc})=0$, we only need to check that $c_\cyc^\bullet \notin H^1_{\FFF_{\star,\bullet}}(K,\TT_{f,\chi}^\cyc)$. The existence of a choice of $\star,\bullet \in \{\#,\flat\}$ with this property follows from Corollary~\ref{cor:lonecorrect}. The rest of (iii) follows from Poitou-Tate global duality.
(iv) As in the proof of (iii), we only need to verify under the running hypotheses that $c_\ac^\bullet \notin H^1_{\FFF_{\bullet,\bullet}}(K,\TT_{f,\chi}^\ac)$ and Corollary~\ref{cor:lthreeandhalfcorrect} tells us that there exists a choice of $\bullet\in \{\#,\flat\}$ with this property.
(v) Set $H^1_{/\bullet}(K_{\fp},\TT_{f,\chi}^\ac):=H^1(K_{\fp},\TT_{f,\chi}^\ac)\big{/}H^1_{\bullet}(K_{\fp},\TT_{f,\chi}^\ac)$ and let $\res_{\bullet/\fp}$ denote the composition of the arrows 
$$H^1(K,\TT_{f,\chi}^\ac)\ra H^1(K_{\fp},\TT_{f,\chi}^\ac) \ra H^1_{/\bullet}(K_{\fp},\TT_{f,\chi}^\ac).$$ 
Notice that $H^1_{/\bullet}(K_{\fp},\TT_{f,\chi}^\ac)$ is torsion-free (being isomorphic to a submodule of $\LL_{\cO_L}(\Gamma^\ac)$ via $\col_{\bullet,\fp}^\ac$).
It follows from Proposition~\ref{prop:lfourcorrect} that $\res_{\bullet/\fp}(c_\ac^\bullet)=0$. Since the quotient module $H^1_{\FFF_{\emptyset,\bullet}}(K,\TT_{f,\chi}^{\ac})\big{/}\LL_{\cO_L}(\Gamma^\ac)\cdot c_\ac^\bullet$ is torsion by (i).c, it follows that 
$\res_{\bullet/\fp}\left(H^1_{\FFF_{\emptyset,\bullet}}(K,\TT_{f,\chi}^{\ac})\right)=0.$
The proof of (v) now follows from (i).c and the global duality sequence
\[
0\ra H^1_{\FFF_{\bullet,\bullet}}(K,\TT_{f,\chi}^{\ac})\lra H^1_{\FFF_{\emptyset,\bullet}}(K,\TT_{f,\chi}^{\ac})\stackrel{\res_{\bullet/\fp}}{\lra} H^1_{/\bullet}(K_\fp,\TT_{f,\chi}^{\ac})
\lra \mathfrak{X}_{\bullet,\bullet}^{\ac} \lra \mathfrak{X}_{\bullet}^{\ac}\ra 0.
\]
\end{proof}
We now formulate our signed main conjectures.
\begin{conj}
\label{conj:signedmainconjpartiallyandoubly}
For $\bullet,\star\in\{\#,\flat\}$ and $?\in \{\ac,\cyc,\emptyset\}$, the following two assertions hold true:
\begin{itemize}
\item[(i)]   $\Char\left(\mathfrak{X}_{\bullet}^{?}\right)\dot{=}\, \Char\left(H^1_{\FFF_{\emptyset,\bullet}}(K,\TT_{f,\chi}^{?})\big{/}\LL_{\cO_L}(\Gamma^?)\cdot c_{?}^\bullet\right)$\,.
\item[(ii)] $\Char\left(\mathfrak{X}_{\star,\bullet}^{?}\right)\dot{=}\, \col_{\star,\fp}^?\left(\res_{\fp}(c_{?}^\bullet/\xi_\star^?)\right)$.
\end{itemize}
Here, $``\dot{=}"$ means equality up to powers of $p$ and  $\xi_{\star}^?$ is as given at the end of \S\ref{S:semilocal}.
\end{conj}
We note that these two assertions are equivalent under certain conditions, see Proposition~\ref{prop:twoassertionsequivalent}.
We shall call the first assertion the \emph{partially-signed main conjecture} and the second assertion the \emph{doubly-signed main conjecture} for $f\otimes\chi$. 
\begin{remark}
\label{rem:trivialinthedefinitecase}
Conjecture~\ref{conj:signedmainconjpartiallyandoubly} trivially holds when $?=\ac$, $\star=\bullet$, $a_p(f)=0$ and the hypotheses of Proposition~\ref{prop:lfourcorrect} hold true (as it reduces to $0=0$ thanks to Theorem~\ref{thm:leopoldtconjecturesforTfchi}(v) in this set up). In Section~\ref{subsec:refinedindefinitemainconj}, we shall present a refined version of the anticyclotomic main conjecture in this set up, that concerns the torsion-submodule of a Pottharst-style Selmer group.
\end{remark}
\begin{proposition}
\label{prop:twoassertionsequivalent}
Suppose \begin{itemize}
\item either that $?=\emptyset$ and $\bullet,\star \in \{\#,\flat\}$ verifies the conclusion of Corollary~\ref{cor:lonecorrect}, 
\item or else that  $?=\ac$, hypotheses of Proposition~\ref{prop:lthreecorrect} holds and $\bullet,\star \in \{\#,\flat\}$ verifies the conclusion of Corollary~\ref{cor:lthreeandhalfcorrect}.
\end{itemize}
Then, the partially-signed main conjecture and the doubly-signed main conjecture in Conjecture~\ref{conj:signedmainconjpartiallyandoubly} (for these choices of $\bullet$ and $\star$) are equivalent.  
\end{proposition}
\begin{proof}
 For the said choice of $\bullet$ and $\star$, the Poitou-Tate global duality together with Theorem~\ref{thm:leopoldtconjecturesforTfchi}(ii) yield an exact sequence
\begin{align}
\label{eqn:4termexactpartialagainstdouble}
0\ra \frac{H^1_{\FFF_{\emptyset,\bullet}}(K,\TT_{f,\chi}^?)}{\LL_{\cO_L}(\Gamma^?)\cdot c^\bullet_?}\stackrel{\res_{\star/\fp}}{\lra} &\frac{H^1_{/\star}(K_{\fp},\TT_{f,\chi}^?)}{\LL_{\cO_L}(\Gamma^?)\cdot \res_{\star/\fp}(c^\bullet_?)}
\ra \mathfrak{X}^?_{\star,\bullet}\ra \mathfrak{X}_\bullet^?\ra 0
\end{align}
where $H^1_{/\star}(K_{\fp},\TT_{f,\chi}^?):=H^1(K_{\fp},\TT_{f,\chi}^?)\big{/}H^1_{\star}(K_{\fp},\TT_{f,\chi}^?)$ and $ \res_{\star/\fp}$ is the compositum of the arrows 
$$H^1(K,\TT_{f,\chi}^?)\ra H^1(K_{\fp},\TT_{f,\chi}^?) \ra H^1_{/\star}(K_{\fp},\TT_{f,\chi}^?)$$ 
as in the proof of Theorem~\ref{thm:leopoldtconjecturesforTfchi}(v). This shows that the statement of Conjecture~\ref{conj:signedmainconjpartiallyandoubly}(i) is equivalent to the assertion that 
\begin{equation}
\label{eqn:4termexactpartialagainstdoubleconseq}\Char\left(\mathfrak{X}_{\star,\bullet}^{?}\right)\dot{=}\,\Char\left({H^1_{/\star}(K_{\fp},\TT_{f,\chi}^?)}\big{/}{\LL_{\cO_L}(\Gamma^?)\cdot \res_{\star/\fp}(c^\bullet_?)} \right)\,.
\end{equation}
Since $\col^?_{\star,\fp}$ is a pseudo-surjective map into $\LL_{\cO_L}(\Gamma^?)\xi_{\star}^?$ by Lemma~\ref{lem:2varimage}, it follows that $\col^?_{\star,\fp}\left(\res_{\fp}(c_{?}^\bullet)\right)/\xi_\star^?$ generates the characteristic ideal on the right of \eqref{eqn:4termexactpartialagainstdoubleconseq}. This concludes the proof.
\end{proof}

\subsection{Euler system bounds for doubly-signed main conjectures}\label{subsec:ESboundsforsignedSelmer}
The following divisibilities are obtained from making use of the locally restricted Euler system of Beilinson-Flach elements. 
\begin{theorem}
\label{thm:halfofMainConjtwovarCycloDefinite}
Suppose that $\bullet,\star\in\{\#,\flat\}$ and $?\in \{\ac,\cyc,\emptyset\}$.
\begin{itemize}
\item[(i)]   $\Char\left(H^1_{\FFF_{\emptyset,\bullet}}(K,\TT_{f,\chi}^{?})\big{/}\LL_{\cO_L}(\Gamma^?)\cdot c_{?}^\bullet\right)\subset \Char\left(\mathfrak{X}_{\bullet}^{?}\right)\,.$
\item[(ii)] Suppose 
\begin{itemize}
\item either that $?=\ac$ and the pair $(\bullet,\star)$ verifies the conclusion of Corollary~\ref{cor:lthreeandhalfcorrect};
\item or else that $?=\cyc,\emptyset$ and the pair $(\bullet,\star)$ verifies the conclusion of Corollary~\ref{cor:lonecorrect}. 
\end{itemize}
We then have
$$\col_{\star,\fp}^?\left(\res_{\fp}(c_{?}^\bullet)\right)\subset \xi_{\star}^?\cdot\Char\left(\mathfrak{X}_{\star,\bullet}^{?}\right)$$
\item[(iii)] Under the assumptions of \textup{(ii)}, the containment \textup{(i)} is an equality up to a power of $p$ if and only if the same is true for that of \textup{(ii)}.
\end{itemize}
\end{theorem}
\begin{proof}
The first containment follows from employing the techniques of \cite[Appendix A]{buyukbodukleiPLMS}, more specifically, the  equation (1.5) therein (with the choices we have recorded in the proof of  Theorem~\ref{thm:leopoldtconjecturesforTfchi}(i)). Note that this inclusion says something non-trivial only when $c_?^\bullet\neq0$ and the conditions in (ii) guarantee that.
It follows from the sequence (\ref{eqn:4termexactpartialagainstdouble}) and the first part of our theorem that 
$$\Char\left({H^1_{/\star}(K_{\fp},\TT_{f,\chi}^?)}\big{/}{\LL_{\cO_L}(\Gamma^?)\cdot \res_{\star/\fp}(c^\bullet_?)} \right)\subset \Char\left(\mathfrak{X}_{\star,\bullet}^{?}\right)
$$
with equality if and only we have equality in the first part. Noting that 
$$\xi_{\star}^?\cdot\Char\left({H^1_{/\star}(K_{\fp},\TT_{f,\chi}^?)}\big{/}{\LL_{\cO_L}(\Gamma^?)\cdot \res_{\star/\fp}(c^\bullet_?)} \right)=\col^?_{\star,\fp}\left(\res_\fp(c_?^\bullet)\right),$$
the proof follows.
\end{proof}
\begin{theorem}
\label{thm:integralBFundersignedlog}
Suppose  that the hypotheses of Proposition~\ref{prop:lfourcorrect} are in effect and $\bullet\in\{\#,\flat\}$ verifies the conclusion of this proposition. Then,
\begin{align*}\textup{char}_{\LL_{\cO_L}(\Gamma^\ac)}\left(\mathfrak{X}_{\emptyset,0}\right)\, &\,\,\Big{|}\,\,\textup{char}_{\LL_{\cO_L}(\Gamma^\ac)}\left(H^1_\bullet(K_{\fp^c},\TT^\ac_{f,\chi})\big{/}\LL(\Gamma^\ac)\cdot\res_{\fp^c}(c_\bullet^\ac)\right),
\end{align*}
Furthermore, 
 $$\textup{char}_{\LL_{\cO_L}(\Gamma^\ac)}\left(\mathfrak{X}_{\emptyset,0}\right) \dot{=}\ \textup{char}_{\LL_{\cO_L}(\Gamma^\ac)}\left(H^1_\bullet(K_{\fp^c},\TT^\ac_{f,\chi})\big{/}\LL(\Gamma^\ac)\cdot\res_{\fp^c}(c_\bullet^\ac)\right)$$
(equality up to a power of $p$) if and only if 
$$\Char\left(H^1_{\FFF_{\emptyset},\bullet}(K,\TT_{f,\chi}^{?})\big{/}\LL_{\cO_L}(\Gamma^?)\cdot c_{?}^\bullet\right)\dot{=}\ \Char\left(\mathfrak{X}_{\bullet}^{?}\right)\,.$$
 \end{theorem}
\begin{proof}
The Poitou-Tate global duality together with Theorem~\ref{thm:leopoldtconjecturesforTfchi}(ii) give rise to an exact sequence
\begin{align*}
0\ra \frac{H^1_{\FFF_{\emptyset,\bullet}}(K,\TT_{f,\chi}^\ac)}{\LL_{\cO_L}(\Gamma^\ac)\cdot c^\bullet_\ac}\stackrel{\res_{\fp^c}}{\lra} &\frac{H^1_\bullet(K_{\fp^c},\TT_{f,\chi}^\ac)}{\LL_{\cO_L}(\Gamma^\ac)\cdot \res_{\fp^c}(c^\bullet_\ac)}\ra \mathfrak{X}^\ac_{\emptyset,0}\ra \mathfrak{X}_\bullet^\ac\ra 0
\end{align*}
The proof now follows from Theorem~\ref{thm:halfofMainConjtwovarCycloDefinite}(i).
\end{proof}
\section{A refinement (Indefinite Anticyclotomic Main Conjectures)}
\label{subsec:refinedindefinitemainconj}
Suppose throughout this section that ${(\hbox{Sign} -)}$ is in effect as well as that the nebentype $\varepsilon_f$ is trivial\footnote{We impose this condition because we have verified the functional equation for the three-variable (geometric) $p$-adic $L$-function of Loeffler and Zerbes only under this hypothesis. We expect no essential difficulty to extend it beyond that, except for the fact that the root number calculations in the general case would necessitate a much lengthier discussion.}. Theorem~\ref{thm:leopoldtconjecturesforTfchi}(iii) shows that the signed Selmer groups (over the anticyclotomic tower) are non-torsion and the signed-main anticyclotomic conjectures in this situation could not assert anything non-trivial. However, one may consider a suitable Pottharst-style trianguline Selmer group over the anticyclotomic tower (we are grateful to Laurent Berger for explaining to us how the main constructions of Pottharst would extend without difficulty to our case of interest. Having said that, all inaccuracies in this section are of course ours). While still being a non-torsion object itself, it turns out that the torsion-subgroup of the trianguline Selmer group will acquire an interpretation in terms of $p$-adic $L$-functions, much in the spirit of Perrin-Riou's conjectures in the $p$-ordinary set up (and its parallel in our companion article, Theorem 1.1(ii) of \cite{buyukbodukleianticycloord}).
We shall follow the notation of Pottharst in \cite{jaycyclotmotives, jayanalyticfamilies}; in particular,  we shall write $\Lambda_\infty$ for the analytification of the \emph{anticyclotomic} Iwasawa algebra. We   set $\TT^\dagger:=T_{f,\chi}\,\widehat{\otimes}\,\LL_\infty^\iota$ where $G_K$-acts diagonally as usual (and via the universal anticyclotomic character on the latter factor). 

Let $\mathfrak{q}\in\{\fp,\fp^c\}$. Under \textbf{(H.SS.)}, $\q$ is totally ramified in $K^{\ac}$. It follows from Local Class Field Theory  $K^\ac_\q/K_\q$ can be realized as a sub-extension of  a Lubin-Tate extension of height one. By an abuse of notation, we shall identify $\Gamma^\ac$ with the local Galois group $\Gal(K^\ac_\q/K_\q)$.

\begin{remark}
\label{rem:extensionofpottharsttheorytolubintatetowers}
In order to extend Pottharst's theory to our set up, we need to establish the required formalism of $(\varphi,\Gamma^\ac)$-modules. There has been much activity recently concerning this matter, see in particular \cite{Bergermultivariable, FourquauxXie, SchneiderVenjakob, KisinRen}. We outline here the basic constructions relevant to our studies. 
 In order to define the trianguline Selmer groups along the cyclotomic tower, one makes an essential use of the Fontaine-Herr complex of cyclotomic $(\varphi,\Gamma)$-modules to recover the Iwasawa cohomology; Herr's formulation of local duality is also needed. These results are readily generalized by Fourquaux-Xie and Schneider-Venjakob; see \cite[Theorem 4.2]{FourquauxXie} and  \cite[Section 3]{SchneiderVenjakob}. One crucial point is that the ``F-analytic" condition that makes an appearance in the theory of Lubin-Tate $(\varphi,\Gamma)$-modules is vacuous when the relevant Lubin-Tate group has height one. 
The second ingredient in the construction of cyclotomic triangulordinary Selmer groups is the work of Berger and Kisin, that allows one to attach a cyclotomic $(\varphi,\Gamma)$-submodule of $D \subset D_\textup{rig}^\dagger(V)$ to any $\varphi$-finite-slope-eigenspace within the crystalline Dieudonn\'e module of $V$. This is carried out in the context of Lubin-Tate extensions of height one in \cite[Sections 2.2 and 2.3]{KisinRen}.
 \end{remark}

For each $\q\in\{\p,\p^c\}$, we let $D_\mathfrak{q}$ denote the $(\varphi,\Gamma^\ac)$-module associated to $T_{f,\chi}|_{G_{K_\fq}}$. For $\lambda\in\{\alpha,\beta\}$, we write $F_\lambda^{(\mathfrak{q})}\subset D_\mathfrak{q}$ for the associated saturated $(\varphi,\Gamma^\ac)$-submodules of rank one associcated to corresponding the $\varphi$-eigenspace of $\Dcris(T_{f,\chi})$ consrtucted by Kisin and Ren. We shall drop $\mathfrak{q}$ from the notation when no confusion may arise and write $F_\lambda$ and $D_p$ in place of $F_\lambda^{(\mathfrak{q})}$ and $D_\mathfrak{q}$. We shall write $\han^*(\,\cdot\,,\,\cdot\,)$ for the analytic Iwasawa cohomology groups in order to distinguish these objects from their classical counterparts. Note that Pottharst in \cite{jaycyclotmotives} denotes these objects by $H^*_\Iw(\,\cdot\,,\,\cdot\,)$.
\subsection{Analytic Selmer complexes}
\label{subsec:analyticSelmergroups}
Fix from now on $\lambda\in\{\alpha,\beta\}$. The basic constructions in \cite[\S3.2]{jayanalyticfamilies} and \cite[\S2]{jaycyclotmotives} (or rather, their mild generalization to our set up, through the works of Fourquaux-Xie, Schneider-Venjakob and Kisin-Ren that we have recalled in Remark~\ref{rem:extensionofpottharsttheorytolubintatetowers}) equip us with four  analytic Selmer complexes $\widetilde{\textbf{R}\Gamma}(G_{K,S},\Delta_{\emptyset,0},\TT^\dagger)$, $\widetilde{\textbf{R}\Gamma}(G_{K,S},\Delta_{\emptyset,\lambda},\TT^\dagger)$, $\widetilde{\textbf{R}\Gamma}(G_{K,S},\Delta_{\lambda,\lambda},\TT^\dagger)$ and $\widetilde{\textbf{R}\Gamma}(G_{K,S},\Delta_{0,\lambda},\TT^\dagger)$, which are  objects in the derived category of $\LL_\infty$-modules that may be represented by perfect complexes in degree $[0,3]$ (perfectness of these complexes may be deduced as a consequence of the anticyclotomic counterparts of \cite[Theorem 4.4.6]{KPX}),  which correspond to the following four local conditions at primes above $p$:
\begin{itemize}
\item $\Delta_{\emptyset,0}$\,: No local condition at $\fp$ (given by $\textbf{R}\Gamma(G_{K_\fp},D_\fp)\lra \textbf{R}\Gamma(G_{K_\fp},D_\fp)$ in the notation of \cite[\S3.2]{jayanalyticfamilies})  and the strict local conditions at $\fp^c$ (given by $\textbf{R}\Gamma(G_{K_{\fp^c}},0)\ra \textbf{R}\Gamma(G_{K_{\fp^c}},D_{\fp^c})$).
\item $\Delta_{\emptyset,\lambda}$\,:  No local condition at $\fp$, whereas that at $\fp^c$ is given via $\textbf{R}\Gamma(G_{K_{\fp^c}},F_\lambda^{(\fp^c)})\ra \textbf{R}\Gamma(G_{K_{\fp^c}},D_{\fp^c})$. 
\item $\Delta_{\lambda,\lambda}$\,: For both  $\mathfrak{q}\in\{\fp, \fp^c\}$, the condition at $\mathfrak{q}$ is given via $\textbf{R}\Gamma(G_{K_{\mathfrak{q}}},F_\lambda)\ra \textbf{R}\Gamma(G_{K_{\mathfrak{q}}},D_p)$. We shall henceforth call these local conditions the $\lambda$-Pottharst local condition.
\item $\Delta_{0,\lambda}$\,: Strict condition at $\fp$ and $\lambda$-Pottharst local condition at $\fp^c$.
\end{itemize}
At primes $\frak{l} \in \Sigma$ of $K$ that are coprime to $p$, each one of the four Selmer structures $\Delta_{\invques,?}$ requires the unramified local condition. These are given by 
\[
\Delta_{\invques,?,\frak{l}}^+(\TT^\dagger)= \left [(\TT^{\dagger})^{\mathscr{I}_{\frak{l}}} \xrightarrow{\textup{Fr}_\frak{l}-1} (\TT^{\dagger})^{\mathscr{I}_{\frak{l}}}\right ] \stackrel{\iota_{\frak{l}}^+}{\lra}C^{\bullet}(G_{\frak{l}},\TT^\dagger),
\]
where the terms in the complex on the left are placed in degrees $0$ and $1$ and the morphism $\iota_\frak{l}^+$ is given by 
$$\iota_{\frak{l}}^+(x)= x  \,\,\hbox{in degree 0\,\,\,\, and\,\,\, } (\iota_{\frak{l}}^+(x))(\textup{Fr}_{\frak{l}})=x \,\,\,\textup{in degree } 1.$$
We define the \emph{singular cone} at $\frak{l}$ by setting 
$$\widetilde{\Delta}_{\invques,?,\frak{l}}^-(\TT^\dagger):=\textup{cone}\left(\Delta_{\invques,?,\frak{l}}^+(\TT^\dagger)\stackrel{-i_\frak{l}^+}{\lra}C^{\bullet}(G_{\frak{l}},\TT^\dagger)\right)$$
and define $\widetilde{\Delta}_{\Sigma}(\TT^\dagger):={\displaystyle\bigoplus_{\substack{\frak{l}\in \Sigma\\ \frak{l}\nmid p}}\widetilde{\Delta}_{\invques,?,\frak{l}}^+(\TT^\dagger)}$.

\begin{defn}
\label{define:Selmerobjects}
Consider the following $\LL_\infty$-modules:
\begin{enumerate}[1.]
\item $H^1_{\emptyset,0}(K,\TT^\dagger):=\ker\left(H^1_\an(K,\TT^\dagger)\lra H^1_\an(K_{\fp^c},\TT^\dagger)\oplus H^1(\widetilde{\Delta}_{\Sigma}(\TT^\dagger))\right)$.
\item $H^1_{\emptyset,\lambda}(K,\TT^\dagger):=\ker\left(\han^1(K,\TT^\dagger)\lra \han^1(K_{\fp^c},D_p/F_\lambda)\oplus H^1(\widetilde{\Delta}_{\Sigma}(\TT^\dagger))\right)$.
\item $H^1_{\lambda,\lambda}(K,\TT^\dagger):=\ker\left(\han^1(K,\TT^\dagger)\lra\han^1(K_\fp,D_p/F_\lambda)\oplus \han^1(K_{\fp^c},D_p/F_\lambda)\oplus\, H^1(\widetilde{\Delta}_\Sigma(\TT^\dagger)\right)$.
\item $H^1_{0,\lambda}(K,\TT^\dagger):=\ker\left(\han^1(K,\TT^\dagger)\lra\han^1(K_\fp,D_p)\oplus \han^1(K_{\fp^c},D_p/F_\lambda)\oplus H^1(\widetilde{\Delta}_{\Sigma}(\TT^\dagger)\right)$.
\end{enumerate}
\end{defn}
\begin{proposition}
\label{prop:comparisonofselmerobjects}
We have the following isomorphisms of  $\LL_\infty$-modules.
\begin{enumerate}[i.]
\item $\widetilde{\textbf{R}\Gamma}^1(G_{K,S},\Delta_{\emptyset,0},\TT^\dagger)\stackrel{\sim}{\lra}H^1_{\emptyset,0}(K,\TT^\dagger)$\,.
\item$\widetilde{\textbf{R}\Gamma}^1(G_{K,S},\Delta_{\emptyset,\lambda},\TT^\dagger)\stackrel{\sim}{\lra}H^1_{\emptyset,\lambda}(K,\TT^\dagger)$\,. 
\item $\widetilde{\textbf{R}\Gamma}^1(G_{K,S},\Delta_{\lambda,\lambda},\TT^\dagger)\stackrel{\sim}{\lra}H^1_{\lambda,\lambda}(K,\TT^\dagger)$\,.
\item $\widetilde{\textbf{R}\Gamma}^1(G_{K,S},\Delta_{0,\lambda},\TT^\dagger)\stackrel{\sim}{\lra}H^1_{0,\lambda}(K,\TT^\dagger)$.
\end{enumerate}
\end{proposition}
\begin{proof}
All four isomorphisms follow from the long exact sequence arising from the definition of the Selmer complex as a mapping cone, since $H^0_\an(K_{\mathfrak{q}},\square)$ vanishes for $\mathfrak{q}=\fp, \fp^c$ and for $\square=D_p\,,\, D_p/F_\lambda$.
\end{proof}
We shall denote the degree-two cohomology $\widetilde{\textbf{R}\Gamma}^2(K,\Delta_{\invques,?},\TT^\dagger)$ by the symbol $\mathfrak{X}_{\invques,?}^\an$ in order to remain consistent with our notation in Section~\ref{sec:boundsonselmergroups}. 
\begin{theorem}
\label{thm:vanishingcomingfromclassicalobjects}
If the hypotheses of Proposition~\ref{prop:lfourcorrect} hold true, then $H^1_{\emptyset,0}(K,\TT^\dagger)=0$\, and\, $\mathfrak{X}_{\emptyset,0}^\an$ is a torsion $\LL_\infty$-module.
\end{theorem}
\begin{proof}
The first part follows from the corresponding result (Theorem~\ref{thm:leopoldtconjecturesforTfchi}(ii)) for the Selmer group associated to the Selmer structure $\FFF_{\emptyset,0}$ and \cite[Theorem 1.9]{jayanalyticfamilies}, which allows us to compare classical Iwasawa cohomology and  analytic Iwasawa cohomology groups. The Poitou-Tate global duality exact sequence for classical Selmer groups yields the exact sequence
$$0\lra H^1(K,\TT_{f,\chi}^\ac)\lra H^1(K_{\fp^c},\TT_{f,\chi}^\ac)\lra \mathfrak{X}_{\emptyset,0}^\ac$$
thanks to Theorem~\ref{thm:leopoldtconjecturesforTfchi}(ii). The same result also ensures that $H^1(K,\TT_{f,\chi}^\ac)$ has rank two (this is in fact an alternative way of phrasing the weak Leopoldt conjecture for $T_{f,\chi}$) and that $ \mathfrak{X}_{\emptyset,0}^\ac$ is torsion. By the comparison result \cite[Theorem 1.9]{jayanalyticfamilies} of Pottharst, it follows that the $\LL_\infty$-modules $H^1_\an(K, \TT^\dagger)$ and $H^1_\an(K_{\fp^c}, \TT^\dagger)$ have rank two and that we have an exact sequence
$$0\lra H^1_\an(K,\TT^\dagger)\lra H^1_\an(K_{\fp^c},\TT^\dagger)\lra \mathfrak{X}_{\emptyset,0}^\an$$
arising from the definition of the Selmer complex $\widetilde{\textbf{R}\Gamma}(G_{K,S},\Delta_{\emptyset,0},\TT^\dagger)$ as a mapping cone. This gives $H^1_{\emptyset,0}(K,\TT^\dagger)=0$. The rest is immediate after our comparison result (Proposition~\ref{prop:compareclassicalwithanalytic} below) and Theorem~\ref{thm:leopoldtconjecturesforTfchi}(ii).
\end{proof}

Recall from \S\ref{subsec:CMhidafamilies} that we have defined   the central critical Beilinson-Flach element $c^\lambda_\ac \in H^1(K,\TT^\dagger)$. 
We recall from Proposition~\ref{prop:sameprojection} that $c^\lambda_\ac \in H^1_{\emptyset,\lambda}(K,\TT^\dagger)$. Furthermore, under the hypotheses of Proposition~\ref{prop:lfourcorrect}, $c^\lambda_\ac \neq 0$.
\begin{theorem}
\label{thm:globaldualitysequences}
The following three sequences are exact:
\begin{equation}
\label{eqn:longexactglobal1}
0\lra \frac{H^1_{\emptyset,\lambda}(K,\TT^\dagger)}{H^1_{\emptyset,0}(K,\TT^\dagger)}\stackrel{\res_{\fp^c}}{\lra} \han^1(K_{\fp^c},F_\lambda)\lra\mathfrak{X}_{\emptyset,0}^\an\lra\mathfrak{X}_{\emptyset,\lambda}^\an\lra0
\end{equation}
\begin{equation}
\label{eqn:longexactglobal2}
0\ra \frac{H^1_{\emptyset,\lambda}(K,\TT^\dagger)}{H^1_{\lambda,\lambda}(K,\TT^\dagger)}\stackrel{\res_{\fp/\lambda}}{\lra} \han^1(K_{\fp},D_p/F_\lambda)\lra\mathfrak{X}_{\lambda,\lambda}^\an\lra\mathfrak{X}_{\emptyset,\lambda}^\an\ra0
\end{equation}
\begin{equation}
\label{eqn:longexactglobal3}
0\lra \frac{H^1_{\lambda,\lambda}(K,\TT^\dagger)}{H^1_{0,\lambda}(K,\TT^\dagger)}\stackrel{\res_\fp}{\lra} \han^1(K_{\fp},F_\lambda)\lra\mathfrak{X}_{0,\lambda}^\an\lra\mathfrak{X}_{\lambda,\lambda}^\an\lra0
\end{equation}
\end{theorem}
\begin{proof}
We shall only explain the exactness of (\ref{eqn:longexactglobal1}) as the rest follows in a similar fashion. The definitions of the two Selmer complexes $\widetilde{\textbf{R}\Gamma}(G_{K,S},\Delta_{\emptyset,\lambda},\TT^\dagger)$ and $\widetilde{\textbf{R}\Gamma}(G_{K,S},\Delta_{\emptyset,0},\TT^\dagger)$ together with the natural morphism $\Delta_{\emptyset,0}\lra \Delta_{\emptyset,\lambda}$ of local conditions yield the commutative diagram
$$\xymatrix{0\ar[r]&{\displaystyle\frac{H^1_{\an}\left(K,\TT^\dagger\right)}{H^1_{\emptyset,0}(K,\TT^\dagger)}}\ar[r]^{\res_{\fp^c}}\ar@{->>}[d]& {H^1(K_{\fp^c},D_p)}\ar[r]\ar@{->>}[d]_{\pi_{\fp^c/\lambda}}&\mathfrak{X}_{\emptyset,0}^\an\ar[r]\ar[d]\ar[d]_{\iota^{(2)}}&H^2_\an(K,\TT^\dagger)\ar[d]^{\cong}\ar[r]&0\\
0\ar[r]&{\displaystyle\frac{H^1_{\an}\left(K,\TT^\dagger\right)}{H^1_{\emptyset,\lambda}(K,\TT^\dagger)}}\ar[r]^(.47){\res_{\fp^c/\lambda}}& {H^1(K_{\fp^c},D_p/F_\lambda)}\ar[r]&\mathfrak{X}_{\emptyset,\lambda}^\an\ar[r]&H^2_\an(K,\TT^\dagger)\ar[r]&0
}$$
The surjectivity on the very right on each row follows from the proof of \cite[Corollary 6.17]{jaycyclotmotives}, which generalizes verbatim to our setting to verify that $\han^2(K_p,D_p)=\han^2(K_p,D_p/F_\lambda)=0.$
It follows from the Snake Lemma that we have an isomorphism 
\begin{equation}\label{eqn:snakelemma1}
\frac{H^1(K_{\fp^c},D_p)}{\res_{\fp^c}\left(\frac{H^1_{\an}\left(K,\TT^\dagger\right)}{H^1_{\emptyset,0}(K,\TT^\dagger)}\right)}\stackrel{\sim}{\lra}\ker(\iota^{(2)})
\end{equation}
as well as that $\iota^{(2)}$ is surjective and therefore that
\begin{equation}\label{eqn:snakelemma2}
0\lra\ker(\iota^{(2)})\lra \mathfrak{X}_{\emptyset,0}^\an\stackrel{\iota^{(2)}}{\lra} \mathfrak{X}_{\emptyset,\lambda}^\an\lra0\,.
\end{equation}
The exactness of (\ref{eqn:longexactglobal1}) follows on combining (\ref{eqn:snakelemma1}) and (\ref{eqn:snakelemma2}).
\end{proof}
\begin{defn}
\label{def:lazardandcharideal}
Given a co-admissible torsion $\LL_\infty$-module $M$ (in the sense of \cite{schneiderteitelbaum2003}), Lazard's structure theorem in \cite{lazardPIHES} provides us with an isomorphism 
$$M \cong \prod_{\nu \in I} \LL_\infty/\PP^{n_\nu}_\nu,$$
where $\{\PP_\nu\}_{\nu\in I}$ is a collection of primes that correspond to closed points in the unit disc that only accumulates towards the boundary. In particular, the coadmissible torsion $\LL_\infty$-modules that arise from the base change of finitely generated $\LL_{\cO_L}(\Gamma^\ac)$-modules are precisely those with finite support.
There principal ideal generated by an element in  $ \LL_\infty$ whose divisor equals $\sum_{\nu \in I} n_\nu\cdot\PP_\nu$ is called the characteristic ideal of $M$ and denoted by $\textup{char}_{\LL_\infty}(M)$.
When $M$ is a $\LL_{\cO_L}(\Gamma^\ac)$-module, we write $\textup{char}_{\LL_\infty}(M)$ for the characteristic ideal of its base change, namely:
$$\textup{char}_{\LL_\infty}(M\widehat{\otimes}\LL_\infty)=\textup{char}_{\LL_{\cO_L}(\Gamma^\ac)}(M)\,\widehat{\otimes}\,\LL_\infty.$$ 
\end{defn}

\begin{proposition}
\label{prop:compareclassicalwithanalytic}
$\textup{char}_{\LL_\infty}\left(\mathfrak{X}_{\emptyset,0}^\an\right)=\textup{char}_{\LL_{\cO_L}(\Gamma^\ac)}\left(\mathfrak{X}_{\emptyset,0}^{\ac,\iota}\right)\widehat{\otimes}\LL_\infty\,.$
\end{proposition}
\begin{proof}
Nekov\'a\v{r}'s work \cite{nekovar06} gives us a Greenberg-ordinary Selmer complex $\widetilde{\mathbf{R}\Gamma}(G_{K,S},\Delta_{\emptyset,0},\TT^\ac_{f,\chi^{-1}})$  with local conditions associated to $\Delta_{\emptyset,0}$, which are given in the same manner as their analytic counterpart (so that it determines  the relaxed condition at $\fp$ and the strict condition at $\fp^c$). By \cite[Theorem 1.9]{jayanalyticfamilies} (see also the discussion in \S 6.2 of \textit{op. cit.}),  the natural map
$$\widetilde{\mathbf{R}\Gamma}_{\textup{f},\Iw}(K^\ac/K,\Delta_{\emptyset,0},T_{f,\chi})\stackrel{\textbf{L}}{\otimes}\LL_\infty{\lra} \widetilde{\mathbf{R}\Gamma}(G_{K,S},\Delta_{\emptyset,0},\TT^\dagger)$$
is an isomorphism. This in turn shows that 
$$\textup{char}_{\LL_\infty}\left(\mathfrak{X}_{\emptyset,0}^\an\right)=\textup{char}_{\LL_{\cO_L}(\Gamma^\ac)}\left(\widetilde{H}^2_{\textup{f},\Iw}(K^\ac/K,\Delta_{\emptyset,0},T_{f,\chi})\right)\widehat{\otimes}\LL_\infty$$
and our assertion would follow after verifying 
$$\widetilde{H}^2_{\textup{f},\Iw}(K^\ac/K,\Delta_{\emptyset,0},T_{f,\chi})^\iota\stackrel{?}{\cong} \mathfrak{X}_{\emptyset,0}^\ac:=H^1_{\FFF_{\emptyset,0}^*}(K,\TT_{f,\chi}^{\ac,*})^\vee,$$
where we follow Nekov\'a\v{r}'s notation for the Greenberg-ordinary extended Selmer groups.
Let $A_{f,\chi}=T_{f,\chi}\otimes \QQ_p/\ZZ_p$. By the self-duality of the $G_K$-representation $T_{f,\chi}$, we  have $A_{f,\chi}=\textup{Hom}(T_{f,\chi},\bbmu_{p^\infty})$, which amounts to saying that $A_{f,\chi}=D(T_{f,\chi})$ in Nekov\'a\v{r}'s notation. Since the functor $D_{\Lambda_{\cO_L}(\Gamma^\ac)}$ coincides with the Pontryagin dual functor (c.f., \S9.1.4 of \cite{nekovar06}) it follows from Nekov\'a\v{r}'s global duality \cite[Theorem 8.9.12]{nekovar06} that 
$$\widetilde{H}^2_{\textup{f},\Iw}(K^\ac/K,\Delta_{\emptyset,0},T_{f,\chi})^\iota\cong \widetilde{H}^1_{\textup{f}}(K_S/K^\ac, \Delta_{0,\emptyset}, A_{f,\chi})^\vee,$$
where $ \Delta_{0,\emptyset}$ is the dual local condition, which is obtained from $\Delta_{\emptyset,0}$ by switching the roles of $\fp$ and $\fp^c$. Notice also that the condition $\Delta_{0,\emptyset}$ reduces to the Selmer structure $\FFF_{\emptyset,0}^*$ on $A_{f,\chi}$, in the level of cohomology. It  remains to check that 
$$\widetilde{H}^1_{\textup{f}}(K_S/K^\ac, \Delta_{0,\emptyset}, A_{f,\chi})\stackrel{?}{\cong} H^1_{\FFF_{\emptyset,0}^*}(K,\TT_{f,\chi}^{\ac,*})\,.$$
By   \cite[Proposition 8.8.6, Lemma 9.6.3]{nekovar06}, along with the vanishing of $H^0(K_\mathfrak{q},\overline{T}_{f,\chi})$ for $\mathfrak{q}=\fp,\fp^c$, we have
\[\widetilde{H}^1_{\textup{f}}(K_S/K^\ac, \Delta_{0,\emptyset}, A_{f,\chi})=\varinjlim_n \widetilde{H}^1_{\textup{f}}(G_{K^\ac_n,S}, \Delta_{0,\emptyset}, A_{f,\chi})\stackrel{\sim}{\lra}\varinjlim_n H^1_{\FFF_{\emptyset,0}^*}(K^\ac_n,A_{f,\chi})\cong H^1_{\FFF_{\emptyset,0}^*}(K,\TT^{\ac,*}_{f,\chi}).
\]
This concludes our proof.
\end{proof}
\begin{defn}
\label{def:negligibleerrorRubin87}
Let $H^1_{\lambda}(K_{\fp^c},D_p)\subset H^1_\an(K_{\fp^c},D_p)$ denote the image of $H^1_\an(K_{\fp^c},F_\lambda)$. Notice that $H^1_{\alpha}(K_{\fp^c},D_p)\cap H^1_\beta(K_{\fp^c},D_p)=0$ and therefore the submodule 
$$\mathscr{H}^1_{\alpha,\beta}:=H^1_{\alpha}(K_{\fp^c},D_p)+ H^1_\beta(K_{\fp^c},D_p)\subset H^1_\an(K_{\fp^c},D_p)$$ 
has rank two. Let $\mathbf{Err}_{\alpha,\beta} \subset \LL_\infty$ denote the characteristic ideal of the quotient $H^1_\an(K_{\fp^c},D_p)/\mathscr{H}^1_{\alpha,\beta}$.
We similarly define the submodule 
$$\mathscr{H}^1_{\#,\flat}:=H^1_\#(K_{\fp^c},\TT^\ac_{f,\chi})+ H^1_\flat(K_{\fp^c},\TT^\ac_{f,\chi})\subset H^1(K_{\fp^c},\TT^\ac_{f,\chi})$$ 
and the ideal $\mathbf{Err}_{\#,\flat}\subset \LL_{\cO_L}(\Gamma^\ac)$ as the characteristic ideal of the torsion $ \LL_{\cO_L}(\Gamma^\ac)$-module $H^1(K_{\fp^c},\TT^\ac)/\mathscr{H}^1_{\#,\flat}$.
\end{defn}

\begin{lemma}
\label{lemma:transitionfactorsimplified}
We have the equality
\[\frac{\det(\Tw_{k/2-1}(Q^{-1}M_{\log,\gamma_\ac}))\cdot\det(\textup{Im}(\col^{\ac}_{\flat,\fp^c}))\cdot\mathbf{Err}_{\#,\flat}}{\det(\textup{Im}(\mathcal{L}_{\lambda,\fp^c}^\ac))\cdot \mathbf{Err}_{\alpha,\beta}}=\cH_L(\Gamma^\ac),\]
where $M_{\log,\gamma^\ac}$ denotes the matrix obtained from $\Mlog$ on replacing $\Gamma_0^\cyc$ by  a  topological generator of $\Gamma^\ac$.
\end{lemma}
\begin{proof}
As a consequence of Lemma~\ref{lem:det}, the determinant of $\Mlog$ equals, up to a unit, $\log_{p,k-1}(\Gamma_0^\cyc)/\delta_{k-1}(\Gamma^\cyc)$ (see also \cite[Corollary~3.2]{LLZ0.5}). Therefore, $\det(\Tw_{k/2-1}(Q^{-1}M_{\log,\gamma_\ac}))\sim\Tw_{k/2-1}(\log_{p,k-1}(\gamma_\ac)/\delta_{k-1}(\gamma_\ac))$. 
Lemma~\ref{lem:2varimage} together with Remark~\ref{rk:image} describe $\det(\textup{Im}(\col^{\ac}_{\flat,\fp^c}))$. The term $\mathbf{Err}_{\#,\flat}$ is given by Proposition~\ref{prop:errorterms}, whereas $\det(\textup{Im}(\mathcal{L}_{\lambda,\fp^c}^{\ac}))$ and $\mathbf{Err}_{\alpha,\beta}$ are given by Corollary~\ref{cor:imageLlambda} and Proposition~\ref{prop:errorL} respectively. Hence the result.
\end{proof}
\subsection{$p$-adic (cyclotomic) height pairings}
\label{sec:padicheights}
Let $\fm=\fm_\ac$ denote the maximal ideal of $\LL_{\cO_L}(\Gamma^\ac)$ (note that in  previous sections, we have used the same symbol to denote  moduli of $K$. As these moduli will no longer make an appearance, our notation here should cause no confusion).  For each positive integer $n$, we let $\mathcal{H}_n(\Gamma^\ac)$ denote the $p$-adic completion of $\Lambda_{\cO_L}[\fm^n/p]$, the $\Lambda_{\cO_L}(\Gamma^\ac)$-subalgebra of $\Lambda_{\cO_L}(\Gamma^\ac)[1/p]$ generated by all $r/p$ with $r\in \fm^n$. Each $\calHacn[1/p]$ is a strict affinoid algebra (and in fact, also a Euclidean domain). They give an admissible affinoid covering of $\textup{Sp}\,\LL_\infty$ and $\calHac=\varprojlim \calHacn[1/p]$. See \cite[\S1.7]{dejong1995PIHES} for details. 
Fix a positive integer $n$ and set $A:=\calHacn[1/p]$. As before, let $A^\iota$ denote the rank-one $A$-module that is endowed it with a $G_K$-action given via the maps 
$$G_K\twoheadrightarrow \Gamma^\ac\stackrel{\iota}{\hookrightarrow} \LLac^\times\lra \calHacn[1/p]^\times,$$
where  $\iota$ is the involution $\gamma\mapsto \gamma^{-1}$. We set $V=T_{f,\chi^{-1}}\otimes_{\cO_L} L$ and 
$$\mathscr{D}(V)=\textup{Hom}(V,L)(1) \stackrel{\sim}{\lra} T_{f,\chi}\otimes_{\cO_L} L.$$ Define $V_A:=V\otimes_{\cO_L} A^\iota$ and allow $G_K$ act diagonally. 
We likewise consider 
$$\mathscr{D}(V_A):=\textup{Hom}_A(V_A,A)(1)\cong T_{f,\chi}\otimes_{\cO_L}A \cong V_A^\iota,$$
where  $G_K$ now acts on the second factor of $T_{f,\chi}\otimes_{\cO_L}A$ without the involution $\iota$.
In order to introduce \emph{cyclotomic} $A$-adic height pairings afforded by the cyclotomic deformation of $V_A$, we will work below with cyclotomic $(\varphi,\Gamma)$-modules associated to $V_A$ over the relative Robba ring $\RR_A$ (as well as their triangulations). To that end, for $\mathfrak{q}\in\{\fp,\fp^c\}$, we let 
$$D_{\textup{rig},A}^\dagger(V_A\mid_{G_{K_\mathfrak{q}}})\cong D_{\textup{rig},L}^\dagger(V\vert_{G_{K_\mathfrak{q}}})\otimes_L  D_{\textup{rig},A}^\dagger(A^\iota\mid_{G_{K_\mathfrak{q}}})$$ 
denote the $(\varphi,\Gamma^\cyc)$-module of $V_A\mid_{G_{K_\mathfrak{q}}}$ over $\RR_A$. 
We write $D_\lambda^{(\mathfrak{q})}\subset D_{\textup{rig},L}^\dagger(V\vert_{G_{K_\mathfrak{q}}})$ for the  saturated $(\varphi,\Gamma^\cyc)$-submodules of rank one associated with the respective $\varphi$-eigenspaces of $\Dcris(V\vert_{G_{K_\mathfrak{q}}})$, as constructed by Berger. 
When there is no need to distinguish $K_{\mathfrak{q}}$ from $\QQ_p$, we shall drop $``\mid_{G_{K_{\mathfrak{q}}}}"$ from notation. As before, we shall simply write $D_\lambda$ in place of $D_\lambda^{(\mathfrak{q})}$ when the context makes our choice clear. 
We set $D_{\lambda,A}:=D_\lambda\widehat{\otimes}_L D_{\textup{rig},A}^\dagger(A^\iota) \cong D_\lambda\widehat{\otimes}_L \RR_A^\iota$. This is a saturated rank-one $\RR_A$-submodule of $D_{\textup{rig},A}^\dagger(V_A)$. We let 
$$D_{\lambda,A}^\perp:=\textup{Hom}_{\RR_A}(D_{\textup{rig},A}^\dagger(V_A)/D_{\lambda,A},\RR_A(\chi_0)) $$
denote the orthogonal complement of $D_{\lambda,A}$ under the canonical duality 
$$D_{\textup{rig},A}^\dagger(V_A)\times D_{\textup{rig},A}^\dagger(\mathscr{D}(V_A))\lra \RR_A(\chi_0)\,.$$

Notice that $V_A=\TT^\dagger\otimes_{\calHac} A$ and all we have recorded in the previous sections concerning the Selmer groups attached to $\TT^\dagger$ in fact hold for $A$ by base change. In particular, as we have already done so over the coefficient ring $\LL_\infty$ at the start of Section~\ref{subsec:analyticSelmergroups}, we may define four  analytic Selmer complexes $\widetilde{\textbf{R}\Gamma}(G_{K,S}, \Delta_{\emptyset,0},V_A)$, $\widetilde{\textbf{R}\Gamma}(G_{K,S},\Delta_{\emptyset,\lambda}V_A)$, $\widetilde{\textbf{R}\Gamma}(G_{K,S},\Delta_{\lambda,\lambda},V_A)$ and $\widetilde{\textbf{R}\Gamma}(G_{K,S},\Delta_{0,\lambda},V_A)$.
These are elements of the derived category of $A$-modules that may be represented by perfect complexes  in degree $[0,3]$, which are given via the local conditions $\Delta_{\invques,?}$ at primes above $p$. (For the local conditions, one simply makes use of $D_{\lambda,A}$ in place of $F_\lambda$, etc.) Their cohomology groups in degree one will be denoted by $H^1_{\invques,?}(K,V_A)$ and by $\mathfrak{X}^A_{\invques,?}$ in degree two. The local conditions $\Delta_{\lambda,\lambda}$ produces local conditions $\Delta_{\lambda,\lambda}^\perp$ on $\mathscr{D}(V_A)$; we shall denote the associated Selmer complex by $\widetilde{\textbf{R}\Gamma}(G_{K,S},\Delta_{\lambda,\lambda}^\perp,\mathscr{D}(V_A))$ and its cohomology in degree one by $H^1_{\lambda,\lambda}(K,\mathscr{D}(V_A))$\,.
We are now ready to apply the general formalism of $p$-adic heights developed in \cite{benoisheights}, (see also \cite[Section 3.1]{benoisbuyukboduk}) that generalizes the constructions of Nekov\'a\v{r} in the case eigenform $f$ is $p$-ordinary.
\begin{theorem}[Benois, Nekov\'a\v{r}]
\label{thm:p-adicheight}
There exists a $($cyclotomic$)$ $p$-adic height pairing
$$\langle\,,\,\rangle_{\lambda,\lambda}\,:\,H^1_{\lambda,\lambda}(K,V_A)\otimes H^1_{\lambda,\lambda}(K,\mathscr{D}(V_A)) \lra A,$$
which interpolates Nekov\'a\v{r}'s cyclotomic $p$-adic height pairing
{$$\langle\,,\,\rangle_{\lambda,\lambda}^{\rm Nek}\,:\,H^1_{\textup{f}}(K,V_{f,\chi}\otimes \eta)\otimes H^1_{\textup{f}}(K,V_{f,\chi^{-1}}\otimes \eta^{-1}) \lra {L(\eta)}$$
as $\eta$ runs through characters of $\Gamma^\ac$ of finite order $p^m$ with $m<n$. Here, we recall that the integer $n$ is fixed so that $A=\mathcal{H}_n(\Gamma^\ac)[1/p]$ and $L(\eta)$ stands for the extension of $L$ that contains ${\rm im}(\eta)$.}
\end{theorem}
\begin{proof}{
Let $\eta$ be a character of $\Gamma^\ac$ of finite order $p^m$. Let us set $ L_\eta:=\LL_{{L(\eta)}}(\Gamma^\ac)^\iota\big{/}(\eta(\gamma_\ac)\gamma_\ac-1)$ and consider the natural map 
$$\pi_\eta: \LL_{L}(\Gamma^\ac)^\iota\lra L_\eta$$
of $G_K$-modules (notice that $L_\eta$ is a one-dimensional ${L(\eta)}$-vector space on which $G_K$ acts via $\eta$). When $m<n$, the map $\pi_\eta$ extends to a $G_K$-equivariant map 
$$\pi_\eta: A^\iota \lra L_\eta$$
(and likewise, the evaluation map $\eta:\LL_{L}(\Gamma^\ac)\ra L(\eta)$ extends to $\eta: A\ra L(\eta)$) which in turn induces $G_K$-equivariant specialization maps
$$\pi_\eta: V_A\lra V_{f,\chi}\otimes\eta \,\,\,\,\,\, \,\,\,  \hbox{and }\,\,\,\,\,\,\, \,\,\, \pi_\eta: \mathscr{D}(V_A)\lra V_{f,\chi^{-1}}\otimes\eta^{-1}\,.$$
These give rise to maps 
$$\pi_\lambda: H^1_{\lambda,\lambda}(K,V_A)\lra H^1_{\textup{f}}(K,V_{f,\chi} )\,\,\,\,\,\, \,\,\,  \hbox{and }\,\,\,\,\,\,\, \,\,\, \pi_\eta: H^1_{\lambda,\lambda}(K,\mathscr{D}(V_A))\lra H^1_{\textup{f}}(K,V_{f,\chi^{-1}}\otimes\eta^{-1})$$
thanks to \cite[\S3D]{jayanalyticfamilies}; see also \cite[Theorem 4.1.2]{jaycyclotmotives}. The commutativity of the diagram
$$\xymatrix@C=.1cm{H^1_{\lambda,\lambda}(K,V_A)\ar[d]_{\pi_\eta}&\otimes &H^1_{\lambda,\lambda}(K,\mathscr{D}(V_A))\ar[d]_{\pi_\eta} \ar[rrrrr]^(.67){\langle\,,\,\rangle_{\lambda,\lambda}}&&&&&A\ar[d]^{\eta}\\
H^1_{\textup{f}}(K,V_{f,\chi}\otimes \eta)&\otimes& H^1_{\textup{f}}(K,V_{f,\chi^{-1}}\otimes \eta^{-1})\ar[rrrrr]^(.67){\langle\,,\,\rangle_{\lambda,\lambda}^{\rm Nek}}&&&&&L(\eta)
}$$
follows from the functorial construction (due to Benois and Nekov\'a\v{r}) of the height pairings $\langle\,,\,\rangle_{\lambda,\lambda}$ and $\langle\,,\,\rangle_{\lambda,\lambda}^{\rm Nek}$ in terms of Bockstein morphisms and global duality for Selmer complexes.}
\end{proof}
\begin{defn}
\label{def:regulator}
We define the $p$-adic regulator $\mathscr{R}_{p}$ to be the Fitting ideal of the cokernel of  $\langle\,,\,\rangle_{\lambda,\lambda}$. For $c\in H^1_{\lambda,\lambda}(K,V_A)$, we  set $\mathscr{R}_p(c)$ to be the Fitting ideal of the cokernel of the map 
\begin{align*}
H^1_{\lambda,\lambda}(K,V_A)&\lra A\\
y&\mapsto\langle c,y\rangle_{\lambda,\lambda}\,\,.
\end{align*}
\end{defn}
\begin{defn}
We denote the $p$-adic local Tate (cup product) pairing by 
$$\langle\,,\,\rangle_{\lambda,\lambda}^{(\fp)}: H^1_\an(K_{\fp},D_{\textup{rig}}^\dagger(V_A)/D_{\lambda,A})\otimes H^1_\an(K_{\fp},D_{\lambda,A}^\perp)\lra A\,.$$
For $c\in H^1_{\an}(K_{\fp},D_{\textup{rig}}^\dagger(V_A)/D_{\lambda,A})$, we  define $\mathscr{T}_\fp(c)\subset A$ to be the Fitting ideal of the cokernel of the map 
\begin{align*}
H^1_{\an}(K_{\fp},D_{\lambda,A}^\perp)&\lra A\\
y&\mapsto\langle c,y\rangle_{\lambda,\lambda}^{(\fp)}\,\,.
\end{align*}
\end{defn}
\subsection{Beilinson-Flach elements and bounds on analytic Selmer groups.}
\label{sec:analyticBF}

\begin{proposition}
\label{prop:BFelementsalpha}
With the conventions in Definitions~\ref{def:lazardandcharideal} and \ref{def:negligibleerrorRubin87},
$$\textup{char}_{\LL_\infty}\left(\frac{H^1(K_{\fp^c},F_\lambda)}{\LL_\infty\cdot\res_{\fp^c}(c^\lambda_\ac)}\right)=\textup{char}_{\LL_\infty}\left(\frac{H^1_\#(K_{\fp^c},\TT^{\ac}_{f,\chi})}{\LL_{\cO_L}(\Gamma^\ac)\cdot\res_{\fp^c}(c^\#_\ac)}\right)\,.$$
\end{proposition}
\begin{proof}
We take $\lambda$ to be $\alpha$ in this proof. The left-hand side of the equation equals
\begin{align} \frac{\col^\ac_{\flat,\fp^c}(c^\#_\ac)}{\det\left(\col_{\flat,\fp^c}^\ac(H^1_\#(K_{\fp^c},\TT^{\ac}_{f,\chi}))\right)}
\notag&=\frac{\col^{\ac}_{\flat,\fp^c}(c_\#^\ac)}{\det(\textup{Im}(\col^{\ac}_{\flat,\fp^c}))}\cdot\det\left(\frac{H^1(K_{\fp^c},\TT^\ac_{f,\chi})}{\mathscr{H}^1_{\#,\flat}}\right)^{-1}\\
\notag&=\frac{\col^{\ac}_{\flat,\fp^c}(c^\#_\ac)}{\det(\textup{Im}(\col^{\ac}_{\flat,\fp^c}))}\cdot\mathbf{Err}_{\#,\flat}^{-1}\\
\label{eqn:comparisonofcolemanvspr1}&=\frac{\mathcal{L}_{\beta,\fp^c}^{\ac}(c^\alpha_\ac)\cdot \mathbf{Err}_{\#,\flat}^{-1}}{\det(\textup{Im}(\col^{\ac}_{\#,\fp^c}))\cdot\det(\Tw_{k/2-1}(Q^{-1}M_{\log,\gamma_\ac}))},
\end{align}
where the last equality follows from Corollary~\ref{cor:ColBF}. Likewise, 
\begin{equation}\label{eqn:comparisonflaralpha1}
\textup{char}_{\LL_\infty}\left(\frac{H^1(K_{\fp^c},F_\alpha)}{\LL_\infty\cdot\res_{\fp^c}(c^\alpha_\ac)}\right)=\frac{\mathcal{L}_{\beta,\fp^c}^{\ac}(c^\alpha_\ac)}{\det(\textup{Im}(\mathcal{L}_{\beta,\fp^c}^{\ac}))}\cdot\mathbf{Err}_{\alpha,\beta}^{-1}
\end{equation}
and the proof follows comparing \eqref{eqn:comparisonofcolemanvspr1} with \eqref{eqn:comparisonflaralpha1} and applying Lemma~\ref{lemma:transitionfactorsimplified}.
\end{proof}
\begin{corollary}
\label{cor:analyticmainconj0alpha} Assuming the validity of the hypotheses of Proposition~\ref{prop:lfourcorrect}, we have the divisibility 
$$\textup{char}_{\LL_\infty}\left(\mathfrak{X}_{\emptyset,\lambda}^{\rm an}\right)\mid\textup{char}_{\LL_\infty}\left(H^1_{\emptyset,\lambda}(K,\TT^\dagger)\big{/}\LL_\infty\cdot c^\lambda_\ac\right)$$
of ideals in $\LL_\infty$, which is an equality if and only if the divisibility in Theorem~\ref{thm:integralBFundersignedlog} is an equality.
\end{corollary}
\begin{proof}
Combining (\ref{eqn:longexactglobal1}) and Theorem~\ref{thm:vanishingcomingfromclassicalobjects}, it follows  that the sequence 
\begin{equation}
\label{eqn:fourtermexactemptyzeroemptyalpha}0\lra {H^1_{\emptyset,\lambda}(K,\TT^\dagger)}\stackrel{\res_{\fp^c}}{\lra} \han^1(K_{\fp^c},F_\lambda)\lra\mathfrak{X}^\an_{\emptyset,0}\lra\mathfrak{X}^\an_{\emptyset,\lambda}\lra0
\end{equation}
is exact. Our claim follows combining this with Theorem~\ref{thm:integralBFundersignedlog} and Proposition~\ref{prop:BFelementsalpha}.
\end{proof}
\begin{corollary}
\label{cor:whenBFelementsarenonzeroAnalytic1}
If the hypotheses of Proposition~\ref{prop:lfourcorrect} are valid, then the $\cH(\Gamma^\ac)$-module $H^1_{\emptyset,\lambda}(K,\TT^\dagger)$ has rank one and $\mathfrak{X}_{\emptyset,\lambda}^{\rm an}$ is torsion. \end{corollary}
\begin{proof}
The 4-term exact sequence (\ref{eqn:fourtermexactemptyzeroemptyalpha}) shows that the rank of the  
 $\LL_\infty$-module $H^1_{\emptyset,\lambda}(K,\TT^\dagger)$ is at most one. {Combining Propositions~\ref{prop:lfourcorrect} and~\ref{prop:BFelementsalpha}, we see that} the class $c^\lambda_\ac \in H^1_{\emptyset,\lambda}(K,\TT^\dagger)$ is non-zero. {It follows that the rank of $H^1_{\emptyset,\lambda}(K,\TT^\dagger)$} is precisely one. In this case, the same exact sequence proves also the second assertion.
\end{proof}
We recall the property $\mathbf{(L.\lambda)}$ concerning the local position of the class $c^\lambda_\ac$ at the prime $\fp$. Notice that the property $\mathbf{(L.\lambda)}$ is equivalent to the requirement that 
\begin{equation}\label{eqn_BF_ac_lcoal_position_indefinite}
\res_\fp(c^\lambda_\ac) \in \ker\left(H^1_\an(K_{\fp},D_p)\stackrel{\pi_{\fp/\lambda}}{\lra} H^1_\an(K_{\fp},D_p/F_\lambda)\right).
\end{equation}
\begin{corollary}
\label{cor:comparisonofBKselmerwithrelaxedBKintheindefinitecase} 
Suppose that the hypotheses of Proposition~\ref{prop:lfourcorrect} hold true. Then we have $H^1_{\lambda,\lambda}(K,\TT^\dagger)=H^1_{\emptyset,\lambda}(K,\TT^\dagger)$ and $\mathfrak{X}_{\lambda,\lambda}^{\rm an}$ has rank one. Furthermore, the sequence
\begin{equation}
\label{eqn:longexactglobal2degenerated}
0\lra \han^1(K_{\fp},D_p/F_\lambda)\lra\mathfrak{X}_{\lambda,\lambda}^\an\lra\mathfrak{X}_{\emptyset,\lambda}^\an\lra0
\end{equation}
is exact.
\end{corollary}
\begin{proof}
Under the running hypotheses, it follows from Proposition~\ref{prop:lfourcorrect} and Corollary~\ref{cor:whenBFelementsarenonzeroAnalytic1} that the quotient $H^1_{\emptyset,\lambda}(K,\TT^\dagger)/H^1_{\lambda,\lambda}(K,\TT^\dagger)$ is torsion. On the other hand, this torsion module injects via $\res_{\fp/\lambda}$ into $H^1_\an(K_\fp,D_p/F_\lambda)$, which is torsion-free (of rank one). This proves that  $H^1_{\emptyset,\lambda}(K,\TT^\dagger)=H^1_{\lambda,\lambda}(K,\TT^\dagger)$. The second assertion and the exactness of (\ref{eqn:longexactglobal2degenerated}) easily follow from the exactness of (\ref{eqn:longexactglobal2}). 
\end{proof}
For the rest of the article, we assume the validity of the hypotheses of Proposition~\ref{prop:lfourcorrect} (namely that $p\nmid N_f\rm{disc}(K/\QQ)$ and $N^-$ is a square-free product of an even number of primes). Corollary~\ref{cor:comparisonofBKselmerwithrelaxedBKintheindefinitecase} therefore applies and shows that the $\cH(\Gamma^\ac)$-module $\mathfrak{X}_{\lambda,\lambda}^{\rm an}$ has rank one. 
\begin{defn}
Following Pottharst, we set 
$$\mathscr{D}^iM=\begin{cases} \textup{Hom}_{\LL_\infty}(M/M_{\textup{tor}},\LL_\infty),& \textup{ if } i=0,\\
\prod_{\lambda\in I} \mathfrak{P}_\nu^{-n_\nu}/\LL_\infty,& \textup{ if } i=1\,
\end{cases}$$
for a coadmissible $\LL_\infty$-module $M$ with $M_{\textup{tor}}\stackrel{\sim}{\lra} \prod_{\nu\in I} \LL_\infty/\mathfrak{P}_\nu^{n_\nu}$\,.
\end{defn}
Note that $\textup{char}_{\LL_\infty}\left(M_{\textup{tor}}\right)=\textup{char}_{\LL_\infty}\left(\mathscr{D}^1M\right).$
\begin{theorem}
\label{thm:Hzeroalphavanishesandcomparinftorsion}
\begin{enumerate}[i.]
\item We have $H^1_{0,\lambda}(K,\TT^\dagger)=0$ and the sequence
\begin{equation}
\label{eqn:longexactglobal3degenerated}
0\lra {H^1_{\lambda,\lambda}(K,\TT^\dagger)}\stackrel{\res_\fp}{\lra} \han^1(K_{\fp},F_\lambda)\lra\mathfrak{X}_{0,\lambda}^\an\lra\mathfrak{X}_{\lambda,\lambda}^\an\lra0
\end{equation}
is exact.
\item There exists a canonical isomorphism $\mathscr{D}^1\mathfrak{X}_{0,\lambda}^\an \stackrel{\sim}{\lra}\mathfrak{X}^\an_{\emptyset,\lambda}\,.$
\item We have 
$$\textup{char}_{\LL_\infty}\left(\mathscr{D}^1\mathfrak{X}^\an_{\lambda,\lambda}\right)=\textup{char}_{\LL_\infty}\left(\mathscr{D}^1\mathfrak{X}^\an_{0,\lambda}\right)\cdot \mathscr{J}_\fp\,,$$ 
where $\mathscr{J}_\fp$ is the characteristic ideal of 
$\textup{coker}\left({H^1_{\lambda,\lambda}(K,\TT^\dagger)}\stackrel{\res_\fp}{{\hookrightarrow}} \han^1(K_{\fp},F_\lambda)\right)$
in the sequence (\ref{eqn:longexactglobal3degenerated}).
\end{enumerate}
\end{theorem}
\begin{proof}
The generalization of Greenberg's duality theorem due to Porttharst in \cite[Theorem 4.1(3)]{jaycyclotmotives} (which we slightly extend here in order to apply it with the Selmer complexes we have introduced at the start of Section~\ref{subsec:analyticSelmergroups}, and make further use of the fact that $T_{f,\chi^{-1}}$ is self-dual) yields an exact sequence 
\begin{equation}
\label{eqn:greenbergduality}
0\lra \mathscr{D}^1\mathfrak{X}_{0,\lambda}^\an\lra \mathfrak{X}^\an_{\emptyset,\lambda}\lra \mathscr{D}^0H^1_{0,\lambda}(K,\TT^\dagger)\lra0\,.
\end{equation}
Since the module $\mathfrak{X}^\an_{\emptyset,\lambda}$ is torsion by Corollary~\ref{cor:whenBFelementsarenonzeroAnalytic1}, the vanishing of the module $\mathscr{D}^0H^1_{0,\lambda}(K,\TT^\dagger)$ follows. But the module $H^1_{0,\lambda}(K,\TT^\dagger) \subset H^1_{\emptyset,\lambda}(K,\TT^\dagger)$ is torsion-free and the first assertion in (i) follows. The exactness of the sequence (\ref{eqn:longexactglobal3degenerated}) is now a direct consequence of the exactness of (\ref{eqn:longexactglobal3}), whereas (ii) follows from (i) used with the exact sequence (\ref{eqn:greenbergduality}). Finally, (iii) follows from (\ref{eqn:longexactglobal3degenerated}).
\end{proof}
\subsection{The cyclotomic deformation and Rubin-style formula}
\label{cyclodeformRubin}
Recall that  the hypotheses of Proposition~\ref{prop:lfourcorrect} are in effect. This in particular means that the Beilinson-Flach class $c_\lambda^\ac$ is a non-trivial element of $H^1_{\lambda,\lambda}(K,\TT^\dagger)$. 
We write  $c^\lambda_A\in H^1_{\lambda,\lambda}(K,V_A)$ for the image of $c^\lambda_\ac$ under the map induced from the {natural map} $\calHac \ra \calHacn=A$. 
In what follows, we shall explain how to incorporate the cyclotomic deformation in this picture. We shall mostly follow \cite[Section 2.4 and 2.7]{benoisbuyukboduk} and also adopt the notation of \textit{loc. cit.} Set $\overline{V}_A:=V_A\widehat{\otimes}_L\,\calHcyc^{\iota}$ and for any $(\varphi,\Gamma^\cyc)$-module $\mathbb{D}$ over $\RR_A$, we set $\overline{\mathbb{D}}:=\mathbb{D}\,\widehat{\otimes}\,D_{\textup{rig}}^\dagger(\calHcyc^\iota)$. We may then define (starting off with $\overline{D}_{\lambda,A}\subset \overline{D}_{\textup{rig}}^\dagger(V_A)$) the Iwasawa theoretic Selmer complex 
$$\textbf{R}\Gamma_\Iw(K^\cyc/K,\Delta_{\emptyset,\lambda},V_A):=\textbf{R}\Gamma(G_{K,S},\overline{\Delta}_{\emptyset,\lambda},\overline{V}_A)$$
where $\overline{\Delta}_{\emptyset,\lambda}$ stands for the local conditions determined by the $(\varphi,\Gamma^\cyc)$-module $ \overline{D}_{\textup{rig}}^\dagger(V_A)$ at $\fp$ and by $\overline{D}_{\lambda,A}$ at $\fp^c$. We use similar conventions to what we have adopted before for the cohomology of this object.
\begin{defn}
Let $\bar{c}^\lambda \in H^1(K,\overline{V}_A)$ denote the central critical Beilinson-Flach element associated to the $\lambda$-stabilization of $f$, induced from the analytification morphisms $\Lambda^\ac \ra A$ and $\Lambda^\cyc\ra \calHcyc$. 
\end{defn}
Note that $\bc^\lambda \in H^1_{\emptyset,\lambda}(K,\overline{V}_A)$ by Proposition~\ref{prop:sameprojection}. We write $\pi_A$ for the augmentation morphism $A\widehat{\otimes}\calHcyc\stackrel{\pi_A}{\lra} A$ induced from $\gamma_\cyc\mapsto 1$. This map in turn induces a map $H^1_{\emptyset,\lambda}(K,\overline{V}_A) \stackrel{\pi_A}{\lra}H^1_{\emptyset,\lambda}(K,{V}_A)$ which we denote by the same symbol. Note that $\pi_A(\bc^\lambda)=c^\lambda_A$ and likewise, $\pi_A\left(\res_{\mathfrak{q}}(\bc^\lambda)\right)=\res_\q(c^\lambda_A)$ for each $\q\in \{\fp,\fp^c\}$.
\begin{defn}
Let $\pi_{/\lambda}$ denote the natural morphism
$$\pi_{/\lambda}\,:\,H^1_\an(K_{\fp}, \overline{D}_{\textup{rig}}^\dagger(V_A))\lra H^1_\an(K_{\fp}, \overline{D}_{\textup{rig}}^\dagger(V_A)/\overline{D}_{\lambda,A})\,.$$
We similarly denote by $\pi^A_{/\lambda}$ the morphism 
$$\pi^A_{/\lambda}\,:\,H^1_\an(K_{\fp}, {D}_{\textup{rig}}^\dagger(V_A))\lra H^1_\an(K_{\fp}, {D}_{\textup{rig}}^\dagger(V_A)/{D}_{\lambda,A}).$$
\end{defn}
Note that $\pi_A\circ \pi_{/\lambda}=\pi_{/\lambda}\circ\pi_A=\pi^A_{/\lambda}$ and that $\pi_{/\lambda}^A(\res_\fp(c^\lambda_A))=0$ by Proposition~\ref{prop:Lalpha}.
\begin{proposition}
\label{prop:derivativeofBFclass}
There exists a unique element $\partial_\fp \bc^\lambda \in H^1_\an(K_{\fp}, \overline{D}_{\textup{rig}}^\dagger(V_A)/\overline{D}_{\lambda,A})$ such that $$\pi_{/\lambda}\left(\res_{\fp}(\bc^{\lambda})\right)=\frac{\gamma_\cyc-1}{\log_p\chi_0(\gamma_\cyc)}\cdot \partial_\fp \bc^\lambda\,.$$
\end{proposition}
\begin{proof}
Observe that 
\begin{align*}
\pi_A\circ\pi_{/\lambda}\left(\res_{\fp}(\bc^{\lambda})\right)=\pi_{/\lambda}\circ\pi_A\left(\res_{\fp}(\bc^{\lambda})\right)=\pi_{/\lambda}\left(\res_{\fp}(c^{\lambda}_A)\right)=0\,.
\end{align*}
This shows that $\pi_{/\lambda}\left(\res_{\fp}(\bc^{\lambda})\right)\in \ker(\pi_A)=(\gamma_\cyc-1)H^1_\an(K_{\fp}, \overline{D}_{\textup{rig}}^\dagger(V_A)/\overline{D}_{\lambda,A})$, proving the existence of $\partial_\fp \bc^\lambda$ with the desired property. Its uniqueness follows from the fact that $H^1_\an(K_{\fp}, \overline{D}_{\textup{rig}}^\dagger(V_A)/\overline{D}_{\lambda,A})[\gamma_\cyc-1]=0$ under our running hypotheses.
\end{proof}

We define $\partial_\fp c^\lambda_A:=\pi_A\left(\partial_\fp \bc^\lambda\right) \in H^1_\an(K_{\fp}, {D}_{\textup{rig}}^\dagger(V_A)/{D}_{\lambda,A})$.
\begin{theorem}[Rubin-style formula]
\label{thm:analyticRubinstyleformula}
$\langle c^\lambda_A,c^\lambda_A\rangle_{\lambda,\lambda}=-\left\langle \partial_\fp c_{A}^\lambda,\res_\fp(c_A^\lambda)\right\rangle_{\lambda,\lambda}^{(\fp)}$\,.
\end{theorem}
\begin{proof}
The proof of \cite[Corollary 4.16]{benoisbuyukboduk} applies verbatim\footnote{Even though Corollary 4.16 of  \cite{benoisbuyukboduk} is stated for Deligne's $G_\QQ$-representation attached to $f$ with coefficients in a finite extension of $\QQ_p$, its proof is entirely formal (and ultimately, it follows very closely the proof of \cite[Corollary A.11]{kbbMTT} in the ordinary setting) and applies with the trianguline $G_K$-representation $V_A$ with coefficients in $A$. }, with the following choices: 
\begin{itemize}
\item The pairing denoted by $\frak{H}(\,,\,)$ in op. cit. is to be taken as $\langle\,,\,\rangle_{\lambda,\lambda}\cdot{\dfrac{\gamma_\cyc-1}{\log_p\chi_0(\gamma_\cyc)}} \mod (\gamma_\cyc-1)^2\,;$
\item $[x_{\rm f}]=[y_{\rm f}]=c_A^\lambda$; $[\mathbb{X}]=\overline{c}^\lambda$ and $\frak{d}[\mathbb{X}]=\partial_{\frak{p}}c_A^\lambda$ (the reason why only the prime $\frak{p}$ appears here is  explained below);
\item the pairing denoted by $\langle\,,\,\rangle: H^1(\widetilde{\mathbb{D}}_f)\times  H^1({\mathbb{D}}_f)\ra E$ (in the notation of op. cit.) is to be taken as the $A$-valued $\frak{p}$-local cup product pairing $\left\langle\,,\,\right\rangle_{\lambda,\lambda}^{(\fp)}$.
\end{itemize}
Notice that the local contribution is concentrated at the prime $\frak{p}$ for the following reasons: At primes $\frak{l}\nmid p$, the  specializations of $\res_\frak{l}(c^\lambda_A)$ on a dense subset of points of $\textup{Sp} A$ are trivial by the local-global compatibility of the Langlands correspondence (c.f. \cite{Car86}) and therefore the local contributions at primes $\frak{l}\nmid p$ vanish. For the local contribution at $\frak{p}^c$, we have $\pi_{/\lambda}\circ\res_{\frak{p^c}}(\overline{c}^\lambda)=0$ by Proposition~\ref{prop:sameprojection} and therefore its Bockstein normalized derivative $\partial_{\fp^c} c_{A}^\lambda$ at the prime $\frak{p}^c$ vanishes.
\end{proof}
\begin{defn}
\label{def:algebraicderivedpadicLfunction}
We define the \textbf{derived $p$-adic $L$-function} $\mathfrak{L}^\prime_{\lambda,\lambda,\ac} \in A$ by setting $\mathfrak{L}^\prime_{\lambda,\lambda,\ac}:=\mathcal{L}_{\lambda,A}\left(\partial_\fp c^\lambda_A\right)$,
where $\mathcal{L}_{\lambda,A}:=\pi_A\circ\mathcal{L}_{\lambda,\fp}$ is the restriction of the two-variable Perrin-Riou's map $($for $K_\fp$ and corresponding to the $p$-stabilization with respect to $\lambda$$)$ to the anticyclotomic direction.
\end{defn}
Corollary~\ref{cor:derivativeofLalphaalphaisLalphaalphaprime} below justifies our terminology here.
\begin{remark}
One may directly define a $p$-adic Perrin-Riou map for $V_A$ $($as in \cite{Nakamuraepsilon, nakamura2014Jussieu}; see also \cite[\S5.1]{benoisbuyukboduk}$)$. This map is related to $\pi_A\circ \mathcal{L}_{\lambda,\fp}$ via  \cite[Theorem~$5.1(\textup{iii})$]{benoisbuyukboduk}.
\end{remark}
Recall the Beilinson-Flach $p$-adic $L$-function
$$\mathfrak{L}_{\lambda,\lambda}:=\mathcal{L}_{\lambda,\fp}\left(\res_\fp\left(c^\lambda\right)\right) \in \mathcal{H}_A(\Gamma^\cyc):=A\widehat{\otimes}\calHcyc.$$
\begin{corollary}
\label{cor:derivativeofLalphaalphaisLalphaalphaprime} $\mathfrak{L}_{\lambda,\lambda}=\displaystyle{\frac{\gamma_\cyc-1}{\log_p\chi_0(\gamma_\cyc)}}\cdot\mathfrak{L}_{\lambda,\lambda,\ac}^\prime \mod (\gamma_\cyc-1)^2\,.$
\end{corollary}

\subsection{The $p$-adic BSD formula}
\label{subsec:AadicBSD}
{Recall that the hypotheses of Proposition~\ref{prop:lfourcorrect}  are in effect. Then the following assertions are valid:}
\begin{itemize}
\item {The Beilinson-Flach class $c^\lambda_\ac$ is a non-trivial element of $H^1_{\lambda,\lambda}(K,\TT^\dagger)$: Indeed, as explained in the proof of Corollary~\ref{cor:whenBFelementsarenonzeroAnalytic1}, the class $c^\lambda_\ac$ is non-trivial and \eqref{eqn_BF_ac_lcoal_position_indefinite} shows that it is an element of $H^1_{\lambda,\lambda}(K,\TT^\dagger)$.}
\item The $\cH(\Gamma^\ac)$-module $\mathfrak{X}_{\lambda,\lambda}^{\rm an}$ has rank one {(see Corollary~\ref{cor:comparisonofBKselmerwithrelaxedBKintheindefinitecase})}.
\end{itemize}
The main objective in this subsection is to relate the derived $p$-adic $L$-function to the $p$-adic regulator $\mathscr{R}_p$ and the torsion submodule of $\mathfrak{X}_{\lambda,\lambda}^A$, much in the spirit of the Birch and Swinnerton-Dyer conjecture.
\begin{theorem}
\label{thm:padicGZ}
The element  $\mathfrak{L}^\prime_{\lambda,\lambda,\ac}$ is contained in $\textup{Fitt}_{A}\left(\mathscr{D}^1\mathfrak{X}_{\lambda,\lambda}^A\right) \mathscr{R}_p$. Furthermore, these two ideals are the same if and only if we have equality in Theorem~\ref{thm:integralBFundersignedlog}.
\end{theorem}
\begin{remark}
Recall that the characteristic ideal of the torsion module $\mathscr{D}^1\mathfrak{X}_{\lambda,\lambda}^A$ agrees with the characteristic ideal of the $A$-torsion-submodule of $\mathfrak{X}_{\lambda,\lambda}^A$\,. If we are able to upgrade the containment in Theorem~\ref{thm:padicGZ} to the equality
$$A\cdot\mathfrak{L}^\prime_{\lambda,\lambda,\ac}=\textup{Fitt}_{A}\left(\mathscr{D}^1\mathfrak{X}_{\lambda,\lambda}^A\right)\cdot \mathscr{R}_p\,,$$
our result would indeed gain the flavour of a $p$-adic Birch and Swinnerton-Dyer formula. 
\end{remark}
\begin{proof}[Proof of Theorem~\ref{thm:padicGZ}]
All Fitting ideals below are calculated as those of $A$-modules. For $\mathscr{J}_\fp$ given as in the statement of Theorem~\ref{thm:Hzeroalphavanishesandcomparinftorsion} above, we have
\begin{align}
\notag\mathscr{R}_p\cdot\textup{Fitt}\left(\mathscr{D}^1\mathfrak{X}^A_{\lambda,\lambda}\right)\cdot\mathscr{J}_\fp&=\mathscr{R}_p\cdot\textup{Fitt}\left(\mathfrak{X}^A_{\emptyset,\lambda}\right)\\
\label{eqn:alphamainconjecturecor:analyticmainconj0alpha}&\supset\mathscr{R}_p\cdot\textup{Fitt}\left(H^1_{\emptyset,\lambda}(K,V_A)\big{/}A\cdot c^\lambda_A\right)\\
\label{eqn:h1alphaalphaagreesh1emptyalpha} &=\mathscr{R}_p\cdot\textup{Fitt}\left(H^1_{\lambda,\lambda}(K,V_A)\big{/}A\cdot c^\lambda_A\right)\\
\notag &=\mathscr{R}_p(c^\lambda_A)\\
\label{eqn:rubinstyleformula} &= \textup{Fitt}\left(A\Big{/} \left\langle\partial_\fp c^\lambda_A\,,\,\res_{\fp}\left(H^1_{\lambda,\lambda}(K,V_A)\right)\right\rangle_{\lambda,\lambda}^{(\fp)}\right)\\
\notag &=\textup{Fitt}\left(\frac{\han^1(K_{\fp},D_{\lambda,A})}{\res_\fp({H^1_{\lambda,\lambda}(K,V_A))}}\right)\cdot \textup{Fitt}\left(A/\mathscr{T}_\fp(\partial_\fp c^\lambda_A)\right)\\
\label{eqn:defnofJp} &=\mathscr{J}_\fp\cdot \textup{Fitt}\left(A/\mathscr{T}_\fp(\partial_\fp c^\lambda_A)\right)\\
\label{eqn:surjoftatepairing}&=\mathscr{J}_\fp\cdot \textup{Fitt}\left(\han^1(K_\fp,D_{\textup{rig}}^\dagger(V_A)/D_{\lambda,A})\big{/}A\cdot\partial_\fp c^\lambda_A\right)\\
\notag&=\mathscr{J}_\fp \cdot\mathfrak{L}^\prime_{\lambda,\lambda,\ac}.
\end{align}
Here, the first equality follows from the second and third parts of Theorem~\ref{thm:Hzeroalphavanishesandcomparinftorsion}, the inclusion (\ref{eqn:alphamainconjecturecor:analyticmainconj0alpha}) from Corollary~\ref{cor:analyticmainconj0alpha} (with equality under the said conditions), equality (\ref{eqn:h1alphaalphaagreesh1emptyalpha}) from Corollary~\ref{cor:comparisonofBKselmerwithrelaxedBKintheindefinitecase}, (\ref{eqn:rubinstyleformula}) from Theorem~\ref{thm:analyticRubinstyleformula}, (\ref{eqn:defnofJp}) from the definition of $\mathscr{J}_\fp$ {as the characteristic ideal of ${\rm coker}\left(H^1_{\lambda,\lambda}(K,\TT^\dagger)\stackrel{\res_\fp}{\hookrightarrow} \han^1(K_{\fp},F_\lambda)\right)$},  (\ref{eqn:surjoftatepairing}) from the surjectivity of the local cup product pairing and the very last equality follows from the definition of the derived $p$-adic $L$-function and Corollary~\ref{cor:imageLlambda}. {Since both $H^1_{\lambda,\lambda}(K,\TT^\dagger)$ and $\han^1(K_{\fp},F_\lambda)$ have rank one (where the fact that $H^1_{\lambda,\lambda}(K,\TT^\dagger)$ has rank one follows on combining Corollary~\ref{cor:whenBFelementsarenonzeroAnalytic1} and Corollary~\ref{cor:comparisonofBKselmerwithrelaxedBKintheindefinitecase})}, we  see  that $\mathscr{J}_\fp\neq 0$ from its very definition. The proof follows.
\end{proof}

\bibliographystyle{amsalpha}
\bibliography{references}
\end{document}